\documentclass[9pt, oneside]{amsart}   	

\usepackage{geometry}                	
\usepackage{graphicx}

\usepackage{amssymb}
\usepackage{amsmath}
\usepackage{amsfonts}
\usepackage{amsthm}
\usepackage{color}
\usepackage{amscd}
\usepackage[usenames,dvipsnames,svgnames,table]{xcolor}
\usepackage[utf8]{inputenc}
\usepackage[parfill]{parskip}
\usepackage[colorlinks=true, pdfstartview=FitV,linkcolor=ForestGreen,citecolor=ForestGreen, urlcolor=black]{hyperref}
\usepackage{enumerate}
\usepackage{esint}
\usepackage{cite}
\usepackage{bbm}
\usepackage{comment}
\usepackage{url} 
\usepackage{tikz}
\usetikzlibrary{patterns}
\usetikzlibrary{backgrounds}
\usepackage{mathrsfs}

\newcommand{\dd}{\, {\rm d}}

\newcommand{\abs}[1]{\left\vert#1\right\vert}
\newcommand{\Abs}[1]{\left\vert#1\right\vert}
\newcommand{\norm}[1]{\left\Vert#1\right\Vert}  
\newcommand{\Norm}[1]{\left\Vert#1\right\Vert}  

\newcommand{\T}{\ensuremath{{\mathbb T}}}
\newcommand{\R}{\ensuremath{{\mathbb R}}}

\newcommand{\beq}{\begin{equation}}
\newcommand{\eeq}{\end{equation}}
\newcommand{\beqs}{\begin{equation*}}
\newcommand{\eeqs}{\end{equation*}}
\newcommand{\bal}{\begin{equation}\begin{aligned}}
\newcommand{\eal}{\end{aligned}\end{equation}}
\newcommand{\bals}{\begin{equation*}\begin{aligned}}
\newcommand{\eals}{\end{aligned}\end{equation*}}

 \numberwithin{num}{section}

\theoremstyle{thmstyleone}%
\newtheorem{theorem}{Theorem}
\newtheorem{proposition}[theorem]{Proposition}%
\newtheorem{lemma}[theorem]{Lemma}

\theoremstyle{remark}
\newtheorem{openprob}[theorem]{Open Problems}

\theoremstyle{thmstyletwo}%
\newtheorem{remark}{Remark}%

\theoremstyle{thmstylethree}%
\newtheorem{definition}{Definition}%

\begin{document}

\title[Semi-local behaviour of non-local hypoelliptic equations]{Semi-local behaviour of non-local hypoelliptic equations: divergence form}
\author{Amélie Loher}
\email{amelie.loher@all-souls.ox.ac.uk}
\date{16th September 2025}

\address{Department of Pure Mathematics and Mathematical Statistics, University of Cambridge,
Wilberforce Rd, Cambridge, CB3 0WA, UK}

\begin{abstract}
We derive the Strong Harnack inequality for a class of hypoelliptic integro-differential equations in divergence form.
Our result is semi-local, in the sense that we require the equation to hold globally in velocity. In fact, it is known that the Strong Harnack inequality fails in the kinetic case if the equation is merely satisfied in a bounded velocity domain \cite{KW-failedH}. This is in stark contrast to the case of parabolic equations, for which a local Strong Harnack inequality holds. 

In a first step, we derive a local bound on the non-local tail on upper level sets by exploiting the coercivity of the cross terms. In a second step, we perform a De Giorgi argument in $L^1$, since we control the tail term only in $L^1$. This yields a linear $L^1$ to $L^\infty$ bound.
Consequently, we prove polynomial upper bounds for \(s\in(0,1)\)
and exponential lower bounds for \(s\in(\frac12,1)\) on the
fundamental solution by adapting Aronson's method to non-local hypoelliptic equations.
\end{abstract}


\maketitle
\tableofcontents

\section{Introduction}\label{sec:intro}

\subsection{Problem formulation}
We consider kinetic integro-differential equations of the form
\beq
	\partial_t f + v \cdot \nabla_x f = \mathcal L f + h, \quad t \in \R, ~ x \in \R^d, ~ v \in \R^d,
\label{eq:1.1}
\eeq
where we assume $h = h(t, x, v)$ is a real measurable scalar field and where $\mathcal L$ is defined by 
\beq
	\mathcal L f(t, x, v) := \textrm{PV}\int_{\R^d}  \big[f(t, x, w) - f(t, x, v)\big] K(t, x, v, w)\dd w,
\label{eq:1.2}
\eeq
with a non-negative measurable kernel $K: \R \times \R^d \times \R^d \times \R^d \to [0, +\infty)$. Here, $\textrm{PV}$ denotes the Cauchy principal value, that is
\bals
	\textrm{PV}&\int_{\R^d} K(t, x, v, w) \big[f(t, x, w) - f(t, x, v)\big] \dd w \\
	&= \lim_{\varepsilon\to0}\int_{\R^d\setminus B_\varepsilon(v)} K(t, x, v, w) \big[f(t, x, w) - f(t, x, v)\big] \dd w.
\eals
Given a solution to \eqref{eq:1.1} in some bounded subset of $\R^{1+d}$ in time and space, and on $\R^d$ in velocity, we derive a linear bound on the local supremum of this solution by its local infimum in the future.
Such a relation is known as the Strong Harnack inequality for solutions of \eqref{eq:1.1}. As a consequence, we deduce polynomial upper and exponential lower bounds (for $s \geq \frac{1}{2}$) on the fundamental solution of \eqref{eq:1.1} away from the initial time.

Equation \eqref{eq:1.1} is an integro-differential equation. The non-locality measures fluctuations in velocity for a local time and space. In particular, the non-local operator takes into account values of the function outside any local velocity domain. The regularity of solutions to \eqref{eq:1.1} has first been investigated by Imbert-Silvestre \cite{IS, ISglobal}. In particular, it is known that super-solutions of \eqref{eq:1.1} satisfy the Weak Harnack inequality, \cite[Theorem 1.6]{IS} and \cite[Theorem 1.1]{AL}, that is some local average of a power of the solution in a past cylinder is bounded by the local infimum in a future cylinder. This is even true for solutions in bounded velocity domains. The proof of the Weak Harnack inequality uses a gain of integrability in form of a bound on the supremum by the $L^2$ norm of the solution (an $L^2$-$L^\infty$ bound) \cite[Lemma 6.6]{IS}, \cite[Lemma 4.1]{AL}.
Combining the Weak Harnack with the $L^2$-$L^\infty$ bound suggests that solutions of \eqref{eq:1.1} satisfy the Strong Harnack inequality. Intuitively, however, it is not clear that we expect the supremum of a solution to \eqref{eq:1.1} to be bounded by a local quantity, since the values in the whole velocity domain affect the local supremum. In fact, the Strong Harnack inequality fails for kinetic equations satisfied in bounded velocity domains \cite{KW-failedH}. 
So far only a \textit{nonlinear} relation between the supremum and the $L^2$ norm has been established for \eqref{eq:1.1} in \cite[Lemma 4.1]{AL}, which can be made linear at the cost of \textit{non-local tail terms} that precisely take into account function values of the solution outside the domain in which \eqref{eq:1.1} holds. We refer the reader to \cite[Remark 4.2]{AL}. As a consequence of this nonlinear $L^2$-$L^\infty$ bound \cite[Lemma 4.1]{AL} and the Weak Harnack inequality \cite[Theorem 1.1]{AL}, we deduced a nonlinear upper bound on the local supremum, which we called the “Not-so-Strong Harnack inequality” \cite[Theorem 1.3]{AL}. Here, however, we aim to establish a \textit{linear} $L^2$-$L^\infty$ bound \textit{without tail terms}, so that we improve the “Not-so-Strong Harnack inequality” \cite[Theorem 1.3]{AL} to a Strong Harnack inequality. We show that if \eqref{eq:1.1} is satisfied for any $v \in \R^d$, then the Strong Harnack inequality holds. Such a statement is of independent interest since it captures the local regularity of solutions to a class of hypoelliptic integro-differential equations, and moreover, it bears interesting consequences, such as quantitative bounds on the fundamental solution.

\subsection{Main results}
In this paper, we capture the local behaviour of solutions to non-local hypoelliptic equations.
We derive the Strong Harnack inequality for solutions of \eqref{eq:1.1}, and consequentially, we deduce polynomial upper bounds and exponential lower bounds on the fundamental solution of \eqref{eq:1.1}. 
The kernels of the non-local operator \eqref{eq:1.2} that we are interested in would ideally correspond to the setting that was introduced in \cite{IS59} and used in \cite{IS, ISglobal, AL}: there the ellipticity class of the kernel was prescribed through physical considerations of gas dynamics modelled by the Boltzmann equation without cutoff assumptions.
In particular, the notion of ellipticity, introduced in \cite{IS59} and reused here, generalises the commonly used pointwise bounds on the kernel of the non-local operator. The idea is to relate the non-local operator \eqref{eq:1.2} to the fractional Laplacian in an average sense, see Subsection \ref{subsec:ellipticity-class} below. In short, we assume a coercivity condition in integral form \eqref{eq:coercivity} and an integral upper bound \eqref{eq:upperbound}. However, in contrast to the setting considered by \cite{ISglobal} or \cite{AL}, we make a pointwise symmetry assumption stated in \eqref{eq:symmetry}, which is the non-local analogue of local operators in divergence form. 
Solutions to \eqref{eq:1.1} of this class satisfy the Strong Harnack inequality, \textit{provided that the equation is satisfied for any $v\in \R^d$.}
\begin{theorem}[Strong Harnack inequality]\label{thm:strong-Harnack}
Let $f \in L^2([-3, 0]\times B_1; H^s(\R^d))$ be a non-negative solution of \eqref{eq:1.1}-\eqref{eq:1.2} in $[-3, 0]\times B_1 \times \R^d \subset \R^{1+2d}$ in the sense of Definition \ref{def:solutions} (which requires the solution to decay pointwisely in velocity almost everywhere in $[-3, 0]\times B_1 \times \R^d$),
with a non-negative kernel $K$ of order $2s \in (0, 2)$ that satisfies for given $\lambda_0, \Lambda_0 > 0,$ a coercivity condition on average \eqref{eq:coercivity}, an upper bound on average \eqref{eq:upperbound}, a pointwise divergence form symmetry \eqref{eq:symmetry} and, in case that $s \geq \frac{1}{2}$, a non-divergence form symmetry \eqref{eq:symmetry-nondiv}, and with a non-negative source term $0 \leq h \in L^\infty([-3, 0] \times B_1\times \R^d)$. 
Then there is $C> 0$ depending on $s, d, \lambda_0, \Lambda_0$ such that for any $0 < r_0 <\frac{1}{6}$ the Strong Harnack inequality is satisfied:
\beq\label{eq:strong-Harnack}
	\sup_{\tilde Q^-_{\frac{r_0}{4}}} f \leq C \inf_{Q_{\frac{r_0}{4}}} f + C\norm{h}_{L^\infty([-3, 0] \times Q_1^t)},
\eeq
where $\tilde Q^-_{\frac{r_0}{4}} := [-t_3, -t_2] \times B_{\left(\frac{r_0}{4}\right)^{1+2s}} \times B_{\frac{r_0}{4}}$ with $t_2 = \frac{5}{2}r_0^{2s} - \frac{1}{2}\left(\frac{r_0}{2}\right)^{2s} \leq t_3 = t_2 +\left(\frac{r_0}{4}\right)^{2s}$, and $Q_{\frac{r_0}{4}} := [-t_{1}, 0]\times B_{\left(\frac{r_0}{4}\right)^{1+2s}} \times B_{\frac{r_0}{4}}$ with $t_{1} = \left(\frac{r_0}{4}\right)^{2s}$, see Figure \ref{fig:harnacks}. In particular, any solution $f$ of \eqref{eq:1.1} with zero source term $h = 0$ satisfies 
\beqs
	\sup_{\tilde Q^-_{\frac{r_0}{4}}} f \leq C \inf_{Q_{\frac{r_0}{4}}} f.
\eeqs
\end{theorem}
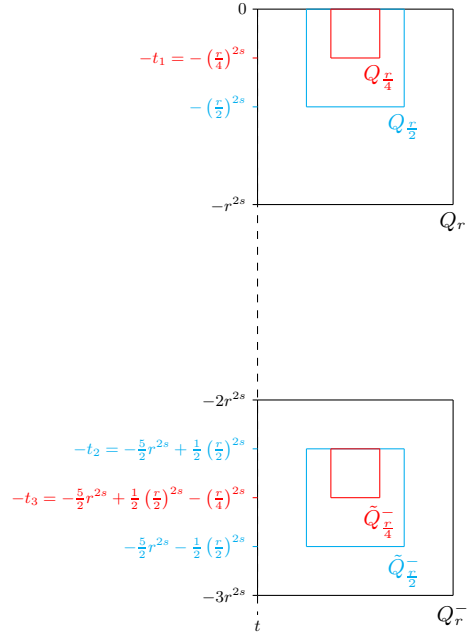
\begin{figure}
\begin{tikzpicture}[scale =1.3]
   \draw [black](-1,0) -- (1, 0);
   \draw [black](-1,-2) -- (1, -2);
   \draw [black](-1,0) -- (-1, -2);
   \draw [black](1,0) -- (1, -2)  node[anchor=north, scale=0.9] {$Q_{r}$};
    \draw[cyan] (-0.5,0) -- (0.5, 0);
   \draw [cyan](-0.5,-1) -- (0.5, -1);
   \draw [cyan](-0.5,0) -- (-0.5, -1);
   \draw [cyan](0.5,0) -- (0.5, -1)  node[anchor=north, scale=0.9] {$Q_{\frac{r}{2}}$};
    \draw[red] (-0.25,0) -- (0.25, 0);
   \draw [red](-0.25,-0.5) -- (0.25, -0.5);
   \draw [red](-0.25,0) -- (-0.25, -0.5);
   \draw [red](0.25,0) -- (0.25, -0.5)  node[anchor=north, scale=0.9] {$Q_{\frac{r}{4}}$};
    \draw[black] (-1,-4) -- (1, -4);
   \draw [black](-1,-6) -- (1, -6);
   \draw [black](-1,-4) -- (-1, -6);
   \draw [black](1,-4) -- (1, -6)  node[anchor=north, scale=0.9] {$Q^-_{r}$};
    \draw[cyan] (-0.5,-4.5) -- (0.5, -4.5);
   \draw [cyan](-0.5,-5.5) -- (0.5, -5.5);
   \draw [cyan](-0.5,-4.5) -- (-0.5, -5.5);
   \draw [cyan](0.5,-4.5) -- (0.5, -5.5)  node[anchor=north, scale=0.9] {$\tilde Q^-_{\frac{r}{2}}$};
    \draw[red] (-0.25,-4.5) -- (0.25, -4.5);
   \draw [red](-0.25,-5) -- (0.25, -5);
   \draw [red](-0.25,-4.5) -- (-0.25, -5);
   \draw [red](0.25,-4.5) -- (0.25, -5)  node[anchor=north, scale=0.9] {$\tilde Q^-_{\frac{r}{4}}$};
   \draw[dashed] (-1, 0) -- (-1, -6.2) node[anchor = north, scale = 0.7] {$t$};
   \draw (-1, 0) -- (-1.05, 0) node[anchor= east, scale=0.7] {$0$};
   \draw[red] (-1, -0.5) -- (-1.05, -0.5) node[anchor= east, scale=0.7] {$-t_1 = -\left(\frac{r}{4}\right)^{2s}$};
   \draw[cyan] (-1, -1) -- (-1.05, -1) node[anchor= east, scale=0.7] {$-\left(\frac{r}{2}\right)^{2s}$};
   \draw (-1, -2) -- (-1.05, -2) node[anchor= east, scale=0.7] {$-r^{2s}$};
   \draw (-1, -4) -- (-1.05, -4) node[anchor= east, scale=0.7] {$-2r^{2s}$};
   \draw[cyan] (-1, -4.5) -- (-1.05, -4.5) node[anchor= east, scale=0.7] {$-t_2 = -\frac{5}{2}r^{2s} + \frac{1}{2}\left(\frac{r}{2}\right)^{2s}$};
   \draw[red] (-1, -5) -- (-1.05, -5) node[anchor= east, scale=0.7] {$-t_3 = -\frac{5}{2}r^{2s} + \frac{1}{2}\left(\frac{r}{2}\right)^{2s} - \left(\frac{r}{4}\right)^{2s}$};
   \draw[cyan] (-1, -5.5) -- (-1.05, -5.5) node[anchor= east, scale=0.7] {$-\frac{5}{2}r^{2s} - \frac{1}{2}\left(\frac{r}{2}\right)^{2s}$};
   \draw (-1, -6) -- (-1.05, -6) node[anchor= east, scale=0.7] {$-3r^{2s}$};
 \end{tikzpicture}
\caption{The kinetic cylinder $Q_r$ and its past cylinder $Q_r^-$ for some $r > 0$.  The dashed vertical line represents the timeline. The blue cylinder $Q_{\frac{r}{2}}$ and its corresponding past $\tilde Q^-_{\frac{r}{2}}$ represent the domains appearing in the Weak Harnack inequality, Theorem \ref{thm:weakH}. The red cylinder $Q_{\frac{r}{4}}$ and its corresponding past $\tilde Q^-_{\frac{r}{4}}$ are the domains appearing in the Strong Harnack inequality, Theorem \ref{thm:strong-Harnack}.}\label{fig:harnacks}
\end{figure}
Figure \ref{fig:harnacks} illustrates a non-local to local bridge: we recover the same domains for the Harnack inequalities as in the local case \cite[Theorem 5]{JGCM}.

As a consequence of Theorem \ref{thm:strong-Harnack}, we derive polynomial upper and, in case that $s \in (\frac12, 1)$,  exponential lower bounds on the fundamental solution of \eqref{eq:1.1} with a non-negative kernel $K$ satisfying \eqref{eq:coercivity}-\eqref{eq:symmetry-nondiv}, \textit{provided that the fundamental solution exists}. To give sense to the next three theorems, we thus assume existence of a non-negative measurable function $J$, which is the fundamental solution of \eqref{eq:1.1} connecting a given point $(t, x, v) \in \R^{1+2d}$ with $(\tau, y, w) \in \R^{1+2d}$, in the sequel denoted by
\beqs
	J(t, x, v; \tau, y, w) = J(t- \tau, x - y - (t-\tau)w, v-w) =: J\big((\tau, y, w)^{-1} \circ (t, x, v)\big),
\eeqs
where $\circ$ denotes the Galilean translation, that is $(t_0, x_0, v_0) \circ (t, x, v) = (t_0 + t, x_0+x + tv_0, v_0 + v)$, which respects the translation invariance of \eqref{eq:1.1}.
Moreover, we assume that the fundamental solution $J$ has the following properties:
\begin{enumerate}[i.]

\item For every $t \in \R_+$ there holds the normalisation
\beq\label{eq:normalisation-J}
	\int_{\R^{2d}} J(t, x, v)  \dd x \dd v = 1.
\eeq

\item There holds $J \geq 0$, and for all $(t, x, v), (\tau, y, w) \in \R_+ \times \R^{2d}$ a form of symmetry
\beq\label{eq:symmetry-J}
J\left((\tau, y, w)^{-1} \circ (t, x, v)\right) = J\left((\tau, x, v)^{-1} \circ (t, y, w)\right).
\eeq

\item For any $0 \leq \tau < \sigma < t < T$ and any $(x, v), (y, w) \in \R^{2d}$ the Chapman-Kolmogorov identity holds
\bal\label{eq:representation-J}
	J(t, x, v; \tau, y, w) &= J\left((\tau, y, w)^{-1} \circ (t, x, v)\right)\\
	&= \int_{\R^{2d}} J\left((\sigma, \varphi, \xi)^{-1} \circ (t, x, v)\right)J\left((\tau, y, w)^{-1} \circ(\sigma, \varphi, \xi)\right) \dd \varphi \dd \xi \\
	&= \int_{\R^{2d}} J(t, x, v; \sigma, \varphi, \xi)J(\sigma, \varphi, \xi; \tau, y, w) \dd \varphi \dd \xi.
\eal
\end{enumerate}
Then, we deduce, on the one hand, polynomial upper bounds.
\begin{theorem}[Polynomial upper bounds on the fundamental solution]\label{thm:upperbounds}
Let $x, v, y_0, w_0 \in \R^d$, and $0 \leq \tau_0 < \sigma < T$. Let $J$ be the fundamental solution of \eqref{eq:1.1} in $[0, T] \times \R^{2d}$ with a vanishing source term $h = 0$, with a pointwise polynomial velocity decay, and with a non-negative kernel $K$ satisfying the coercivity assumption \eqref{eq:coercivity}, the integral upper bound \eqref{eq:upperbound}, the divergence form symmetry \eqref{eq:symmetry} and, in case that $s \geq \frac{1}{2}$, the non-divergence form symmetry \eqref{eq:symmetry-nondiv}. Assume $J$ satisfies \eqref{eq:normalisation-J}, \eqref{eq:symmetry-J} and \eqref{eq:representation-J}. Then there exists $C > 0$ depending on $s, d, \lambda_0, \Lambda_0$ such that
\bal\label{eq:J-up}
	J&(\sigma,x, v; \tau_0, y_0, w_0) \\
	&\leq C (\sigma - \tau_0)^{-\frac{2d(1+s)}{2s}}\\
	&\qquad \times \left[ 1 +\frac{\max\left\{\abs{v- w_0}^{2s},  \abs{x - y_0 - (\sigma - \tau_0) (v-w_0)}^{\frac{2s}{1+2s}}\right\}}{\sigma - \tau_0}\right]^{-\frac{s}{4s}}.
\eal
\end{theorem}
On the other hand, we derive an exponential lower bound.
\begin{theorem}[Exponential lower bounds on the fundamental solution]\label{thm:lowerbounds}
Let \(s\in(\frac12,1)\).
Let $x, v, y_0, w_0 \in \R^d$ and $0 < \tau_0 < \sigma < T$.
Let $J$ be the fundamental solution of \eqref{eq:1.1} in $[0, T] \times \R^{2d}$ with a vanishing source term $h = 0$, with a pointwise polynomial velocity decay and with a kernel satisfying the coercivity assumption \eqref{eq:coercivity}, the integral upper bound \eqref{eq:upperbound}, the divergence form symmetry \eqref{eq:symmetry} and the non-divergence form symmetry \eqref{eq:symmetry-nondiv}. Then there exists a constant $C$ depending on $s, d, \lambda_0, \Lambda_0$ such that
\begin{align*}
J(\sigma,x,v;\tau_0,y_0,w_0)
\geq{}&
c(\sigma-\tau_0)^{-\frac{d(1+s)}s}\\
&\times
\exp\left[
-C\left(
\frac{|v-w_0|}{(\sigma-\tau_0)^{1/(2s)}}
+
\frac{|x-y_0-(\sigma-\tau_0)w_0|}
     {(\sigma-\tau_0)^{1+1/(2s)}}
\right)^{\frac{2s}{2s-1}}
\right].
\end{align*}
\end{theorem}
\begin{remark}\label{rmk:cc-fund-sol}
\begin{enumerate}[i.]
\item 
Note that in general, the fundamental solution has no explicit expression, not even in case of constant coefficients. The constant coefficient equation for \eqref{eq:1.1}, known as the \textit{fractional Kolmogorov equation}, is given by
\begin{equation}\label{eq:frac-kolm}
	\partial_t f + v \cdot \nabla_x f = (-\Delta_v)^s f + h,
\end{equation}
with source term $h \in L^2(\R_+ \times \R^{2d})$, and where $(-\Delta_v)^s$ denotes the fractional Laplacian, defined as
\begin{equation}\label{eq:frac-lapl}
	(-\Delta_v)^s f(v) = C(d, s)\int_{\R^d} \frac{f(v) - f(w)}{\abs{v-w}^{d+2s}} \dd w. 
\end{equation}
Equation \eqref{eq:frac-kolm} admits a fundamental solution, given through an implicit representation,
\beq\label{eq:J0}
	J_0(t, x, v) = C(d, s)t^{-\frac{2d(1+s)}{2s}} \mathcal J_0\left(\frac{x}{t^{1+\frac{1}{2s}}}, \frac{v}{t^{\frac{1}{2s}}}\right),
\eeq
where $\mathcal J_0$ reads explicitly in Fourier variables
\beqs
	\hat {\mathcal J}_0(\varphi, \xi) = \exp\left(- \int_0^1 \abs{\xi - \tau \phi}^{2s} \dd \tau\right). 
\eeqs
In particular, the fundamental solution $J_0$ satisfies \eqref{eq:normalisation-J}, \eqref{eq:symmetry-J}, \eqref{eq:representation-J}, and moreover, $0 \leq \mathcal J_0 \in C^\infty$ decays polynomially at infinity. 
\item In case of rough coefficients, the existence of a fundamental solution operator has only recently been established by Auscher-Imbert-Niebel \cite{auscher-imbert-niebel-1} for local and non-local hypoelliptic equations. In particular, their results are applicable to \eqref{eq:1.1} with kernels satisfying \eqref{eq:coercivity}-\eqref{eq:symmetry}.
\item In Theorem \ref{thm:upperbounds} and Theorem \ref{thm:lowerbounds}, the time decay is optimal, as can be seen by comparing the exponents on time to \eqref{eq:J0}. This is a consequence of the optimality of the linear $L^1$-$L^\infty$ bound of Proposition \ref{prop:L2-Linfty-A} below.
However, neither the decay in space nor in velocity is optimal. 

Concerning the exponential lower bound of Theorem \ref{thm:lowerbounds}, it is not clear to us if we should expect polynomial lower bounds to hold without assuming a pointwise lower bound on the kernel of the non-local operator \eqref{eq:1.2}.

The polynomial upper bounds in Theorem \ref{thm:upperbounds} are sub-optimal in two ways. First, a minor improvement could possibly be achieved, if, in Aronson's method (see Section \ref{sec:aronson} below), when we define the decay function $H$ in \eqref{eq:H}, we replace $\max\{\abs{v-w_0}, \abs{x-y_0-tw_0}^{\frac{1}{1+2s}}\}$ by the sum $\abs{v-w_0} + \frac{\abs{x-y_0-tw_0}}{t}$. It is conceivable that this leads to an improved decay in space to $\frac{s}{2}$ instead of $\frac{s}{2(1+2s)}$. 
Second, the most significant sub-optimality arises due to the way we treat the non-singular contribution of the fundamental solution in Aronson's method (see Proposition \ref{prop:aronson}). To determine the decay of the non-singular part, we only use the integral upper bounds on the kernel \eqref{eq:upperbound}. In the parabolic literature, the non-singular part is usually treated separately from the singular terms, using a technique referred to as Meyer's decomposition \cite{meyer}. Conditional to such a decomposition, see \eqref{eq:meyers}, we obtain the optimal decay at least in velocity, see Theorem \ref{thm:upperbounds-conditional} below. 
Finally note that in the parabolic setting with kernels that are pointwisely bounded from above, the proof of Theorem \ref{thm:upperbounds} - which does not rely on Meyer's decomposition - yields a decay that is not so far from optimal, see Appendix \ref{app:improved-decay}.
\end{enumerate}
\end{remark}
For the statement of the conditional upper bounds, we denote by $J_\rho$ the fundamental solution to the cutoff problem
\beqs
	\partial_t f + v \cdot \nabla_x f = \mathcal L_\rho f,
\eeqs
where $\mathcal L_\rho f$ is defined for any $\rho > 0$ as
\beqs
	\mathcal L_\rho f(t, x, v) := PV \int_{B_\rho(v)} (f(t, x, w) - f(t, x, v)) K(t, x, v, w) \dd w. 
\eeqs
Then we show:
\begin{theorem}[Conditional upper bounds on the fundamental solution]\label{thm:upperbounds-conditional}
Let $x, v, y_0, w_0 \in \R^d$ and $0 < \tau_0 < \sigma < T$.
Let $J$ be the fundamental solution of \eqref{eq:1.1} in $[0, T] \times \R^{2d}$ with a vanishing source term $h = 0$, and with pointwise polynomial velocity decay. 
We assume the kernel is coercive \eqref{eq:coercivity}, satisfies integral upper bounds \eqref{eq:upperbound}, and is symmetric in divergence form \eqref{eq:symmetry}, and if $s\geq \frac{1}{2}$ symmetric in non-divergence form \eqref{eq:symmetry-nondiv}.
If, moreover, there exists a constant $c > 0$, such that for every $\rho > 0$, there holds
\beq\label{eq:meyers}
	J(\sigma, x, v; \tau_0, y_0, w_0) \leq J_\rho(\sigma, x, v; \tau_0, y_0, w_0) + c(\sigma - \tau_0) \rho^{-(2d(1+s)+2s)}, 
\eeq
then there exists $C > 0$ depending on $c, s, d, \Lambda_0, \lambda_0$ such that
\bal\label{eq:J-up-conditional}
	J&(\sigma,x, v; \tau_0, y_0, w_0) \\
	&\leq C (\sigma - \tau_0)^{-\frac{2d(1+s)}{2s}}\\
	&\qquad \times \left[ 1 +\frac{\max\left\{\abs{v- w_0}^{2s},  \abs{x - y_0 - (\sigma - \tau_0) (v-w_0)}^{\frac{2s}{1+2s}}\right\}}{\sigma - \tau_0}\right]^{-\frac{2d(1+s) + 2s}{2s}}.
\eal
\end{theorem}
\begin{remark}
As outlined in Remark \ref{rmk:cc-fund-sol}, the major sub-optimality in Theorem \ref{thm:upperbounds} came from the non-singular contribution of the fundamental solution.
The condition in \eqref{eq:meyers} precisely says that we can consider the non-singular part separately, so that we only need to use Aronson's method on the singular term, for which we get the optimal decay at least in velocity. 

In the parabolic literature, \eqref{eq:meyers} is established (with a corresponding decay in $\rho$ of order $d+2s$ instead of $2d(1+s) +2s)$ using parabolic maximum principles, see \cite{gh08, gh14}, \textit{provided that the kernels are pointwisely bounded from above}: there exists $\Lambda > 0$ such that
\beq\label{eq:upperbounds-pw}
	\forall v, w \in \R^d \qquad K(v, w) \leq \Lambda\abs{v-w}^{-(d+2s)}.
\eeq
This assumption does not seem to be easily relaxable to integral upper bounds \eqref{eq:upperbound}. Moreover, even if it is conceivable that the methods from the parabolic setting could be carried over, it is not clear that one would get the expected decay of $2d(1+s) + 2s$ on $\rho$.
\end{remark}

\subsection{Contribution}
On the one hand, Theorem \ref{thm:strong-Harnack} is the first Strong Harnack inequality for \textit{hypoelliptic} equations without additional tail error term. We have announced the result in \cite{AL-ls}, yet the derivation of the local tail bound \cite[Proposition 3.2]{AL-ls} that is sketched there contains an issue in Step 4, as it would require the supremum to be linear. The method proposed in this article circumvents this issue, by deriving a local tail bound on upper level sets, exploiting the sign of the cross terms, so that it can be directly incorporated into the local boundedness result, if we perform a De Giorgi iteration based on a gain of integrability in $L^1$. 
Provided that the equation is satisfied for all $v \in \R^d$, the Strong Harnack inequality improves the ``Not-so-Strong Harnack inequality" previously obtained in \cite[Theorem 1.3]{AL}, which stated a non-linear relation between the supremum in the past and the infimum in the present. This non-linear relation could be made linear at the prize of a non-local tail term, as was outlined in \cite[Remark 4.2]{AL}, and thus the key to the Strong Harnack inequality \eqref{eq:strong-Harnack} lies in capturing the behaviour of the non-local tail term, which we introduce for any $0 < r < R$ and any $v_0 \in \R^d$ as
\beq\label{eq:tail}
	\forall v \in B_r(v_0): \qquad \mathbf T(f; r, R, v_0) :=  \int_{\R^d \setminus B_R(v_0)} f(w) K(v, w) \dd w.
\eeq
The tail measures the non-singular part of the non-local diffusion operator $\mathcal L$ defined in \eqref{eq:1.2}, see also \cite{PKDC, PKDC-2}. In particular, we want to derive a bound on the tail in terms of local quantities only.
This was first achieved for parabolic equations in \cite[Theorem 1.9]{kassmann-weidner}.
In contrast to the parabolic case, we prove a local tail bound \textit{on upper level sets}, so that when we derive the energy estimate and the gain of integrability on level set functions, we can directly incorporate this local tail bound to obtain purely local estimates, which can then be iterated with a De Giorgi argument. This yields a \textit{linear} relation between the $L^\infty$ and the $L^1$ norm of a \textit{solution} (rather than merely a sub-solution) to \eqref{eq:1.1} as stated in the following  proposition, which in turn implies the Strong Harnack inequality, but which is of independent interest. We emphasise that for the tail bound to hold on level set functions, we require \eqref{eq:1.1} to hold for all $v\in\R^d$.
\begin{proposition}[$L^1$-$L^\infty$ bound]\label{prop:L2-Linfty-A}
Let $R > 0$ and $z_0 = (t_0, x_0, v_0) \in \R^{1+2d}$. Let $\mathcal I :=[-R^{2s} + t_0, t_0]$ and let $\Omega_x := B_{R^{1+2s}}(x_0 + (t-t_0) v_0)$.
Let $f \in L^2(\mathcal I \times \Omega_x; H^s(\R^d))$ be a non-negative solution of \eqref{eq:1.1}-\eqref{eq:1.2} in $\mathcal I \times \Omega_x \times \R^d$ in the sense of Definition \ref{def:solutions}, 
with a kernel satisfying the coercivity condition \eqref{eq:coercivity}, the upper bound \eqref{eq:upperbound} the symmetries \eqref{eq:symmetry}, and (in case that $s \geq \frac12$) \eqref{eq:symmetry-nondiv}, with an essentially bounded source term $h \geq 0$. Then for any $\delta \in (0, 1)$ there exists a large constant $C > 1$ depending on $\delta, s, d, \lambda_0, \Lambda_0$ such that 
\beqs
	 \textrm{for a.e. $z_1 \in Q_{\frac{R}{8}}(z_0)$:} \quad  f(z_1) \leq  C R^{-(2d(1+s) + 2s)}\int_{Q_R(z_0)} f(z) \dd z + \delta R^{2s} \norm{h}_{L^\infty(Q_{R}(z_0))}. 
\eeqs
\end{proposition}
To the best of our knowledge, Proposition \ref{prop:L2-Linfty-A} is the first occurrence of the First De Giorgi Lemma for hypoelliptic integro-differential equations, where the supremum is bounded linearly by the $L^1$ norm without any tail terms. Even in the non-local parabolic case, where the Strong Harnack inequality in \cite[Theorem 1.1]{kassmann-weidner} naturally implies this linear $L^1$-$L^\infty$ bound, such a statement has not been derived directly. In particular, a tail bound on upper level sets has not been deduced before, not even for parabolic equations.
The strength of this statement lies in the fact that it describes the local behaviour of a solution to a non-local equation in terms of local quantities only. The weakness, however, is that this result applies only to functions solving \eqref{eq:1.1} for any $v \in \R^d$. Even though this limits the interest of the statement in itself, it does not affect the consequences on the fundamental solution derived in Theorem \ref{thm:upperbounds} and Theorem \ref{thm:lowerbounds}. Moreover, 
if we consider global solutions of the non-cutoff Boltzmann equation, then these satisfy pointwise decay estimates \cite{IMSfrench}, which seem to be sufficient to deduce a boundedness condition on the antisymmetric part of the kernel, which in turn is sufficient to derive the results of this paper for functions solving the Boltzmann equation for any $v \in \R^d$. In short, concerning applications to a physically relevant kinetic model, the global in $v$ assumption is justifiable.

Furthermore, or rather as a consequence of the local nature of Proposition \ref{prop:L2-Linfty-A}, Theorem \ref{thm:upperbounds} and Theorem \ref{thm:lowerbounds} are the first occurrence of bounds on the fundamental solution of a kinetic integro-differential equation. Previous results have discussed two-sided estimates either for kinetic equations with a local diffusion operator with rough, but uniformly elliptic coefficients \cite{lanconelli-pascucci, lanconelli-pascucci-polidoro, auscher-imbert-niebel-2}, or for non-local parabolic equations, with kernels that satisfy pointwise upper bounds, divergence form symmetry \eqref{eq:symmetry} and coercivity as in \eqref{eq:coercivity}, see \cite{KW-aronson}. In this paper, we use Aronson's method, originally conceived for local parabolic equations \cite{aronson-1, aronson-2}, and establish a strategy based thereon for kinetic integro-differential equations. This yields both the exponential lower bounds from Theorem \ref{thm:lowerbounds}, and the polynomial upper bounds from Theorem \ref{thm:upperbounds}.

\subsection{Literature}

\subsubsection{On the Strong Harnack inequality}
For \textit{local} equations, the Strong Harnack inequality simply follows from the Weak Harnack inequality combined with the $L^2$-$L^\infty$ bound in the first De Giorgi Lemma \cite{DeGiorgi, moser}. We refer the reader to \cite[Theorem 4]{GIMV} and \cite[Theorem 5]{JGCM} for the proof of the Strong Harnack inequality for local kinetic equations. 

For \textit{non-local elliptic} equations, a Strong Harnack inequality was derived in \cite{coz17}: first they show a local $L^2$-$L^\infty$ bound on the solution with a non-local tail term, then they use the Weak Harnack inequality for elliptic non-local equations and a bound on the non-local tail term in terms of local quantities only, to obtain the full Harnack inequality. 

For \textit{non-local parabolic} equations, on the other hand, until recently there had only been probabilistic methods to prove a Strong Harnack inequality without any tail term \cite{bass-levin, chen-kumagai, chen-kumagai-wang-20, chen-kumagai-wang-21}. In \cite{bass-levin,  chen-kumagai}, they consider pointwisely bounded kernels, and establish the Strong Harnack inequality using the corresponding stochastic jump process and heat kernel estimates. In \cite{chen-kumagai-wang-20, chen-kumagai-wang-21} they consider nonlocal Dirichlet forms on doubling metric measure spaces, and characterise Harnack Inequalities in terms of conditions on the underlying geometry and assumptions on the kernel. The first \textit{non-stochastic} approach for parabolic equations was derived in Kassmann-Weidner \cite{kassmann-weidner}. Their idea is to first improve the local boundedness result from \cite{KW-a} by bounding the local $L^\infty$ norm of a solution by the sum of its local $L^2$ norm and a non-local tail term in $L^1$. To this end, they introduce two strategies, both relying on a ``tail corrector term", essentially the $L^1$ norm in time of the tail \eqref{eq:tail} with a negative sign, which has the property that it can be added to a solution to the equation such that the sum of these two terms remains a solution to the equation, and thus the tail corrector term has the effect of cancelling the tail terms arising through the equation. As a second step, they improve the Weak Harnack inequality from \cite{KW-b} by bounding the $L^1$ tail by the local infimum of the solution. Such a tail bound is contained in Proposition \ref{prop:tail-bound} below, if we consider a solution that is constant in $x$, so that we reduce the hypoelliptic case to a parabolic setting, and if we set $l = 0$.
Combining both steps yields the Strong Harnack inequality without any tail term. 

Finally we remark that the Strong Harnack inequality fails in the \textit{hypoelliptic} case if the equation is satisfied merely in a bounded velocity domain, \cite{KW-failedH}.

\subsubsection{On bounds on the fundamental solution}
Regarding the case of \textit{local parabolic} equations, the study of two-sided estimates on the fundamental solution originated with the papers of Nash in 1958 \cite{nash}, Aronson in 1967-68 \cite{aronson-1, aronson-2}, and Fabes and Stroock in 1986\cite{fabes-stroock}. 
More precisely, on the one hand, the first proof of Gaussian bounds for linear parabolic second-order equations has been derived by Aronson for uniformly elliptic, but rough, diffusion coefficients in divergence form. Aronson's derivation is based on energy estimates and Harnack inequalities, and thus exploits the regularity properties of the solution to the underlying equation. On the other hand, the argument of Nash, Fabes and Stroock proves the Gaussian bounds without referring to the regularity properties of the solution, and instead implies estimates on the local behaviour of the solution as a consequence of the Gaussian bounds. When Nash solved Hilbert's 19th problem on the continuity of solutions to equations with uniformly elliptic rough diffusion coefficients \cite{nash}, he stated without proof estimates on the fundamental solution. These have been revisited by Fabes and Stroock \cite{fabes-stroock}, who combine Nash's sketch of the argument together with Davies' idea of replacing the Euclidean distance by the Riemannian distance associated to the coefficients of the equation at hand \cite{davies}. With this, Fabes-Stroock obtain Gaussian upper and lower bounds, from which they further derive the Harnack inequalities. 

Variants of these ideas have been applied to \textit{local kinetic} equations, most notably in \cite{lanconelli-pascucci, lanconelli-pascucci-polidoro}, or also in the recent preprint \cite{auscher-imbert-niebel-2}. For \textit{non-local parabolic} equation, the most recent work that we are aware of is by Kassmann-Weidner \cite{KW-aronson}, who employ ideas from Aronson's method to derive a polynomial upper bound on the fundamental solution.

We end this discussion with a historical remark: in 1958 Nash \cite{nash} derived Hölder continuity of parabolic equations in divergence form with rough coefficients, and stated two-sided non-Gaussian estimates for the fundamental solution as a consequence. In 1986, Fabes and Stroock \cite{fabes-stroock} derived two-sided Gaussian estimates directly from Nash's ideas, and derived the Harnack inequalities as a consequence. Since we know how to derive the Harnack inequalities for \eqref{eq:1.1}, one might investigate the derivation of the Harnack inequalities for \eqref{eq:1.1} as a consequence of the bounds on the fundamental solution. We assemble this remark into a collection of open problems:
\begin{openprob}
\begin{enumerate}
\item Derive the Harnack inequalities of Theorem \ref{thm:weakH} and Theorem \ref{thm:strong-Harnack} for solutions to \eqref{eq:1.1} as a consequence of the polynomial upper and the exponential lower bounds of Theorem \ref{thm:upperbounds} and Theorem \ref{thm:lowerbounds}. Instead of an exponential lower bound, one might need to assume a polynomial lower bound on the fundamental solution. It is also conceivable that it requires the optimal upper bounds.
\item Is the exponential lower bound on the fundamental solution of \eqref{eq:1.1} in Theorem \ref{thm:lowerbounds} optimal for kernels with an average coercivity condition \eqref{eq:coercivity}, an average upper bound \eqref{eq:upperbound} and with the divergence form symmetry \eqref{eq:symmetry}? A polynomial lower bound might require a pointwise lower bound on the kernel.
\item Improve the polynomial upper bounds of Theorem \ref{thm:upperbounds}, or in other words, make Theorem \ref{thm:upperbounds-conditional} unconditional. Again, we think this requires a pointwise upper bound on the kernel \eqref{eq:upperbounds-pw} instead of the average upper bound assumed in \eqref{eq:upperbound}.
\item Does the linear $L^1$ to $L^\infty$ bound hold for mere \textit{sub-solutions} to \eqref{eq:1.1}? The proof method in this article heavily relies on the fact that we work with solutions. However, in the local case, a linear $L^1$ to $L^\infty$ bound can be derived even for sub-solutions. It seems to us that the sub-solution property is not sufficient in the non-local case to yield a linear $L^1$ to $L^\infty$ bound.
\end{enumerate}
\end{openprob}

\subsection{Strategy}

\subsubsection{For the Strong Harnack inequality}

The first step is to derive a tail bound in terms of local quantities: we show that the $L^1$ norm in time and space of the non-local tail \eqref{eq:tail} is bounded by a local Lebesgue norm. We work on level set functions, so that this step is compatible with the local boundedness result that we aim for in a second step. This introduces two main complications; firstly, the fact that a level set function is not a super-solution of the equation anymore, and secondly, the appearance of cross terms, that is when the test function and the non-locality are on different sides of the level $l$:
\beqs
	\chi_{f > l}(v) \chi_{f < l}(w).
\eeqs
The appearance of these cross terms is a purely non-local effect; in the local case, these terms naturally vanish.
We test the equation with a root of the level set function. Due to the concavity of the test function, this gives rise to a signed term if both the local and the non-local function are above the level. Moreover, on the upper level set, the tail has the good sign as well, since we work with super-solutions. The error terms occur on the lower level set and through the cross terms, both of which exploit the global decay in velocity of the solution. The bound for the cross terms moreover relies on the concavity of the test function, whereas the term that arises when the solution and the test function are both in the lower level set is bounded by an appropriate choice of the cutoff function. 


The goal in a second step is to derive a linear $L^1$-$L^\infty$ bound by incorporating the tail bound into the energy estimate and the gain of integrability. This yields local estimates that we can iterate à la De Giorgi. To this end, we work with solutions, which is in contrast to the local case, where a linear $L^1$-$L^\infty$ bound holds for mere sub-solutions.
It might be interesting to investigate whether it is sufficient to work with \textit{sub-solutions} in the non-local case to deduce a linear local boundedness result without tail terms.

\subsubsection{For the bounds on the fundamental solution}
As a consequence of the Strong Harnack inequality, Theorem \ref{thm:strong-Harnack}, we adapt Aronson's method to the non-local kinetic case. In fact, Proposition \ref{prop:L2-Linfty-A} implies directly an on-diagonal upper bound for the fundamental solution to \eqref{eq:1.1}. For the off-diagonal bounds, the idea is to construct a function $H$ that satisfies 
\beqs
	\partial_t H + v \cdot\nabla_x H + \int_{\R^d} \left[H^{\frac{1}{2}}(w) - H^{\frac{1}{2}}(v)\right]^2 K(v, w) \dd w \leq 0. 
\eeqs
By testing \eqref{eq:1.1} with the product of the solution and this function $H$, we derive an energy estimate on the $L^2$ norm of $f H^{\frac{1}{2}}$, which in turn permits to determine the decay of the solution from the decay of the function $H$. Together with the on-diagonal bounds that follow as a consequence of the linear $L^1$-$L^\infty$ bound of Proposition \ref{prop:L2-Linfty-A}, this implies the unconditional and conditional upper bounds on the fundamental solution of Theorem \ref{thm:upperbounds} and Theorem \ref{thm:upperbounds-conditional}.

For \(s\in(\frac12,1)\), the exponential lower bound of Theorem \ref{thm:lowerbounds} follows by applying the Strong Harnack inequality of Theorem \ref{thm:strong-Harnack} repeatedly along a path that connects two given points. This yields a pointwise Harnack inequality, which in turn implies an exponential lower bound on any solution that has a non-vanishing average over some local ball. In particular, it implies the exponential bound on the fundamental solution that connects two given points.

The assumptions on the non-local kernel are described in the next subsection.

\subsection{Ellipticity class of the non-local kernel $K$}\label{subsec:ellipticity-class}
The regularity of integro-differential equations relies on the conditions that are assumed on the kernel of the non-local operator \eqref{eq:1.2}. A widely studied non-local diffusion operator is the fractional Laplacian \eqref{eq:frac-lapl}, see \cite{hitchhiker}. Based on this, a commonly used assumption is that the kernel is pointwisely bounded from above and below by the kernel of the fractional Laplacian, and satisfies the pointwise symmetry assumptions, $K(v, w) = K(w, v)$ and $K(v, v+w) = K(v, v-w)$, \cite{coz17, kim19, kim20}. In \cite{felsinger-kassmann, KW-a, KW-b}, a more general class of non-local kernels was considered for which the Weak Harnack inequality is derived. They relax the pointwise upper and lower bounds to conditions in integral form \cite{felsinger-kassmann}, and additionally they relax the pointwise symmetry assumptions in \cite{KW-a, KW-b}. 

In the case of kinetic equations, the assumptions on the kernel are guided by the Boltzmann equation. Even though the kernel depends in general on time and space we will often omit to write out this dependency for the sake of brevity. We let $0 < \lambda_0 < \Lambda_0$ and $s \in (0, 1)$. Let $\varphi : \R^d \to \R$ and $\Omega \subseteq \R^d \times \R^d$. Then we assume that $K$ is coercive in the sense that
\beq
	\iint_{\Omega} \big(\varphi(v)-\varphi(w)\big)^2K(v, w) \dd w\dd v \geq \lambda_0 \iint_{\Omega}   \frac{\abs{\varphi(v)-\varphi(w)}^2}{\abs{v-w}^{d+2s}} \dd v\dd w.
\label{eq:coercivity}
\eeq
Moreover, we require the following upper bound for $r > 0$
\beq
	\forall v \in \R^d \quad \int_{\R^d \setminus B_r(v)}K(v,w)\dd w \leq \Lambda_0 r^{-2s}.
\label{eq:upperbound}
\eeq
Furthermore, we assume a pointwise divergence form symmetry:
\beq
	\forall v, w \in \R^d \quad K(v, w) = K(w, v).
\label{eq:symmetry}
\eeq
The local analogue is an operator in divergence form. In case that $s \geq \frac{1}{2}$ we assume a pointwise non-divergence form symmetry:
\beq
	\forall v, w \in \R^d \quad K(v, v+w) = K(v, v-w).
\label{eq:symmetry-nondiv}
\eeq
This is similarly the non-local analogue of an operator in non-divergence form. The coercivity \eqref{eq:coercivity} and the upper bound \eqref{eq:upperbound} are inspired from the series of work by Imbert-Mouhot-Silvestre on the regularity of solutions of the non-cutoff Boltzmann equation conditional to certain macroscopic bounds \cite{IS, ISschauder, ISglobal, IS59, IMS, IS?}. For in fact, provided that certain macroscopic quantities remain bounded, the non-cutoff Boltzmann equation (see \cite{ISglobal}) can be understood as an integro-differential equation of the form \eqref{eq:1.1}, with a kernel satisfying \eqref{eq:coercivity}-\eqref{eq:upperbound}, and a weaker notion of divergence form symmetry than \eqref{eq:symmetry}, which encodes that on average the non-symmetric part of the kernel is of lower order. 
We aim to address the Strong Harnack inequality for the non-cutoff Boltzmann equation in a follow-up article.


\subsection{Invariant transformations}
Finally, we introduce the kinetic cylinders that define the domain of solutions to \eqref{eq:1.1}. Let $f$ solve \eqref{eq:1.1}. Then for $r \in [0, 1]$ the scaled function $f_r(t, x, v) = f(r^{2s}t, r^{1+2s}x, rv)$ satisfies
\beqs
	\partial_t f_r + v\cdot \nabla_x f_r = \mathcal L_r f_r + h_r,
\eeqs
where $\mathcal L_r$ is the non-local operator associated to the scaled kernel
\beqs
	K_r(t, x, v, w) = r^{d+2s} K(r^{2s}t, r^{1+2s}x, r v, rw), 
\eeqs
and the source term is scaled as
\beqs
	h_r(t, x, v) = r^{2s} h(r^{2s}t, r^{1+2s}x, r v).
\eeqs
For $r \in [0, 1]$ the scaled kernel $K_r$ satisfies \eqref{eq:coercivity}-\eqref{eq:symmetry-nondiv}. Moreover, $\norm{h_r}_{L^\infty(Q_1)} \leq r^{2s}\norm{h}_{L^\infty(Q_1)} \leq \norm{h}_{L^\infty(Q_1)}$.

Furthermore, the equation is invariant under Galilean transformations, which, we recall, are given by $z \to z_0 \circ z = (t_0 + t, x_0+x + tv_0, v_0 + v)$ with $z_0 = (t_0, x_0, v_0)\in \R^{1+2d}$. If $f$ is a solution of \eqref{eq:1.1} then its Galilean transformation $f_{z_0}(z) = f(z_0 \circ z)$ solves
\beqs
	\partial_t f_{z_0} + v\cdot \nabla_x f_{z_0} = \mathcal L_{z_0} f_{z_0} + h_{z_0},
\eeqs
where $\mathcal L_{z_0}$ is the non-local operator associated to the translated kernel
\beqs
	K_{z_0}(t, x, v, w) = K(z_0 \circ z, v_0 + w), 
\eeqs
and with source
\beqs
	h_{z_0}(t, x, v) = h(z_0 \circ z).
\eeqs
Again the modified kernel $K_{z_0}$ satisfies \eqref{eq:coercivity}-\eqref{eq:symmetry-nondiv} provided that $K$ does, and $h_{z_0}$ is bounded provided that $h$ is. 

In view of these invariances we define kinetic cylinders 
\beqs
	Q_r(z_0) := \{(t, x, v) : -r^{2s} \leq t - t_0 \leq 0, \abs{v - v_0} < r, \abs{x - x_0 -(t-t_0)v_0} < r^{1+2s}\}
\eeqs
for $r > 0$ and $z_0 = (t_0, x_0, v_0) \in \R^{1+2d}$. For later reference, we also introduce the cylinder shifted to the past
\beqs
	Q^-_r(z_0) := Q_r\big(z_0 - (2r^{2s}, 2r^{2s}v_0, 0)\big),
\eeqs
so that in particular for $z_0 = 0$
\beqs
	Q^-_r := Q_r(-2r^{2s}, 0, 0) = (-3r^{2s}, -2r^{2s}] \times B_{r^{1+2s}} \times B_r.
\eeqs
Similarly the cylinder shifted to the future is denoted as
\beqs
	Q^+_r(z_0) := Q_r\big(z_0 + (2r^{2s}, 2r^{2s}v_0, 0)\big).
\eeqs

\subsection{Notation}
We write $z := (t, x, v) \in \R^{1+2d}$ for a time, space, and velocity variable. We denote by $\mathcal T$ the transport operator $\mathcal T := \partial_t + v\nabla_x$. 

In general, our domains $Q \subset \R^{1+2d}$ are cylinders in time, space and velocity. If we slice the cylinder in time, we denote it by $Q^t \subset \R^{2d}$, that is $Q^t = \left(\{t\} \times \R^d \times \R^d\right)\cap Q$.

The positive part of $a \in \R$ is denoted by $a_+ = \max\{a, 0\} \geq 0$. Correspondingly, we write for the negative part $a_- = \min\{a, 0\} \leq 0$. In particular, any $a$ can be written as $a = a_+ + a_-$. 

The bilinear form associated to the non-local operator $\mathcal Lf$ occurring in \eqref{eq:1.1} is defined for any $(t, x, v) \in \R^{1+2d}$ by
\beq\label{eq:bilinear-form}
	\mathcal E(f, g) = -\int_{\R^d} \mathcal Lf(v) g(v) \dd v = \int_{\R^d}\int_{\R^d} \big(f(v) - f(w)\big) g(v) K(v, w) \dd w \dd v.
\eeq
Note that the Dirichlet form associated to a local operator in divergence form has a positive sign.

\subsection{Outline} The paper continues by stating some preliminaries in Section \ref{sec:prelims}. We then bound the non-local tail on the upper level set in terms of a local Lebesgue norm in Section \ref{sec:tail-bound}. Section \ref{sec:strategy-A} demonstrates Proposition \ref{prop:L2-Linfty-A}. We conclude the proof of Theorem \ref{thm:strong-Harnack} in Section \ref{sec:strong_H}. Finally, the upper bound on the fundamental solution is discussed in Section \ref{sec:upperbound}, whereas the lower bound is derived in Section \ref{sec:lowerbound}.

\section{Preliminaries}\label{sec:prelims}

\subsection{The non-local kernel}
The results in this subsection are stated in a simplified manner since we are only concerned with symmetric kernels in the sense of \eqref{eq:symmetry}.
Here we summarise two observations on the non-local kernel $K$ of \eqref{eq:1.2} which are also used later in this paper. The first theorem states the boundedness in $H^s$ in velocity of the bilinear functional associated to the non-local operator $\mathcal L$ from \eqref{eq:1.2}. 
\begin{theorem}[$H^s$ estimate \protect{\cite[Theorem 4.1, Corollary 5.2]{IS},  \cite[Theorem 2.1]{AL}}]\label{thm:boundedness-Hs}
Let $K : B_{\bar R} \times \R^d \to \R$ be a non-negative kernel satisfying \eqref{eq:upperbound} and \eqref{eq:symmetry}. For any $f,g  \in \dot H^s(\R^d)$ supported in $B_{\frac{\bar R}{2}}$ there holds
\beqs
	\mathcal E(f, g) \leq \Lambda \norm{f}_{\dot H^s(\R^d)} \norm{g}_{\dot H^s(\R^d)},
\eeqs
where $\mathcal E$ is the bilinear form \eqref{eq:bilinear-form} associated to the non-local operator $\mathcal L$ of \eqref{eq:1.2}.
\end{theorem}
Theorem \ref{thm:boundedness-Hs} also motivates our notion of solutions to \eqref{eq:1.1}. 
\begin{definition}[Notion of solutions]\label{def:solutions}
Assume $0 \leq K$ satisfies \eqref{eq:coercivity}-\eqref{eq:symmetry} and for $\bar R > 0$, let $(t_0, t_1) \times \Omega_x := \big(-(\frac{\bar R}{2})^{2s}, 0\big)\times B_{(\frac{\bar R}{2})^{1+2s}}$ and $\Omega_v \subseteq \R^d$.
We say that $f:(t_0, t_1)\times \Omega_x\times \R^d \to \R$ is a solution of \eqref{eq:1.1}-\eqref{eq:1.2} in $(t_0, t_1) \times \Omega_x \times \Omega_v$ for a source $h \in L^2\big((t_0, t_1) \times \Omega_x \times \Omega_v\big)$, if 
\begin{enumerate}
	\item $f \in C^0((t_0, t_1), L^2(\Omega_x\times \Omega_v))\cap L^2((t_0, t_1) \times \Omega_x, L^\infty(\R^d) +H^s\big(\R^d))$,
	\item $\mathcal T f \in L^2((t_0, t_1)\times \Omega_x \times \Omega_v)$,
	\item $f \geq 0$ almost everywhere in $(t_0, t_1) \times \Omega_x \times \Omega_v$,
	\item $f$ conserves mass,
	\item for some $q > d$ and some $C_q$ depending on $q$, we know that $f(t, x, v) \leq C_q \langle v \rangle^{-q}$ for almost every $(t, x, v) \in [t_0, t_1]\times \Omega_x\times \R^d$,
	\item for all non-negative $\varphi \in L^2((t_0, t_1) \times \Omega_x,  H^s(\Omega_v))$ such that $\forall (t,x) \in (t_0, t_1) \times \Omega_x$, $\varphi(t, x, \cdot)$ is supported in $\Omega_v$, there holds
	\beqs
		\int_{(\tau_0, \tau_1)\times \omega_x \times \Omega_v} (\mathcal T f)\varphi\dd z + \int_{(\tau_0, \tau_1)}\int_{\omega_x}\mathcal E(f, \varphi)\dd x\dd t - \int_{(\tau_0, \tau_1)\times \omega_x \times \Omega_v} h\varphi\dd z = 0,
	\eeqs
	for any $t_0 \leq \tau_0 \leq \tau_1 \leq t_1$, and any subset $\omega_x \subseteq \Omega_x$.
\end{enumerate}
\end{definition}
\begin{remark}\label{rmk:indicator-testing}
In fact, in case that $s \geq \frac{1}{2}$, we want to work with a stronger notion of solutions which allows the test function $\varphi$ to be in $L^2((t_0, t_1) \times \Omega_x,  H^{\frac{s}{2}}(\Omega_v))$. We argue that there are examples of solutions $f$ to \eqref{eq:1.1} that satisfy higher differentiability than what would normally be expected from an energy estimate, namely $f \in L^2((t_0, t_1) \times \Omega_x,  H^{\frac{3s}{2}}(\Omega_v))$. This is true for instance if we work with constant coefficients, that is $K(v, w) \sim \abs{v-w}^{-(d+2s)}$. We don't know if this is a general principle, but we line out that for the purposes of this article, we can show that test functions in $H^{\frac{s}{2}}(\Omega_v)$ are admissible. The justification is based on a comparison principle involving the fundamental solution to the equation with constant coefficients, which can be found in Appendix \ref{sec:approx}. We argue that \textit{qualitatively} the solution is in $H^{\frac{3s}{2}}(\Omega_v)$, so that we can test the equation with a function in $H^{\frac{s}{2}}(\Omega_v)$, which does not affect the final results in this paper \textit{quantitatively}. 

We emphasise this for the purpose of the tail estimate in Proposition \ref{prop:tail-bound}, where we use a test function which is merely in $H^{\frac{\bar s}{2}}(\Omega_v)$ for any $\bar s \in (0, 1)$. 
\end{remark}

The second observation is a non-local analogue of the product rule.
\begin{lemma}[Commutator estimates \protect{\cite[Lemmata 4.10, 4.11]{IS}}]\label{lem:commutator}
Let $K$ be a non-negative kernel satisfying \eqref{eq:upperbound}. Let $D$ be a closed set and $\Omega$ open such that $D \Subset \Omega \subset B_{\bar R/2} \subset \R^d$ for $0 < \bar R \leq 2$. Let $\varphi$ be a smooth function with support in $D$, let $f \in H^s(\Omega) \cap L^\infty(\R^d)$ and let $\rho = \frac{\textrm{dist}(D, \R^d\setminus \Omega)}{2}$. We write 
\beqs
	\mathcal L [\varphi f] - \varphi \mathcal L f = h_1 + h_2, 
\eeqs
where $h_1, h_2$ are given by
\bals
	&h_1(v) = \int_{\R^d\setminus B_\rho(v)} f(w) \big(\varphi(w) - \varphi(v)\big) K(v, w) \dd w, \\
	&h_2(v) = \int_{B_\rho(v)} f(w) \big(\varphi(w) - \varphi(v)\big) K(v, w) \dd w.
\eals
Then, by construction $h_2 = 0$ outside $\Omega$, and there holds
\beq\label{eq:h1}
	\norm{h_1}_{L^2(\R^d\setminus \Omega)} \leq \Lambda  \rho^{-2s} \norm{\varphi}_{L^\infty} \norm{f}_{L^2(D)}.
\eeq
Moreover, if $s \in (0, 1/2)$, there holds
\beq\label{eq:h2}
	\norm{h_2}_{L^2(\R^d)} \leq \Lambda \rho^{1-2s} \norm{\varphi}_{C^1} \norm{f}_{L^2(\Omega)}.
\eeq
Else if $s \in [1/2, 1)$ and if additionally $K$ satisfies \eqref{eq:symmetry}, then there exists $h_{22}, h_{23} \in L^2(\R^d)$ such that 
\beqs
	h_2 = h_{22} + (-\Delta_v)^{\frac{s}{2}} h_{23},
\eeqs
with
\bal\label{eq:h2-2}
	\Norm{h_{22}}_{L^2(\R^d)} &\leq \Lambda \rho^{2-2s} \norm{\varphi}_{C^2} \norm{f}_{L^2(\Omega)} +\Lambda \rho^{1-s}  \norm{\varphi}_{C^1}\norm{f}_{H^s(\Omega)}, \\
	\Norm{h_{23}}_{L^2(\R^d)} &\leq \Lambda \rho^{1-s}  \norm{\varphi}_{C^1}\norm{f}_{L^2(\Omega)}.
\eal
\end{lemma}
We reprove this lemma to make the dependence on $\rho$ in \eqref{eq:h2} and \eqref{eq:h2-2} precise. We need it for the optimal exponent on the radius in Proposition \ref{prop:L2-Linfty-A}. 
\begin{proof}
We let $E = D + B_\rho$ so that $D \Subset E \Subset \Omega$ with $\textrm{dist}(D, \R^d \setminus E) = \rho$ and $\textrm{dist}(E, \R^d \setminus \Omega) = \rho$.
\begin{figure}{
\centering
\begin{tikzpicture}[scale=2.2]
  \draw (-5,1) -- (-5.05, 1) node[anchor=east, scale =0.7] {\footnotesize{$1$}};
  \draw (-5,0) -- (-5.05, 0) node[anchor=east, scale =0.7] {\footnotesize{$0$}};
  \draw (-3, 0) -- (-3, -0.05) node[anchor=north, scale =0.7] {\small$D$};
  \draw (-1.5, 0) -- (-1.5, -0.05) node[anchor=north, scale =0.7] {\small$E$};
  \draw (0, 0) -- (0, -0.05) node[anchor=north, scale =0.7] {\small$\Omega$};
  \draw[WildStrawberry, line width = 0.7pt](-5, 1) -- (-4, 1) node[anchor=south, scale =0.8] {$\varphi$};
  \draw[WildStrawberry, line width = 0.7pt] (-4, 1) -- (-3, 0);
  \draw[WildStrawberry, line width = 0.7pt] (-3, 0) -- (0.5, 0);
  \begin{scope}[on background layer]
 \draw[<->, PineGreen, line width = 1pt] (-3, 0) -- node[anchor=north, scale =0.7] {$\rho$} (-1.5, 0) ;
 \draw[<->, PineGreen, line width = 1pt] (-1.5, 0) -- node[anchor=north, scale =0.7] {$\rho$} (0, 0) ;
    \draw[->] (-5, 0) -- (0.5, 0) node[right] {\footnotesize{$ v $}};
  \draw[->] (-5, 0) -- (-5, 1.1);
  \end{scope}
\end{tikzpicture}}
\caption{Illustration of the smooth cutoff $\varphi$ from Lemma \ref{lem:commutator} with which we commute $f$, and the supports of the cutoff and the local contribution of the commutator. Note $D = \textrm{supp }\varphi$, and $E = \textrm{supp } h_2$.}\label{fig:cutoffs-1}
\end{figure}
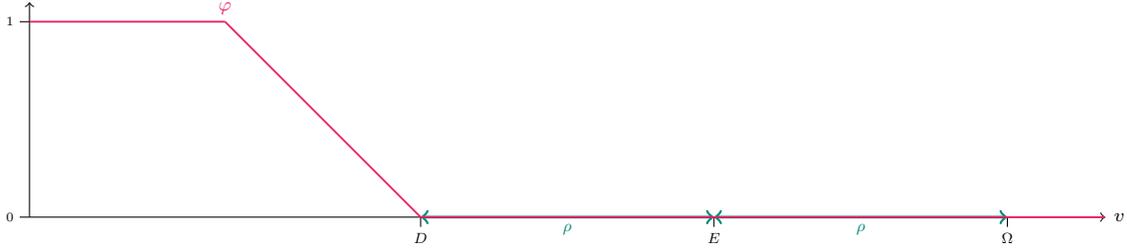

To bound $h_1$, we notice that $\varphi(v) = 0$ for $v \notin D$. Thus if $v \notin \Omega \supset D$, then
\beqs
	h_1 = \int_{\R^d\setminus B_\rho(v)} f(w) \varphi(w) K(v, w) \dd w = \int_{D} f(w) \varphi(w) K(v, w) \dd w.
\eeqs
Therefore, using Cauchy-Schwarz, \eqref{eq:upperbound} and Fubini's theorem
\bals
	\int_{\R^d \setminus  \Omega} h_1^2 \dd v &= \int_{\R^d \setminus  \Omega} \Bigg( \int_{D} f(w) \varphi(w) K(v, w) \dd w\Bigg)^2 \dd v\\
	&\leq \norm{\varphi}_{L^\infty}^2 \int_{\R^d \setminus  \Omega}\Bigg(\int_{D} f^2(w) K(v, w) \dd w\Bigg)\Bigg(\int_{D}  K(v, w) \dd w\Bigg)  \dd v\\
	&\leq \Lambda \rho^{-2s} \norm{\varphi}_{L^\infty}^2\int_{D} f^2(w) \int_{ \abs{v-w} \geq 2\rho} K(v, w) \dd v \dd w\\
	&\leq \Lambda^2 \rho^{-4s} \norm{\varphi}_{L^\infty}^2 \norm{f}_{L^2(D)}^2.
\eals
This yields \eqref{eq:h1}.

To bound $h_2$ we first consider $s \in (0, 1/2)$. We use Cauchy-Schwarz, \eqref{eq:upperbound} for $s < 1/2$ and Fubini
\bals
	\norm{h_2}_{L^2(\R^d)}^2 &= \int_E \Bigg(\int_{B_\rho(v)} f(w) \big(\varphi(w) - \varphi(v)\big) K(v, w) \dd w\Bigg)^2 \dd v\\
	&\leq   \int_E \Bigg(\int_{B_\rho(v)} f^2(w)\Abs{\varphi(w) - \varphi(v)} K(v, w) \dd w\Bigg) \\
	&\qquad \times\Bigg(\int_{B_\rho(v)}  \Abs{\varphi(w) - \varphi(v)} K(v, w) \dd w\Bigg)  \dd v\\
	&\leq  \norm{\varphi}_{C^1}^2 \int_E \Bigg(\int_{B_\rho(v)} f^2(w) \Abs{v-w}K(v, w) \dd w\Bigg) \\
	&\qquad \times \Bigg(\int_{B_\rho(v)}  \Abs{v-w} K(v, w) \dd w\Bigg)  \dd v\\
	&\leq  \Lambda \rho^{1-2s} \norm{\varphi}_{C^1}^2 \int_{\Omega} f^2(w)\int_{E \cap B_\rho(w)}  \Abs{v-w}K(v, w) \dd v \dd w\\
	&\leq \Lambda^2 \rho^{2-4s} \norm{\varphi}_{C^1}^2\norm{f}_{L^2(\Omega)}^2.
\eals
This yields \eqref{eq:h2}.

Second we consider $s \in [1/2, 1)$. Due to \eqref{eq:symmetry}, we find
\bals
	h_2 = \underbrace{\frac{1}{2} \int_{B_\rho(v)} f(w) \big(\varphi(w) - \varphi(v)\big) \big(K(v, w)+K(w, v)\big) \dd w}_{=: h_2^{\textrm{sym}}}. 
\eals
We estimate $h_2^{\textrm{sym}}$ by duality. Let $g \in H^s(\R^d)$. Then, since $\textrm{supp }h_2^{\textrm{sym}} \subseteq E$, we have
\bals
	\int_E h_2^{\textrm{sym}}(v) g(v) \dd v &= \int_E \int_{B_\rho(v)} g(v) f(w) \big(\varphi(w) - \varphi(v)\big)K(v, w) \dd w \dd v\\
	&= \frac{1}{2} \int_{\Omega} \int_{\Omega \cap \abs{v-w} < \rho} f(v) \big(g(v) - g(w)\big)  \big(\varphi(w) - \varphi(v)\big)K(v, w)\dd w \dd v\\ 
	&\quad +\frac{1}{2} \int_{\Omega} \int_{\Omega \cap \abs{v-w} < \rho} g(v) \big(f(w) - f(v)\big)  \big(\varphi(w) - \varphi(v)\big) K(v, w) \dd w \dd v.
\eals
Thus by Cauchy-Schwarz, \eqref{eq:upperbound} and Theorem \ref{thm:boundedness-Hs}
\begin{align*}
	&\int_E h_2^{\textrm{sym}}(v) g(v) \dd v \\
	&\leq \norm{f}_{L^2(\Omega)} \Bigg\{\int_{\Omega} \Bigg(\int_{\Omega \cap \abs{v-w} < \rho}  \big(g(v) - g(w)\big)  \big(\varphi(w) - \varphi(v)\big) K(v, w)\dd w\Bigg)^2 \dd v\Bigg\}^{\frac{1}{2}}\\
	&\quad + \norm{g}_{L^2(\Omega)} \Bigg\{\int_{\Omega} \Bigg(\int_{\Omega \cap \abs{v-w} < \rho}  \big(f(w) - f(v)\big)  \big(\varphi(w) - \varphi(v)\big) K(v, w) \dd w\Bigg)^2 \dd v\Bigg\}^{\frac{1}{2}}\\
	&\leq \norm{f}_{L^2(\Omega)} \Bigg\{\int_{\Omega}\Bigg(\int_{\Omega \cap \abs{v-w} < \rho}  \big(g(v) - g(w)\big)^2 K(v, w) \dd w\Bigg) \\
	&\qquad \qquad \qquad\qquad \qquad\times \Bigg(\int_{\Omega \cap \abs{v-w} < \rho} \Abs{\varphi(w) - \varphi(v)}^2K(v, w) \dd w \Bigg)\dd v\Bigg\}^{\frac{1}{2}}\\
	&\quad + \norm{g}_{L^2(\Omega)} \Bigg\{\int_{\Omega}\Bigg(\int_{\Omega \cap \abs{v-w} < \rho}  \big(f(w) - f(v)\big)^2 K(v, w) \dd w\Bigg) \\
	&\qquad \qquad \qquad\qquad \qquad\times \Bigg(\int_{\Omega \cap \abs{v-w} < \rho} \Abs{\varphi(w) - \varphi(v)}^2K(v, w) \dd w \Bigg)\dd v\Bigg\}^{\frac{1}{2}}\\
	&\leq \Lambda^{\frac{1}{2}} \rho^{1-s} \norm{f}_{L^2(\Omega)} \norm{\varphi}_{C^1} \Bigg\{\int_{\Omega}\int_{\Omega \cap \abs{v-w} < \rho}  \big(g(v) - g(w)\big)^2 K(v, w) \dd w \dd v\Bigg\}^{\frac{1}{2}}\\
	&\quad + \Lambda^{\frac{1}{2}} \rho^{1-s} \norm{g}_{L^2(\Omega)} \norm{\varphi}_{C^1} \Bigg\{\int_{\Omega}\int_{\Omega \cap \abs{v-w} < \rho}  \big(f(w) - f(v)\big)^2 K(v, w) \dd w \dd v\Bigg\}^{\frac{1}{2}}\\
	&\leq \Lambda \rho^{1-s}  \norm{\varphi}_{C^1} \Big( \norm{f}_{L^2(\Omega)}\norm{g}_{H^s(\Omega)} + \norm{g}_{L^2(\Omega)}\norm{f}_{H^s(\Omega)}\Big). 
\end{align*}
This estimate implies that there exists $h_{22}, h_{23}$ such that
\beqs
	h_2 = h_2^{\textrm{sym}} = h_{22}^{\textrm{sym}} + (-\Delta_v)^{\frac{s}{2}} h_{23}^{\textrm{sym}},
\eeqs
with
\bals
	\Norm{h_{22}}_{L^2(\R^d)} \leq  \Lambda \rho^{1-s}  \norm{\varphi}_{C^1}\norm{f}_{H^s(\Omega)}, \qquad \Norm{h_{23}}_{L^2(\R^d)}\leq  \Lambda\rho^{1-s}  \norm{\varphi}_{C^1}\norm{f}_{L^2(\Omega)}.
\eals
\end{proof}

\subsection{Weak Harnack}
We make use of the Weak Harnack inequality derived in \cite[Theorem 1.6]{IS} and \cite[Theorem 1.1]{AL}.
\begin{theorem}[Weak Harnack inequality \protect{\cite[Thm. 1.6]{IS}, \cite[Thm. 1.1]{AL}}]\label{thm:weakH}
Let $f$ be a non-negative super-solution to \eqref{eq:1.1}-\eqref{eq:1.2} in $[-3, 0] \times Q_1^t := [-3, 0] \times B_1 \times B_1$ with a non-negative kernel $K$ satisfying \eqref{eq:coercivity}-\eqref{eq:symmetry}. 
Then there is $C > 0$ and $\zeta > 0$ depending on $s, d, \lambda,\Lambda$ such that for $r_0 < \frac{1}{3}$ the Weak Harnack inequality is satisfied:
\beq
	\Bigg(\int_{\tilde Q_{\frac{r_0}{2}}^-} f^\zeta\dd z\Bigg)^{\frac{1}{\zeta}} \leq C \Big(\inf_{Q_{\frac{r_0}{2}}} f + \norm{h}_{L^\infty(Q_1)}\Big),
\label{eq:weakH}
\eeq
where $\tilde Q_{\frac{r_0}{2}}^-:= Q_{\frac{r_0}{2}}\big((-\frac{5}{2}r_0^{2s} + \frac{1}{2} (\frac{r_0}{2})^{2s}, 0, 0)\big)$, see Figure \ref{fig:harnacks}.
 \label{thm:harnack}
\end{theorem}
The proof of the Weak Harnack inequality follows the De Giorgi method: in a first step one shows the local gain of regularity from $L^2$ to $L^\infty$, then in a second step from $L^\infty$ to $C^\alpha$. The Weak Harnack inequality follows in the end from the $C^\alpha$ regularity by using a covering argument. We refer the reader to \cite{IS} for a proof that is constructive in case that $s \in (0, 1/2)$ and based on compactness arguments in case that $s \geq 1/2$; or to \cite{AL} for a quantitative proof for any $s \in (0, 1)$. 

\subsection{Iteration argument}

Finally, we use a standard iteration argument, see for example \cite[Lemma 8.18]{GiaquintaMartinazzi}.
\begin{lemma}\label{lem:covering}
Let $\phi: [0, T]\to \R$ be a non-negative bounded function. Suppose that for $0 \leq r < R \leq T$ we have
\beqs
	\phi(r) \leq A(R-r)^{-\alpha} + \delta \phi(R)
\eeqs
for some $A, \alpha > 0$, $0 \leq \delta < 1$. Then there exists a constant $c$ depending on $\alpha, \delta$ such that for $0 \leq r < R \leq T$ we have
\beqs
	\phi(r) \leq c A (R-r)^{-\alpha}.
\eeqs
\end{lemma}

\section{Tail bound on upper level sets}\label{sec:tail-bound}
In this section we derive the first step to the Strong Harnack inequality: we show a bound on the non-local tail term in $L^1$ by a local quantity. The idea, inspired from \cite{kassmann-weidner}, is to test the equation against $f^{-(1-\zeta)}$ for $\zeta \in (0, 1)$ and where $f > 0$ solves \eqref{eq:1.1}, with the aim of obtaining a local bound on the $L^1$ norm of the tail \eqref{eq:tail}. 
If we work with super-solutions, then the tail has the good sign just by virtue of the definition of $\mathcal L$ \eqref{eq:1.2}. Moreover, for $\zeta < 1$, the concavity of the test function gives a signed term. Finally, if we work with non-negative super-solutions, then we end up with the local tail bound in $L^1$.
We then want to incorporate this $L^1$ tail bound in the next step into the derivation of a local boundedness result. However, this boundedness result is established on level set functions in the spirit of De Giorgi. Even though $f-l$ is a solution to \eqref{eq:1.1} for any $l \in\R_+$, the positive part $(f-l)_+$ is merely a sub-solution. The tail bound in this section, however, requires a super-solution. Thus, in order to exploit the $L^1$ tail bound directly in the energy estimate \eqref{eq:aux1-improved-DG}, we test with the level set function to some inverse power, that is  $(f-l)_+^{-(1-\zeta)}$ for any $l \in \R_+$, and for the sake of obtaining a local bound, we localise with some cutoff.
The non-locality in velocity introduces cross terms, that is terms when $f(v)$ and $f(w)$ are on different sides of the level. These are dealt with by exploiting the polynomial velocity decay of the super-solution. Moreover, the term that arises when both, the function and the test function, are below level $l$, can be made sufficiently small by appropriately twisting the cutoff function on the lower level set: its derivative has to be sufficiently small. 

The appearance of cross terms is a purely non-local effect. In order to use the velocity decay of the solution in the whole space, we need to assume \eqref{eq:1.1} to be satisfied for all $v\in \R^d$.

\begin{proposition}[Non-local to local bound]\label{prop:tail-bound}
Let $R > 0$, $z_0 \in \R^{1+2d}$, and denote by $[t_1, t_2] \times \Omega_x$ the set $[t_1, t_2]  \times \Omega_x := [-R^{2s} + t_0, t_0] \times B_{R^{1+2s}}(x_0 + (t-t_0) v_0)$. For a given $l \in \R_+$, and for any non-negative super-solution $f$ of \eqref{eq:1.1} in $[t_1, t_2] \times\Omega_x \times \R^d$ with a non-negative source term $h$, such that $f(t, x, v) \leq C_q \langle v \rangle^{-q}$ for almost every $(t, x, v) \in [t_1, t_2] \times \Omega_x \times \Omega_v$ for $q > d$; 
and such that the non-negative kernel satisfies \eqref{eq:upperbound}-\eqref{eq:symmetry-nondiv}, we can find for any $\zeta \in \left(0, 1\right)$, a constant $C > 0$ depending on $\zeta, s, d, \Lambda_0, C_q$ such that, 
in case that $l = 0$, 
\bal\label{eq:L1-tail-l0}
	\int_{Q_{\frac{R}{2}}}&\int_{\R^d \setminus B_R}(f-l)_+(w) K(v, w)\chi_{f(v) > l}  \dd w \dd z\\
	&\leq C R^{n(1-\zeta)-2s}\left(\sup_{Q_R} (f-l)_+\right)^{1-\zeta} \left(\int_{Q_R} (f-l)_+ \dd z \right)^{\zeta};
\eal
and in case that $l > 0$,
%
%
\bal\label{eq:L1-tail}
	\int_{Q_{\frac{R}{2}}}\int_{\R^d \setminus B_R}& (f-l)_+(w) \chi_{f > l}(v) K(v, w) \dd w\dd v \\
	&\leq  CR_0^{n(1-\zeta)-2s} \mathcal M_l \abs{\big\{ f > l\big\}  \cap Q_{R_0}}^\zeta,
\eal
where 
\bal\label{M-l}
	&\mathcal M_l \\
	&:=  \left[1 +\left(\frac{\sup_{Q_{R_0}} (f-l)_+}{l}\right)^{\zeta}+ \left(\frac{\sup_{Q_{R_0}} (f-l)_+}{l}\right)^{1-\zeta}\right] \left(l + \sup_{Q_{R_0}}(f-l)_+\right).
\eal
and where $R_0 = 2\max\left\{R, \left(\frac{C_q}{l}\right)^{\frac{1}{q}}\right\}$ with $l > 0$. 
\end{proposition}
\begin{remark}\label{rmk:Lzeta-tail}
\begin{enumerate}[i.]
\item For $l = 0$ we could relate the local $L^\zeta$ norm on the right hand side to the local infimum using the Weak Harnack inequality of Theorem \ref{thm:weakH}, if we additionally assume \eqref{eq:coercivity}. Moreover, for $l =0$, it would suffice to suppose \eqref{eq:1.1} to hold on a bounded subset of $\R^d$.
\end{enumerate}
\end{remark}
For the proof, we collect preliminary estimates in the following lemma.
\begin{lemma}\label{lem:prelim}
Let $R > 0$ and $v_0 \in \R^d$. Let $\eta \in C_c^\infty(\R^d)$ be some smooth function with support in $B_R(v_0)$. Let $f: \R^d \to \R$ be non-negative, $l \geq 0$, $f_{l,\varepsilon} = (f -l)_+ +\varepsilon$ for some $\varepsilon > 0$ and $\zeta \in (0, 1)$. 
Then there holds:
\begin{enumerate}[i.] 
\item First, we have for any $v, w$ outside $f(v) \leq l \cap f(w) \leq l$
\bal\label{eq:3.1}
	\big[f(v) - f(w)\big] &\Big[f_{l,\varepsilon}^{-(1-\zeta)}(v)-  f_{l,\varepsilon}^{-(1-\zeta)}(w)\Big] \\
	&\leq - \frac{4(1-\zeta)}{\zeta^2}\frac{\big[f(v) - f(w)\big] }{\big[f_{l,\varepsilon}(v)- f_{l,\varepsilon}(w)\big] } \Big(f^{\frac{\zeta}{2}}_{l, \varepsilon}(v) - f^{\frac{\zeta}{2}}_{l,\varepsilon}(w)\Big)^2.
\eal
\item Second, for any $v, w$ outside $f(v) \leq l \cap f(w) \leq l$
\bal\label{eq:3.2}
	&\min\Big\{f_{l, \varepsilon}^{-(1-\zeta)}(v), f_{l, \varepsilon}^{-(1-\zeta)}(w)\Big\} \abs{f(v) - f(w)} \\
	&\qquad \qquad\leq \frac{2}{\zeta}\frac{\big[f(v) - f(w)\big] }{\big[f_{l,\varepsilon}(v)- f_{l,\varepsilon}(w)\big] } \max\left\{f_{l, \varepsilon}^{\frac{\zeta}{2}}(v), f_{l, \varepsilon}^{\frac{\zeta}{2}}(w)\right\}\abs{f_{l, \varepsilon}^{\frac{\zeta}{2}}(v)- f_{l, \varepsilon}^{\frac{\zeta}{2}}(w)}.
\eal
\item Third, 
\bal\label{eq:3.4}
	\frac{1}{2} \left\vert\left(\eta f_{l, \varepsilon}^{\frac{\zeta}{2}}\right)(v) -\left(\eta  f_{l, \varepsilon}^{\frac{\zeta}{2}}\right)(w)\right\vert^2 &- \big(\eta(v) - \eta(w)\big)^2  f_{l, \varepsilon}^\zeta(w) \\
	&\leq \min\big\{\eta^2(v), \eta^2(w)\big\} \abs{ f_{l, \varepsilon}^{\frac{\zeta}{2}}(v) -  f_{l, \varepsilon}^{\frac{\zeta}{2}}(w)}^2.
\eal
and 
\bal\label{eq:3.5}
	&\max\big\{\eta^2(v), \eta^2(w)\big\} \abs{  f_{l, \varepsilon}^{\frac{\zeta}{2}}(v) -  f_{l, \varepsilon}^{\frac{\zeta}{2}}(w) }^2 \\
	&\quad\leq 2  \abs{ \left(\eta  f_{l, \varepsilon}^{\frac{\zeta}{2}}\right)(v) - \left(\eta  f_{l, \varepsilon}^{\frac{\zeta}{2}}\right)(w) }^2 + 2\big(\eta(v) - \eta(w)\big)^2  f_{l, \varepsilon}^\zeta(v).
\eal
\end{enumerate}
\end{lemma}
\begin{proof}[Proof of Lemma \ref{lem:prelim}]
\begin{enumerate}[i.]
	\item We define for $\xi > 0$
		\beqs
			F(\xi) := \xi^{-(1-\zeta)},
		\eeqs
		so that 
		\beqs
			F'(\xi) = -(1-\zeta) \xi^{-(2-\zeta)} \leq 0.
		\eeqs
		Then we compute using Jensen's inequality for all $v, w$ outside $f(v) \leq l \cap f(w) \leq l$
		\bals
			\big[f(v) - f(w)\big] &\Big[f_{l,\varepsilon}^{-(1-\zeta)}(v)-   f_{l,\varepsilon}^{-(1-\zeta)}(w)\Big]\\
			 &= \big[f(v) - f(w)\big] \Big[ F\big(f_{l,\varepsilon}(v)\big)-   F\big(f_{l,\varepsilon}(w)\big)\Big]\\
			&= \big[f(v) - f(w)\big]\int_{f_{l,\varepsilon}(w)}^{f_{l,\varepsilon}(v)} F'(\xi) \dd \xi\\
			&= -(1-\zeta)\big[f(v) - f(w)\big] \int^{f_{l,\varepsilon}(v)}_{f_{l,\varepsilon}(w)} \xi^{-(2-\zeta)} \dd \xi\\
			&\leq - (1-\zeta)\frac{\big[f(v) - f(w)\big] }{\big[f_{l,\varepsilon}(v)- f_{l,\varepsilon}(w)\big] }\Bigg(\int^{f_{l,\varepsilon}(v)}_{f_{l,\varepsilon}(w)} \xi^{-\frac{(2-\zeta)}{2}} \dd \xi\Bigg)^{2}\\
			&= - \frac{4(1-\zeta)}{\zeta^2}\frac{\big[f(v) - f(w)\big] }{\big[f_{l,\varepsilon}(v)- f_{l,\varepsilon}(w)\big] } \Big(f^{\frac{\zeta}{2}}_{l, \varepsilon}(v) - f^{\frac{\zeta}{2}}_{l,\varepsilon}(w)\Big)^2.
		\eals
		This proves \eqref{eq:3.1}.
	\item To prove \eqref{eq:3.2} we denote by 
	\beqs
		\mathcal F'(\xi) := \sqrt{-F'(\xi)} = (1-\zeta)^{\frac{1}{2}} \xi^{-\frac{(2-\zeta)}{2}},
	\eeqs
	so that 
	\beqs
		\mathcal F(\xi) = \frac{2(1-\zeta)^{\frac{1}{2}}}{\zeta} \xi^{\frac{\zeta}{2}}.
	\eeqs
	Then, for all $v, w$ outside $f(v) \leq l \cap f(w) \leq l$, we note by Cauchy's mean value theorem
	\bal\label{eq:aux-3.2}
		\frac{\zeta}{2(1-\zeta)^{\frac{1}{2}}}&\frac{\abs{f(v) - f(w)}}{\Big\vert f_{l, \varepsilon}^{\frac{\zeta}{2}}(v) - f_{l, \varepsilon}^{\frac{\zeta}{2}}(w)\Big\vert} = \frac{\abs{f(v) - f(w)}}{\abs{\mathcal F\big(f_{l, \varepsilon}(v)\big) - \mathcal F\big(f_{l, \varepsilon}(w)\big)}} \\
		&\leq\frac{\abs{f(v) - f(w)}}{\abs{f_{l, \varepsilon}(v) - f_{l, \varepsilon}(w)}} \max\Bigg\{\frac{1}{\mathcal F'(f_{l, \varepsilon}(v))}, \frac{1}{\mathcal F'(f_{l,\varepsilon}(w))}\Bigg\}.
	\eal
	Upon multiplying \eqref{eq:aux-3.2} by $\min\big\{f_{l, \varepsilon}^{-(1-\zeta)}(v), f_{l, \varepsilon}^{-(1-\zeta)}(w)\big\}$ we deduce
	\bals
		\frac{\zeta}{2}&\min\Big\{f_{l, \varepsilon}^{-(1-\zeta)}(v), f_{l, \varepsilon}^{-(1-\zeta)}(w)\Big\}\frac{\abs{f(v) - f(w)}}{\Big\vert f_{l, \varepsilon}^{\frac{\zeta}{2}}(v) - f_{l, \varepsilon}^{\frac{\zeta}{2}}(w)\Big\vert}  \\
		&\leq \frac{\abs{f(v) - f(w)}}{\abs{f_{l, \varepsilon}(v) - f_{l, \varepsilon}(w)}} \max\Bigg\{f_{l, \varepsilon}^{\frac{(2-\zeta)}{2}}(v),f_{l, \varepsilon}^{\frac{(2-\zeta)}{2}}(w)\Bigg\}\min\Big\{f_{l, \varepsilon}^{-(1-\zeta)}(v), f_{l, \varepsilon}^{-(1-\zeta)}(w)\Big\}\\
		&\leq\frac{\abs{f(v) - f(w)}}{\abs{f_{l, \varepsilon}(v) - f_{l, \varepsilon}(w)}}  \max \Bigg\{f_{l, \varepsilon}^{\frac{\zeta}{2}}(v),f_{l, \varepsilon}^{\frac{\zeta}{2}}(w)\Bigg\},
	\eals
	or equivalently \eqref{eq:3.2}.
	\item To prove \eqref{eq:3.4} we assume with no loss in generality that $\eta(v) \leq \eta(w)$. Then we simply write
	\bals
		\eta(v)  f_{l, \varepsilon}^{\frac{\zeta}{2}}(v) - \eta(w)  f_{l, \varepsilon}^{\frac{\zeta}{2}}(w) = \eta(v)\Big(  f_{l, \varepsilon}^{\frac{\zeta}{2}}(v) -   f_{l, \varepsilon}^{\frac{\zeta}{2}}(w)\Big) +   f_{l, \varepsilon}^{\frac{\zeta}{2}}(w) \Big(\eta(v) - \eta(w)\Big).
	\eals
	This implies
	\bals
		\Big\vert \eta(v)  f_{l, \varepsilon}^{\frac{\zeta}{2}}(v) - \eta(w)  f_{l, \varepsilon}^{\frac{\zeta}{2}}(w) \Big\vert^2 \leq2\eta^2(v)\Big\vert  f_{l, \varepsilon}^{\frac{\zeta}{2}}(v) -   f_{l, \varepsilon}^{\frac{\zeta}{2}}(w)\Big\vert^2+ 2 f_{l, \varepsilon}^{\zeta}(w) \Big(\eta(v) - \eta(w)\Big)^2.
	\eals
	Similarly, to prove \eqref{eq:3.5} we note
	\beqs
		\eta(w)\Big( f_{l,\varepsilon}^{\frac{\zeta}{2}}(v)-   f_{l, \varepsilon}^{\frac{\zeta}{2}}(w)\Big) = \eta(v)  f_{l, \varepsilon}^{\frac{\zeta}{2}}(v) -  \eta(w)  f_{l, \varepsilon}^{\frac{\zeta}{2}}(w) + \Big(\eta(w) - \eta(v)\Big) f_{l, \varepsilon}^{\frac{\zeta}{2}}(v),
	\eeqs
	so that
	\beqs
		\eta^2(w)\Big\vert  f_{l, \varepsilon}^{\frac{\zeta}{2}}(v)-   f_{l, \varepsilon}^{\frac{\zeta}{2}}(w)\Big\vert^2 \leq 2  \Big\vert\eta(v)  f_{l, \varepsilon}^{\frac{\zeta}{2}}(v) -  \eta(w)  f_{l, \varepsilon}^{\frac{\zeta}{2}}(w)\Big\vert^2 +2 \Big(\eta(w) - \eta(v)\Big)^2 f_{l, \varepsilon}^{\zeta}(v).
	\eeqs
\end{enumerate}\qedhere
\end{proof}
\begin{proof}[Proof of Proposition \ref{prop:tail-bound}]
Our aim is to get a local bound on the tail on the upper level set. To this end we do a concavity estimate on the upper level sets by testing the equation with the level set function of the solution to some inverse power. This creates \textit{cross} terms. These are dealt with by exploiting the polynomial decay of the super-solution. 

With no loss in generality, we may assume $h = 0$. Indeed, if $h \geq 0$, then $f$ is a super-solution to \eqref{eq:1.1} with $h = 0$ as well. Moreover, due to the translation invariance, we may, with no loss in generality, consider $z_0 = (0, 0, 0)$. 


As discussed in Remark \ref{rmk:indicator-testing}, the justification of the computations below in case that $s \geq \frac{1}{2}$ relies on a comparison of the solution for \eqref{eq:1.1} to a solution of the fractional Kolmogorov equation, and quantifying the loss of regularity that the roughness of the coefficient introduces. The solution can be represented through the fundamental solution of the fractional Kolmogorov equation for an appropriate source term. We show that we can work with a solution that enjoys sufficient differentiability in order for the test function constructed below to be admissible. In the limit, where the smoothening of the fractional Laplacian is lost, this higher regularity vanishes as well. But only once the statement of Proposition \ref{prop:tail-bound} is proven for more regular solutions, we pass to the limit, which solves \eqref{eq:1.1}. Details are carried out in the Appendix \ref{sec:approx}. For $s < \frac{1}{2}$ this argument is obsolete.

For $l \in \R_+$ we consider $\psi_l \geq 0$ given by
\beq\label{eq:psi-l}
	\psi_l(t, x, v)  := \eta^2(t, x, v) \chi_{f > l}(t, x, v) + \chi_{f < l}(t, x, v)\xi(v),
\eeq
where $\eta \in C_c^\infty(\R^{2d+1})$ with $0 \leq \eta \leq 1$ such that $\eta = 1$ in $Q_{\frac{3R}{4}}$ and $\eta = 0$ outside $Q_{\frac{7R}{8}}$; moreover, $\xi$ is a cutoff in $v$ such that $\xi = 1$ in $B_r$ and $\xi = 0$ outside $B_{\bar R}$ with $r, \bar R$ \textit{to be determined.}

With no loss in generality, we assume there is at least one point $z_0$ in $Q_{R}$ where $f(z_0) > l$. Otherwise the statement \eqref{eq:L1-tail} is trivially satisfied.

\textit{Step 1: Non-local operator.}

We first consider the non-local operator $\mathcal E$ for fixed $(t, x)$.
Due to the divergence structure of $K$ \eqref{eq:symmetry}, we find
\bals
	\mathcal E\left(f, \psi f_{l, \varepsilon}^{-(1-\zeta)}\right)&= \int_{\R^d}\int_{\R^d}\big[f(v) - f(w)\big] f_{l,\varepsilon}^{-(1-\zeta)}(v) \psi_l(v) K(v, w) \dd w \dd v\\
	&= \frac{1}{4}\int_{\R^d}\int_{\R^d} \big[f(v) - f(w)\big] \Big[\big(f_{l,\varepsilon}^{-(1-\zeta)}\psi_l\big)(v)-  \big(f_{l,\varepsilon}^{-(1-\zeta)}\psi_l\big)(w)\Big]\\
	&\qquad\qquad\qquad \times \big[K(v, w)+K(w, v)\big] \dd w \dd v.
\eals
We then split $\mathcal E$ into three parts: we distinguish when $f(v), f(w) > l$, when $f(v), f(w) < l$ and when either $f(v) > l > f(w)$ or $f(w) > l >f(v)$. We write 
\bals
	&\mathcal E\left(f, \psi f_{l, \varepsilon}^{-(1-\zeta)}\right)\\
	&= \int_{\R^d}\int_{\R^d}\big[f(v) - f(w)\big] f_{l,\varepsilon}^{-(1-\zeta)}(v) \psi_l(v)K(v, w) \dd w \dd v\\
	&= \int_{\R^d}\int_{\R^d}\big[f(v) - f(w)\big]  f_{l,\varepsilon}^{-(1-\zeta)}(v)\psi_l(v) K(v, w) \chi_{f(v) > l}  \chi_{f(w) > l}  \dd w \dd v\\
	&\quad +\int_{\R^d}\int_{\R^d}\big[f(v) - f(w)\big]  f_{l,\varepsilon}^{-(1-\zeta)}(v) \psi_l(v)K(v, w) \chi_{f(v) < l}  \chi_{f(w) < l}  \dd w \dd v\\
	&\quad +\int_{\R^d}\int_{\R^d}\big[f(v) - f(w)\big]  f_{l,\varepsilon}^{-(1-\zeta)}(v) \psi_l(v) K(v, w) \big[\chi_{f(v) > l > f(w)}+  \chi_{f(w) > l > f(v)}\big]  \dd w \dd v\\
	&=: \mathcal E_{up} + \mathcal E_{low} + \mathcal E_{cross}. 
\eals
The indicator functions that we added respect the symmetry of each integrand.

\textit{Step 1-i: Non-locality on the lower level set.}

We find writing $K^{sym}(v, w):= K(v, w)+K(w, v)$
\begin{align*}
	\mathcal E_{low} &=\frac{1}{4}\int_{\R^d}\int_{\R^d} \big[f(v) - f(w)\big] \Big[\big( f_{l,\varepsilon}^{-(1-\zeta)}\psi_l\big)(v)-  \big( f_{l,\varepsilon}^{-(1-\zeta)}\psi_l\big)(w)\Big]\\
	&\qquad \qquad\qquad \times K^{sym}(v, w)\chi_{f(v) < l}  \chi_{f(w) < l} \dd w \dd v\\
	&= \frac{\varepsilon^{-(1-\zeta)}}{4}\int_{\R^d}\int_{\R^d} \big[f(v) - f(w)\big] \Big[\xi(v)-  \xi(w)\Big]K^{sym}(v, w)\chi_{f(v) < l}  \chi_{f(w) < l} \dd w \dd v \\
	&= \frac{\varepsilon^{-(1-\zeta)}}{2}\int_{\R^d}\int_{\R^d} f(v)\Big[\xi(v)-  \xi(w)\Big]K^{sym}(v, w)\chi_{f(v) < l}  \chi_{f(w) < l} \dd w \dd v\\
	&\leq \frac{\varepsilon^{-(1-\zeta)}}{2}\int_{B_{\bar R}}\int_{\R^d} f(v)\Big[\xi(v)-  \xi(w)\Big]K^{sym}(v, w)\chi_{f(v) < l}  \chi_{f(w) < l} \dd w \dd v,
\end{align*}
where the last inequality uses the definition of $\xi$. 

Due to \eqref{eq:upperbound}, we have for $\abs{v-w} > \bar R -r$ 
\bals
	\int_{B_{\bar R}}&\int_{\abs{v-w} > \bar R -r} f(v)\Big[\xi(v)-  \xi(w)\Big]\big[K(v, w)+K(w, v)\big]\chi_{f(v) < l}  \chi_{f(w) < l} \dd w \dd v\\
	&\leq \int_{B_{\bar R}}\int_{\abs{v-w} > \bar R -r} f(v)\xi(v)\big[K(v, w)+K(w, v)\big]\chi_{f(v) < l}  \chi_{f(w) < l} \dd w \dd v\\
	&\leq C (\bar R -r)^{-2s} \int_{B_{\bar R}} f(v) \dd v. 
\eals
For $\abs{v-w} < \bar R-r$ we Taylor-expand $\xi$, in case that $s < \frac{1}{2}$ up to first order, in case that $s \geq \frac{1}{2}$ up to second order. We start with $s < \frac{1}{2}$
\bals
\int_{B_{\bar R}}&\int_{\abs{v-w} < \bar R -r} f(v)\Big[\xi(v)-  \xi(w)\Big]\big[K(v, w)+K(w, v)\big]\chi_{f(v) < l}  \chi_{f(w) < l} \dd w \dd v\\ 
&\leq \norm{D \xi}_{L^\infty} \int_{B_{\bar R}} f(v) \int_{\abs{v-w} < \bar R -r} \abs{v-w} \big[K(v, w)+K(w, v)\big] \dd w \dd v\\
&\leq C (\bar R -r)^{1-2s} (\bar R -r)^{-1} \int_{B_{\bar R}} f(v) \dd v, 
\eals
using that $\norm{D \xi}_{L^\infty} \sim  (\bar R -r)^{-1}$, and \eqref{eq:upperbound}. 

For $s \geq \frac{1}{2}$, we need to do a further distinction. We abbreviate $K^{sym}(v, w) = K(v, w)+K(w, v)$, so that
\bals
	&\int_{B_{\bar R}}\int_{\abs{v-w} < \bar R -r} f(v)\Big[\xi(v)-  \xi(w)\Big]K^{sym}(v, w)\chi_{f(v) < l}  \chi_{f(w) < l} \dd w \dd v\\ 
	&=\int_{B_{\bar R}}\int_{\abs{z} < \bar R -r} f(v)\Big[\xi(v)-  \xi(v+z)\Big]K^{sym}(v, v+z)\chi_{f(v) < l}  \chi_{f(v+z) < l} \dd z \dd v\\ 
	&= \int_{B_{\bar R}}\int_{\abs{z} < \bar R -r} f(v)\Big[\xi(v)-  \xi(v+z)\Big]K^{sym}(v, v+z)\chi_{f(v) < l}  \chi_{f(v+z) < l}\chi_{f(v-z) <l}  \dd z \dd v \\
	&\quad +\int_{B_{\bar R}}\int_{\abs{z} < \bar R -r} f(v)\Big[\xi(v)-  \xi(v+z)\Big]K^{sym}(v, v+z)\chi_{f(v) < l}  \chi_{f(v+z) < l}\chi_{f(v-z) >l}  \dd z \dd v.
\eals

We Taylor expand the first term, and use the symmetries \eqref{eq:symmetry} and \eqref{eq:symmetry-nondiv}, due to which $K^{sym}(v, v+z) = K^{sym}(v, v-z)$,
\bals
	\int_{B_{\bar R}}&\int_{\abs{z} < \bar R -r}f(v)  \Big[\xi(v)-  \xi(v+z)\Big]K^{sym}(v, v+z)\chi_{f(v) < l}  \chi_{f(v+z) < l}\chi_{f(v-z) <l}  \dd z \dd v \\
	&=\int_{B_{\bar R}}f(v)D\xi(v)\chi_{f(v) < l}\int_{\abs{z} < \bar R -r}   (-z)K^{sym}(v, v+z)\chi_{f(v+z) < l}\chi_{f(v-z) <l}  \dd z \dd v \\
	&\quad + \norm{D^2 \xi}_{L^\infty}\int_{B_{\bar R}}f(v)\chi_{f(v) < l}\int_{\abs{z} < \bar R -r}   \abs{z}^2K^{sym}(v, v+z)  \chi_{f(v+z) < l}\chi_{f(v-z) <l}  \dd z \dd v \\
	&\leq C (\bar R - r)^{-2} (\bar R - r)^{2-2s}\int_{B_{\bar R}}f(v) \dd v. 
\eals
For the second term, we use the fact that on the set $\{f(v-z) > l\}$ we have $l < f(v-z) \leq C_q \langle v-z\rangle^{-q}$, so that $\abs{ v-z} \leq\left( \frac{C_q}{l}\right)^{\frac{1}{q}}$. Keeping this in mind, we find
\begin{align*}
	&\int_{B_{\bar R}}\int_{\abs{z} < \bar R -r} f(v)\Big[\xi(v)-  \xi(v+z)\Big]K^{sym}(v, v+z)\chi_{f(v) < l}  \chi_{f(v+z) < l}\chi_{f(v-z) >l}  \dd z \dd v\\
	&= \int_{B_{\bar R}}\int_{\abs{z} < \bar R -r}f(v)  \big[\xi(v) - 1\big]K^{sym}(v, v+z)\chi_{f(v) < l}  \chi_{f(v+z) < l}\chi_{f(v-z) >l}  \dd z \dd v\\
	&\quad + \int_{B_{\bar R}}\int_{\abs{z} < \bar R -r}f(v)  \big[1-\xi(v+z) \big]K^{sym}(v, v+z)\chi_{f(v) < l}  \chi_{f(v+z) < l}\chi_{f(v-z) >l}  \dd z \dd v\\
	\intertext{using that $\xi \leq 1$}
	&\leq  \int_{B_{\bar R}}\int_{\abs{z} < \bar R -r}f(v)  \big[1-\xi(v+z) \big]K^{sym}(v, v+z)\chi_{f(v) < l}  \chi_{f(v+z) < l}\chi_{f(v-z) >l}  \dd z \dd v.
\end{align*}
But $1-\xi(v+z)= 0$ for $\abs{v+z} < r$. Thus $\abs{v+z} \geq r$ and if we pick $r \geq 2\left( \frac{C_q}{l}\right)^{\frac{1}{q}}$, then $\abs{v-z} < \frac{r}{2}$, so that
\[
	\abs{z} \geq \frac{1}{2}\left(\abs{v+z} - \abs{v-z}\right) \geq \frac{1}{2}\left(r - \frac{r}{2}\right) = \frac{r}{4}. 
\]
Therefore, we find
\begin{align*}
&\int_{B_{\bar R}}\int_{\abs{z} < \bar R -r}f(v)  \big[1-\xi(v+z) \big]K^{sym}(v, v+z)\chi_{f(v) < l}  \chi_{f(v+z) < l}\chi_{f(v-z) >l}  \dd z \dd v \\
&= \int_{B_{\bar R}}\int_{\frac{r}{4} < \abs{z} < \bar R -r}f(v)  \big[1-\xi(v+z) \big]K^{sym}(v, v+z)\chi_{f(v) < l}  \chi_{f(v+z) < l}\chi_{f(v-z) >l}  \dd z \dd v \\
\intertext{using $w = v-z$, Fubini, bounding $1-\xi \leq 1$, and using $\abs{v-w} = \abs{2v - w -v} = \abs{v-(2v-w)}$}
&\leq \int_{\abs{w} <\frac{r}{2}}\int_{\abs{v} \leq \bar R, \frac{r}{4} < \abs{v-w} < \bar R-r}f(v)  K^{sym}(v, 2v-w)\chi_{f(v) < l}  \chi_{f(w) >l}  \dd v \dd w \\
&\leq Cl r^{-2s} \int_{\abs{w} <\frac{r}{2}} \chi_{f(w) >l} \dd w. 
\end{align*}

We conclude 
\bal\label{eq:E-neg}
\mathcal E_{low} \leq C \varepsilon^{-(1-\zeta)}(\bar R - r)^{-2s}\int_{B_{\bar R}} f(v) \dd v + C\varepsilon^{-(1-\zeta)} l r^{-2s}  \abs{\{f>l\}  \cap B_{\frac{r}{2}}}.
\eal

\textit{Step 1-ii: Cross non-locality.}

For the cross term, we note that $l > 0$. If $l = 0$, then $\chi_{f(v) > l > f(w)}+  \chi_{f(w) > l > f(v)} = 0$ due to the non-negativity of $f$. We find
\begin{align*}
	\mathcal E_{cross} &= \frac{1}{4}\int_{\R^d}\int_{\R^d} \big[f(v) - f(w)\big] \Big[\big( f_{l,\varepsilon}^{-(1-\zeta)}\psi_l\big)(v)-  \big( f_{l,\varepsilon}^{-(1-\zeta)}\psi_l\big)(w)\Big]\\
	&\qquad \qquad\qquad \times \big[K(v, w)+K(w, v)\big]\big[\chi_{f(v) > l > f(w)}+  \chi_{f(w) > l > f(v)}\big] \dd w \dd v\\
	&=\frac{1}{2}\int_{\R^d}\int_{\R^d} \big[f(v) - f(w)\big] \Big[\big( f_{l,\varepsilon}^{-(1-\zeta)}\psi_l\big)(v)-  \big( f_{l,\varepsilon}^{-(1-\zeta)}\psi_l\big)(w)\Big]\\
	&\qquad \qquad\qquad \times \big[K(v, w)+K(w, v)\big]\chi_{f(v) > l > f(w)}\dd w \dd v\\
	&=\frac{1}{2}\int_{\R^d}\int_{\R^d} \big[f(v) - f(w)\big] \Big[\eta^2(v) (f-l+\varepsilon)^{-(1-\zeta)}(v)-  \varepsilon^{-(1-\zeta)}\xi(w)\Big] \\
	&\qquad \qquad\qquad \times\big[K(v, w)+K(w, v)\big]\chi_{f(v) > l > f(w)}\dd w \dd v\\
	&=\frac{1}{2}\int_{\R^d}\int_{\R^d} \big[f(v) - f(w)\big] \Big[\eta^2(v) (f-l+\varepsilon)^{-(1-\zeta)}(v)-  \varepsilon^{-(1-\zeta)}\Big] \\
	&\qquad \qquad\qquad \times\big[K(v, w)+K(w, v)\big]\chi_{f(v) > l > f(w)}\dd w \dd v\\
	&\quad+ \frac{\varepsilon^{-(1-\zeta)}}{2}\int_{\R^d}\int_{\R^d} \big[f(v) - f(w)\big] \big[1-\xi(w)  \big] \\
	&\qquad \qquad\qquad \times\big[K(v, w)+K(w, v)\big]\chi_{f(v) > l > f(w)}\dd w \dd v. 
\end{align*}
Due to the concavity of the test function, the first part of the cross term has a sign. For the second part we use the fact that $f(v) \leq C_q \langle v \rangle^{-q}$ for any $q > 0$, so that on $\{f(v) > l\}$ we have $\abs{v} < \langle v \rangle < 
\left(\frac{C_q}{l}\right)^{\frac{1}{q}}$. At the same time, $(1-\xi)(w) =0$ for $\abs{w} \leq r$, so if we pick $r \geq 2\left(\frac{C_q}{l}\right)^{\frac{1}{q}}$ then $\abs{v-w} > \frac{r}{2}$. We therefore find
\bal\label{eq:E-cross}
	\mathcal E_{cross}&\leq \frac{\varepsilon^{-(1-\zeta)}}{2}\int_{\R^d}\int_{\abs{v-w} > \frac{r}{2}} \big[f(v) - f(w)\big] \big[1-\xi(w)  \big] \\
	&\qquad \qquad\qquad \times\big[K(v, w)+K(w, v)\big]\chi_{f(v) > l > f(w)}\dd w \dd v\\
	&\leq \frac{\varepsilon^{-(1-\zeta)}}{2}\int_{B_{\frac{r}{2}}}\int_{\abs{v-w} > \frac{r}{2}} \big[(f - l)_+(v)+l\big]\\
	&\qquad \qquad\qquad \times\big[K(v, w)+K(w, v)\big]\chi_{f(v) > l > f(w)}\dd w \dd v \\
	&\leq \frac{\varepsilon^{-(1-\zeta)}}{2}\left[\left(\frac{r}{2}\right)^{-2s} \int_{B_{\frac{r}{2}}} (f - l)_+(v) \dd v + l \abs{\{ f > l \} \cap B_{\frac{r}{2}}}\right].
\eal

\textit{Step 1-iii: Non-locality on the upper level set.}

It remains to bound $\mathcal E_{up}$. 
For fixed time and space, we find 
\bal\label{eq:E-pos}
	\mathcal E_{up} &= \int_{\R^d}\int_{\R^d}\big[f(v) - f(w)\big]  f_{l,\varepsilon}^{-(1-\zeta)}(v) \eta^2(v)K(v, w)\chi_{f(v) > l} \chi_{f(w)> l} \dd w \dd v\\
	&=\int_{B_R}\int_{B_R}\big[f(v) - f(w)\big]  f_{l,\varepsilon}^{-(1-\zeta)}(v)\eta^2(v) K(v, w)\chi_{f(v) > l} \chi_{f(w)> l} \dd w \dd v\\
	&\quad + \int_{B_R}\int_{\R^d \setminus B_R}\big[f(v) - f(w)\big] f_{l,\varepsilon}^{-(1-\zeta)}(v)\eta^2(v) K(v, w)\chi_{f(v) > l} \chi_{f(w)> l} \dd w \dd v\\
	&=: \mathcal E_{up}^{loc} + \mathcal E_{up}^{tail}. 
\eal
For $\mathcal E^{tail}_{up}$, we note that there is no singularity, due to the choice of $\eta$, which vanishes outside $B_{\frac{7R}{8}}$. Therefore, by \eqref{eq:upperbound}
\bal\label{eq:E-pos-tail}
	\mathcal E_{up}^{tail} &= \int_{B_R}\int_{\R^d \setminus B_R}\big[f(v) - f(w)\big] f_{l,\varepsilon}^{-(1-\zeta)}(v) \eta^2(v)K(v, w)\chi_{f(v) > l} \chi_{f(w)> l} \dd w \dd v\\
	&= \int_{B_R}\int_{\R^d \setminus B_R} (f-l)_+(v) f_{l,\varepsilon}^{-(1-\zeta)}(v)\eta^2(v) K(v, w)\chi_{f(v) > l} \chi_{f(w)> l} \dd w \dd v\\
	&\quad - \int_{B_R}\int_{\R^d \setminus B_R}(f-l)_+(w)  f_{l,\varepsilon}^{-(1-\zeta)}(v)\eta^2(v) K(v, w)\chi_{f(v) > l} \chi_{f(w)> l} \dd w \dd v\\
	&\leq \Lambda_0 \left(\frac{R}{8}\right)^{-2s} \int_{B_R} ((f-l)_++\varepsilon)^{\zeta}(v) \chi_{f > l}(v) \dd v\\
	&\quad- \int_{B_R}\int_{\R^d \setminus B_R}(f-l)_+(w) f_{l,\varepsilon}^{-(1-\zeta)}(v)\eta^2(v) K(v, w)\chi_{f(v) > l} \dd w \dd v.
\eal
In particular, the tail of $(f-l)_+$ has a negative sign.

We now abbreviate $K^{sym}(v, w) := K(v, w) +K(w, v)$. Then we find for the local contribution of $\mathcal E_{up}$
\bal\label{eq:E-pos-loc-def}
	\mathcal E_{up}^{loc} &= \int_{B_R}\int_{B_R}\big[f(v) - f(w)\big] f_{l,\varepsilon}^{-(1-\zeta)}(v) \eta^2(v)K(v, w)\chi_{f(v) > l} \chi_{f(w)> l} \dd w \dd v\\
	&= \frac{1}{4}\int_{B_R}\int_{B_R}\big[f(v) - f(w)\big] \Big[ f_{l,\varepsilon}^{-(1-\zeta)}(v)\eta^2(v) -  f_{l,\varepsilon}^{-(1-\zeta)}(w)\eta^2(w)\Big]\\
	&\qquad \qquad\qquad \times K^{sym}(v, w) \chi_{f(v) > l} \chi_{f(w)> l} \dd w \dd v\\
	&= \frac{1}{4}\int_{B_R}\int_{B_R}\big[f(v) - f(w)\big] \Big[ f_{l,\varepsilon}^{-(1-\zeta)}(v)\eta^2(v) -  f_{l,\varepsilon}^{-(1-\zeta)}(w)\eta^2(w)\Big]\\
	&\qquad \qquad\qquad \times K^{sym}(v, w)\chi_{f(v) > l} \chi_{f(w)> l}\big[\chi_{f(v) > f(w)} + \chi_{f(w) > f(v)}\big] \dd w \dd v\\
	&= \frac{1}{2}\int_{B_R}\int_{B_R}\big[f(v) - f(w)\big] \Big[ f_{l,\varepsilon}^{-(1-\zeta)}(v)\eta^2(v) -  f_{l,\varepsilon}^{-(1-\zeta)}(w)\eta^2(w)\Big]\\
	&\qquad \qquad\qquad \times K^{sym}(v, w)\chi_{f(v) > l} \chi_{f(w)> l}\chi_{f(w) > f(v)} \dd w \dd v\\
	&= \frac{1}{2}\int_{B_R}\int_{B_R}\big[f(v) - f(w)\big] \Big[ f_{l,\varepsilon}^{-(1-\zeta)}(v) -  f_{l,\varepsilon}^{-(1-\zeta)}(w)\Big]\eta^2(v)\\
	&\qquad \qquad\qquad \times K^{sym}(v, w)\chi_{f(v) > l} \chi_{f(w)> l}\chi_{f(w) > f(v)} \dd w \dd v\\
	&\quad+ \frac{1}{2}\int_{B_R}\int_{B_R}\big[f(v) - f(w)\big]  f_{l,\varepsilon}^{-(1-\zeta)}(w)\big[\eta^2(v) - \eta^2(w) \big]\\
	&\qquad \qquad\qquad \times K^{sym}(v, w) \chi_{f(v) > l} \chi_{f(w)> l}\chi_{f(w) > f(v)} \dd w \dd v\\
	&=: I_{sign} + I_{sym}.
\eal
The term $I_{sign}$ has a sign due to the concavity of the test function, whereas in $I_{sym}$ the cutoff $\eta$ absorbs some of the singularity of the kernel $K$.
To quantitatively bound $I_{sign}$ and $I_{sym}$ we use the estimates of Lemma \ref{lem:prelim}. For $I_{sign}$ we use \eqref{eq:3.1}, \eqref{eq:3.4}, and \eqref{eq:upperbound}
\bal\label{eq:I-sign}
	I_{sign} &= \frac{1}{2}\int_{B_R}\int_{B_R}\big[f(v) - f(w)\big] \Big[ f_{l,\varepsilon}^{-(1-\zeta)}(v) -  f_{l,\varepsilon}^{-(1-\zeta)}(w)\Big]\eta^2(v)\\
	&\qquad \qquad \qquad\qquad \times K^{sym}(v, w)\chi_{f(w) > f(v) > l}  \dd w \dd v\\
	 	&\leq -\frac{2(1-\zeta)}{\zeta^2}\int_{B_R}\int_{B_R}\left( f_{l,\varepsilon}^{\frac{\zeta}{2}}(v) -  f_{l,\varepsilon}^{\frac{\zeta}{2}}(w)\right)^2\eta^2(v)\\
		&\qquad \qquad \qquad \qquad\times K^{sym}(v, w) \chi_{f(w) > f(v) > l}  \dd w \dd v\\
		&\leq \frac{2(1-\zeta)}{\zeta^2}\int_{B_R}\int_{B_R} f_{l,\varepsilon}^{\zeta}(w) \big(\eta(v)-\eta(w))^2 \\
		&\qquad \qquad \qquad \qquad\times K^{sym}(v, w)\chi_{f(w) > f(v) > l}  \dd w \dd v\\
		&\quad -\frac{1-\zeta}{\zeta^2}\int_{B_R}\int_{B_R}\left[ \left(f_{l,\varepsilon}^{\frac{\zeta}{2}}\eta\right)(v) -  \left(f_{l,\varepsilon}^{\frac{\zeta}{2}}\eta\right)(w)\right]^2\\
		&\qquad \qquad \qquad \qquad\times  K^{sym}(v, w) \chi_{f(w) > f(v) > l}  \dd w \dd v\\
		&\leq C(\zeta, \Lambda_0) R^{-2s} \int_{B_R} (f-l+\varepsilon)^{\zeta}(w)  \chi_{f(w) > l} \dd w\\
		&\quad -\frac{1-\zeta}{\zeta^2}\int_{B_R}\int_{B_R}\left[ \left(f_{l,\varepsilon}^{\frac{\zeta}{2}}\eta\right)(v) -  \left(f_{l,\varepsilon}^{\frac{\zeta}{2}}\eta\right)(w)\right]^2\\
		&\qquad \qquad \qquad \qquad\times K^{sym}(v, w)\chi_{f(w) > f(v) > l}  \dd w \dd v.
\eal
This was the good term. Concerning the unsigned term $I_{sym}$, we use \eqref{eq:3.2}, Young's inequality, \eqref{eq:3.5} and \eqref{eq:upperbound} so that for any $\delta_0 \in (0, 1)$
\bal\label{eq:I-sym}
	&I_{sym} =\frac{1}{2}\int_{B_R}\int_{B_R}\big[f(v) - f(w)\big]  f_{l,\varepsilon}^{-(1-\zeta)}(w)\big[\eta^2(v) - \eta^2(w) \big] \\
	&\qquad \qquad \qquad \qquad \times K^{sym}(v, w) \chi_{f(w) > f(v)> l} \dd w \dd v\\
	&\leq \frac{1}{\zeta}\int_{B_R}\int_{B_R}f_{l,\varepsilon}^{\frac{\zeta}{2}}(w) \left[f_{l,\varepsilon}^{\frac{\zeta}{2}}(w) - f_{l,\varepsilon}^{\frac{\zeta}{2}}(v)\right]  \abs{\eta^2(w) - \eta^2(v)} \\
	&\qquad \qquad \qquad \qquad \times K^{sym}(v, w) \chi_{f(w) > f(v)> l} \dd w \dd v\\
	&\leq C(\delta_0, \zeta) \int_{B_R}\int_{B_R}f_{l,\varepsilon}^{\zeta}(w) \abs{\eta(w) - \eta(v)}^2K^{sym}(v, w) \dd w \dd v \\
	&\quad +\frac{\delta_0}{2}\int_{B_R}\int_{B_R} \max\{\eta^2(v), \eta^2(w)\} \left[f_{l,\varepsilon}^{\frac{\zeta}{2}}(w) - f_{l,\varepsilon}^{\frac{\zeta}{2}}(v)\right]^2 \\
	&\qquad \qquad \qquad\qquad \times K^{sym}(v, w)  \chi_{f(w) > f(v)> l} \dd w \dd v\\
	&\leq C(\delta_0, \zeta) \int_{B_R}\int_{B_R}f_{l,\varepsilon}^{\zeta}(w) \abs{\eta(w) - \eta(v)}^2K(v, w)\chi_{f  >l}(w) \ \dd w \dd v \\
	&\quad + \delta_0 \int_{B_R}\int_{B_R}f_{l,\varepsilon}^{\zeta}(v) \abs{\eta(w) - \eta(v)}^2K(v, w)\chi_{f  >l}(v)  \dd w \dd v \\
	&\quad +\delta_0\int_{B_R}\int_{B_R} \left[\left(\eta f_{l,\varepsilon}^{\frac{\zeta}{2}}\right)(w) - \left(f_{l,\varepsilon}^{\frac{\zeta}{2}}\right)(v)\right]^2 K^{sym}(v, w)   \chi_{f(w) > f(v)> l} \dd w \dd v\\
	&\leq C(\delta_0, \zeta, \Lambda_0) R^{-2s} \int_{B_R}(f-l + \varepsilon)^{\zeta}(v) \chi_{f  >l}(v) \dd v \\
	&\quad +\delta_0\int_{B_R}\int_{B_R} \left[\left(f_{l,\varepsilon}^{\frac{\zeta}{2}}\eta\right)(w) - \left(f_{l,\varepsilon}^{\frac{\zeta}{2}}\eta\right)(v)\right]^2 K^{sym}(v, w)  \chi_{f(w) > f(v)> l} \dd w \dd v.
\eal
In particular, the last term can be absorbed with the signed term of $I_{sign}$ in \eqref{eq:I-sign} for $\delta_0$ sufficiently small and for any $\zeta < 1$, so that we conclude from \eqref{eq:E-pos-loc-def}, \eqref{eq:I-sign}, and \eqref{eq:I-sym}
\bal\label{eq:E-pos-loc}
	&\mathcal E_{up}^{loc} = I_{sign} + I_{sym}\\
	&\leq C(\delta_0, \zeta, \Lambda_0) R^{-2s} \int_{B_R}(f-l + \varepsilon)^{\zeta}(v) \chi_{f  >l}(v) \dd v\\
	&\quad -\frac{1-\zeta}{\zeta^2}\int_{B_R}\int_{B_R}\left[ \left(f_{l,\varepsilon}^{\frac{\zeta}{2}}\eta\right)(v) -  \left(f_{l,\varepsilon}^{\frac{\zeta}{2}}\eta\right)(w)\right]^2K^{sym}(v, w)\chi_{f(w) > f(v) > l}  \dd w \dd v\\
	&\quad +\delta_0\int_{B_R}\int_{B_R} \left[\left(f_{l,\varepsilon}^{\frac{\zeta}{2}}\eta\right)(w) - \left(f_{l,\varepsilon}^{\frac{\zeta}{2}}\eta\right)(v)\right]^2 K^{sym}(v, w) \chi_{f(w) > f(v)> l} \dd w \dd v\\
	&\leq  C(\delta_0, \zeta, \Lambda_0) R^{-2s} \int_{B_R}(f-l + \varepsilon)^{\zeta}(v)\chi_{f  >l}(v) \dd v.
\eal

\textit{Step 1-iv: Bound on the non-locality.}

We conclude from \eqref{eq:E-neg}, \eqref{eq:E-cross}, \eqref{eq:E-pos}, \eqref{eq:E-pos-tail}, and \eqref{eq:E-pos-loc}
\bal\label{eq:E}
	\mathcal E&\left(f, \psi f_{l, \varepsilon}^{-(1-\zeta)}\right) = \mathcal E_{up} + \mathcal E_{low} + \mathcal E_{cross} \\
	&\leq C(\zeta, \Lambda_0) R^{-2s} \int_{B_R}(f-l + \varepsilon)^{\zeta}(v)\chi_{f > l} \dd v\\
	&\quad- \int_{B_R}\int_{\R^d \setminus B_R}(f-l)_+(w) \eta^2(v) f_{l,\varepsilon}^{-(1-\zeta)}(v) K(v, w)\chi_{f(v) > l} \dd w \dd v \\
	&\quad + C \varepsilon^{-(1-\zeta)}(\bar R - r)^{-2s}\int_{B_{\bar R}} f(v) \dd v+  \frac{\varepsilon^{-(1-\zeta)}}{2}\left(\frac{r}{2}\right)^{-2s} \int_{B_{\frac{r}{2}}} (f - l)_+(v) \dd v\\
	&\quad + Cl r^{-2s}\varepsilon^{-(1-\zeta)} \abs{\{f > l \}\cap B_{\frac{r}{2}}}. 
\eal
Note that this estimate holds true for any $(t, x)$. 

\textit{Step 2: Transport.}

We start with the transport operator. We integrate over a torus, 
\[
	\T_{R_1, R_2}(y) := \left\{ \left(\sqrt{(x_1-y_1)^2 + (x_2-y_2)^2} - R_2\right)^2 + \sum_{i=3}^d (x_i-y_i)^2 = R_1^2\right\},
\]
with minor radius $R_1 := \frac{1}{2}R^{1+2s}$, major radius $R_2 := \frac{3}{4}R^{1+2s}$, centred at $y = (\frac{3}{4}R^{1+2s}, 0, \dots, 0)$. Then $\T_{R_1, R_2}(y) \subset B_{2R^{1+2s}}$, and we get
\bals
	\int_{-R^{2s}}^0 \int_{\T_{R_1, R_2}(y)}& \int_{\R^d} v \cdot\nabla_x f(z)\psi_l(z) f_{l, \varepsilon}^{-(1-\zeta)}(z)   \dd v \dd x \dd t  \\
	&=\int_{\R}\int_{\R^d}  \int_{\R^d}  v \cdot\nabla_x (f-l)_+(z)\eta^2(z) f_{l, \varepsilon}^{-(1-\zeta)}(z)   \dd v \dd x\dd t\\
	&\quad +\varepsilon^{-(1-\zeta)}\int_{-R^{2s}}^0\int_{\T_{R_1, R_2}(y)}  \int_{\R^d}  v \cdot\nabla_x (f-l)_-(z)   \xi(v)  \dd v \dd x\dd t\\
	&=\frac{1}{\zeta}\int_{-R^{2s}}^0 \int_{\R^d}  \int_{\R^d} v \cdot\nabla_x \big((f-l)_++\varepsilon\big)^{\zeta}(z)\eta^2(z)    \dd v \dd x\dd t.
\eals
The last equality holds due to the divergence theorem applied to the integration in the spatial variable $x$. 
Now we use $\abs{v \cdot\nabla_x \eta} \sim R^{-2s}$
\bals
	\frac{1}{\zeta}&\int_{-R^{2s}}^0 \int_{\R^d}  \int_{\R^d} v \cdot\nabla_x \big((f-l)_++\varepsilon\big)^{\zeta}(z)\eta^2(z)    \dd v \dd x\dd t\\
	&=-\frac{1}{\zeta}\int_{-R^{2s}}^0 \int_{\R^d}  \int_{\R^d}  \big((f-l)_++\varepsilon\big)^{\zeta}(z)v \cdot \nabla_x \eta^2(z)   \dd v \dd x\dd t\\
	&\leq CR^{-2s} \int_{Q_R}  \big((f-l)_++\varepsilon\big)^{\zeta}(z)\dd z.
\eals

We continue with the time derivative. We denote by $t_1 :=-R^{2s}$ and by $t_2 = 0$, so that
\bals
	\int_{-R^{2s}}^0& \int_{\T_{R_1, R_2}(y)} \int_{\R^d} \partial_t f(z)\psi_l(z) f_{l, \varepsilon}^{-(1-\zeta)}(z)   \dd v \dd x \dd t \\
	&=\int_{\R} \int_{\T_{R_1, R_2}(y)} \int_{\R^d} \partial_t (f-l)_+(z)\eta^2(z) f_{l, \varepsilon}^{-(1-\zeta)}(z)   \dd v \dd x \dd t  \\
	&\quad +\varepsilon^{-(1-\zeta)}\int_{t_1}^{t_2} \int_{\T_{R_1, R_2}(y)} \int_{\R^d} \partial_t (f-l)_-(z)  \xi(v)\dd v \dd x \dd t.
\eals
We integrate both terms by part. For the first, we get:
\bals
	\int_{\R}& \int_{\T_{R_1, R_2}(y)} \int_{\R^d} \partial_t (f-l)_+(z)\eta^2(z) f_{l, \varepsilon}^{-(1-\zeta)}(z)   \dd v \dd x \dd t \\
	&= \int_{\R} \int_{\T_{R_1, R_2}(y)} \int_{\R^d} \partial_t ((f-l)_++\varepsilon)^{\zeta}(z)\eta^2(z)    \dd v \dd x \dd t \\
	&= -\int_{\R} \int_{\T_{R_1, R_2}(y)} \int_{\R^d}  ((f-l)_++\varepsilon)^{\zeta}(z)\partial_t\eta^2(z)    \dd v \dd x \dd t \\
	&\leq C R^{-2s}\int_{Q_R} ((f-l)_++\varepsilon)^{\zeta}(z)\dd z. 
\eals

For the second term, we find:
\bal\label{eq:aux-temporal}
	\varepsilon^{-(1-\zeta)}&\int_{t_1}^{t_2} \int_{\T_{R_1, R_2}(y)} \int_{\R^d} \partial_t (f-l)_-(z)  \xi(v)\dd v \dd x \dd t\\
	&=\varepsilon^{-(1-\zeta)} \int_{\T_{R_1, R_2}(y)} \int_{\R^d} \big((f-l)_-(t_2) -  (f-l)_-(t_1)\big) \xi(v)\dd v \dd x\\
	&=\varepsilon^{-(1-\zeta)} \int_{\T_{R_1, R_2}(y)} \int_{\R^d} \big((f-l)(t_2) -  (f-l)(t_1)\big) \xi(v)\dd v \dd x\\
	&\quad - \varepsilon^{-(1-\zeta)} \int_{\T_{R_1, R_2}(y)} \int_{\R^d} \big((f-l)_+(t_2) -  (f-l)_+(t_1)\big) \xi(v)\dd v \dd x.
\eal
We now want to use the mass conservation and the decay for large velocities. We find
\begin{align*}
	 \int_{\T_{R_1, R_2}(y)} &\int_{\R^d} \big((f-l)(t_2) -  (f-l)(t_1)\big) \xi(v)\dd v \dd x\\
	 &= \int_{\T_{R_1, R_2}(y)} \int_{\R^d} \big(f(t_2) -  f(t_1)\big) \xi(v)\dd v \dd x\\
	 &=\int_{\T_{R_1, R_2}(y)} \int_{\R^d} \big(f(t_2) -  f(t_1)\big) \dd v \dd x\\
	 &\quad +\int_{\T_{R_1, R_2}(y)} \int_{\R^d} \big(f(t_2) -  f(t_1)\big) (\xi-1)(v)\dd v \dd x\\
	 &\leq \int_{\T_{R_1, R_2}(y)} \int_{\R^d} f(t_1) (1-\xi)(v) \dd v\dd x,
\end{align*}
where for last inequality we used the mass conservation and the fact that $\xi \leq 1$. 
Then we recall our definition of $\xi$, which implies that $1-\xi = 0$ in $B_r$. Thus
\bals
	&\int_{\T_{R_1, R_2}(y)} \int_{\R^d} f(t_1) (1-\xi)(v) \dd v\dd x \\
	&=  \int_{\T_{R_1, R_2}(y)} \int_{B_{\bar R} \setminus B_r} f(t_1) (1-\xi)(v) \dd v\dd x + \int_{B_{R^{1+2s}}} \int_{\R^d \setminus B_{\bar R}} f(t_1) (1-\xi)(v) \dd v\dd x\\
	&\leq C_q \int_{\T_{R_1, R_2}(y)} \int_{B_{\bar R} \setminus B_r}  \langle v\rangle^{-q}   (1-\xi)(v) \dd v\dd x + C_q \int_{\T_{R_1, R_2}(y)} \int_{\R^d \setminus B_{\bar R}} \langle v\rangle^{-q}  \dd v\dd x.
\eals
Since $q > d$, 
\[
	\int_{\T_{R_1, R_2}(y)} \int_{\R^d \setminus B_{\bar R}} \langle v\rangle^{-q}  \dd v\dd x \leq C R^{d(1+2s)} {\bar R}^{-q+d}. 
\]	
Moreover, we use that on $B_{\bar R} \setminus B_r$, the cutoff $(1-\xi)(v) \sim \frac{\abs{v}-r}{\bar R -r} $ and $\langle v \rangle \geq \abs{v}-r$. Thus for $d < q < d+1$, 
\bals
	 C_q& \int_{\T_{R_1, R_2}(y)} \int_{B_{\bar R} \setminus B_r}  \langle v\rangle^{-q}   (1-\xi)(v) \dd v\dd x \\
	 &\leq C \int_{\T_{R_1, R_2}(y)} \int_{B_{\bar R} \setminus B_r}  \frac{(\abs{v}-r)^{1-q}}{\bar R -r} \dd v\dd x \\
	 &\leq C \omega_d R^{d(1+2s)}\frac{1}{\bar R -r} \int_{0}^{\bar R - r}\rho^{-q+1+d-1}  \dd \rho \\
	 & \leq C R^{d(1+2s)}\frac{(\bar R-r)^{d-q+1}}{\bar R -r}   = C \frac{R^{d(1+2s)}}{(\bar R -r)^{q-d}}. 
\eals
We combine these estimates to find:
\bals
	 \int_{\T_{R_1, R_2}(y)} &\int_{\R^d} \big((f-l)(t_2) -  (f-l)(t_1)\big) \xi(v)\dd v \dd x \leq C \frac{R^{d(1+2s)}}{(\bar R -r)^{q-d}} + C R^{d(1+2s)} { \bar R}^{-q+d}. 
\eals
This deals with the first part of \eqref{eq:aux-temporal}. We are left with bounding
\begin{align*}
 - &\varepsilon^{-(1-\zeta)} \int_{\T_{R_1, R_2}(y)} \int_{\R^d} \big((f-l)_+(t_2) -  (f-l)_+(t_1)\big) \xi(v)\dd v \dd x \\
&\leq \varepsilon^{-(1-\zeta)} \int_{\T_{R_1, R_2}(y)} \int_{\R^d}  (f-l)_+(t_1) \xi(v)\dd v \dd x \\
\intertext{we use that on $\{f(v) > l\}$, we have $\abs{v} \leq \left(\frac{C_q}{l}\right)^{\frac{1}{q}} = \frac{r}{2}$}
&\leq \varepsilon^{-(1-\zeta)} \int_{\T_{R_1, R_2}(y)} \int_{B_{\frac{r}{2}}}  (f-l)_+(t_1) \dd v \dd x.
\end{align*}
We conclude:
\bals
	\int_{-R^{2s}}^0& \int_{\T_{R_1, R_2}(y)} \int_{\R^d} \partial_t f(z)\psi_l(z) f_{l, \varepsilon}^{-(1-\zeta)}(z)   \dd v \dd x \dd t \\
	&\leq C R^{-2s}\int_{Q_R} ((f-l)_++\varepsilon)^{\zeta}(z)\dd z + C\varepsilon^{-(1-\zeta)} \frac{R^{d(1+2s)}}{(\bar R -r)^{q-d}} \\
	&\quad + \varepsilon^{-(1-\zeta)} \int_{B_{2R^{1+2s}}} \int_{B_r}  (f-l)_+(t_1) \dd v \dd x.
\eals

Overall, we find
\bal\label{eq:transport}
	\int_{-R^{2s}}^0& \int_{\T_{R_1, R_2}(y)} \int_{\R^d} \mathcal T f(z)\psi_l(z) f_{l, \varepsilon}^{-(1-\zeta)}(z)   \dd v \dd x \dd t  \\
	&\leq CR^{-2s} \int_{Q_R}  \big((f-l)_++\varepsilon\big)^{\zeta}(z)\dd z  + C\varepsilon^{-(1-\zeta)} \frac{R^{d(1+2s)}}{(\bar R -r)^{q-d}} \\
	&\quad + \varepsilon^{-(1-\zeta)} \int_{B_{2R^{1+2s}}} \int_{B_r}  (f-l)_+(t_1) \dd v \dd x.
\eal

\textit{Step 3: Relation of transport and non-locality.}

Finally we use the equation. By \eqref{eq:1.1}, \eqref{eq:transport}, and \eqref{eq:E} we see
\bals
	0 &\leq \int_{-R^{2s}}^0 \int_{\T_{R_1, R_2}(y)}  \int_{\R^d} \mathcal T f(z)\psi_l(z) f_{l, \varepsilon}^{-(1-\zeta)}(z)   \dd v \dd x \\
	&\quad+ \int_{-R^{2s}}^0\int_{\T_{R_1, R_2}(y)}  \mathcal E\left(f, \psi_l f_{l, \varepsilon}^{-(1-\zeta)}\right) \dd x \dd t\\
	&\leq C(\zeta) R^{-2s}\int_{Q_{2R}}  ((f-l)_++\varepsilon)^{\zeta}(z) \dd z\\
	&\quad-\int_{-R^{2s}}^0  \int_{\T_{R_1, R_2}(y)} \int_{B_R}\int_{\R^d \setminus B_R}(f-l)_+(w) \eta^2(z)\phi(t) f_{l,\varepsilon}^{-(1-\zeta)}(v) K(v, w)\chi_{f(v) > l}  \dd w \dd z\\
	&\quad +C \varepsilon^{-(1-\zeta)}(\bar R - r)^{-2s}\int_{Q_{2R}^v \times B_{\bar R}} f(v) \dd z+  \frac{\varepsilon^{-(1-\zeta)}}{2}\left(\frac{r}{2}\right)^{-2s} \int_{Q_{2R}^v \times B_{\frac{r}{2}}} (f - l)_+(v) \dd z\\
	 &\quad+ C\varepsilon^{-(1-\zeta)} \frac{R^{d(1+2s)}}{(\bar R -r)^{q-d}}  + \varepsilon^{-(1-\zeta)} \int_{B_{2R^{1+2s}}} \int_{B_r}  (f-l)_+(t_1) \dd v \dd x \\
	 &\quad + C l\varepsilon^{-(1-\zeta)} r^{-2s}  \abs{\{f>l\}  \cap Q_{2R}^v \times B_{\frac{r}{2}}}.
\eals
or rearranged, (using that $B_{\frac{R^{1+2s}}{2}} \subset \T_{R_1, R_2}(y)$),
\bals
	&\int_{Q_{\frac{R}{2}}}\int_{\R^d \setminus B_R}(f-l)_+(w)  f_{l,\varepsilon}^{-(1-\zeta)}(v) K(v, w)\chi_{f(v) > l}  \dd w \dd z\\
	&\leq C R^{-2s}\int_{Q_{2R}}  (f-l)_+^{\zeta}(z) \dd z + C R^{n-2s}\varepsilon^{\zeta}+ C\varepsilon^{-(1-\zeta)} \frac{R^{d(1+2s)}}{(\bar R -r)^{q-d}}\\
	&\quad + C \varepsilon^{-(1-\zeta)}(\bar R - r)^{-2s}\int_{Q_{2R}^v \times B_{\bar R}} f(v) \dd z+  \varepsilon^{-(1-\zeta)}\left(\frac{r}{2}\right)^{-2s} \int_{Q_{2R}^v \times B_{\frac{r}{2}}} (f - l)_+(v) \dd v\\
	&\quad + C l\varepsilon^{-(1-\zeta)} r^{-2s}  \abs{\{f>l\}  \cap Q_{2R}^v \times B_{\frac{r}{2}}}.
\eals
We want to pick $\bar R$ sufficiently large, so that 
\[
	\frac{1}{(\bar R-r)^{q-d}} + \frac{\norm{f}_{L^1(Q_R^v\times \R^d)}}{(\bar R-r)^{2s}} \leq \frac{2^{1-2s}}{r^{2s}} \int_{Q_{2R}^v \times B_{\frac{r}{2}}} (f - l)_+(v) \dd v.
\]
(The right hand side is strict positive, since with no loss in generality there is at least one point where $f > l$ in $Q_R \subset Q_R^v \times B_{\frac{r}{2}}$, so $\bar R < \infty$.)
Then, recalling that $2R \leq r$,
\bal\label{eq:almost-the-end}
	&\int_{Q_{\frac{R}{2}}}\int_{\R^d \setminus B_R}(f-l)_+(w)  f_{l,\varepsilon}^{-(1-\zeta)}(v) K(v, w)\chi_{f(v) > l}  \dd w \dd z\\
	&\leq C R^{-2s}\int_{Q_{2R}}  (f-l)_+^{\zeta}(z) \dd z+ C R^{n-2s}\varepsilon^{\zeta} + C \varepsilon^{-(1-\zeta)}r^{-2s} \int_{Q_{2R}^v \times B_{\frac{r}{2}}} (f - l)_+(v) \dd z\\
	&\quad + C l\varepsilon^{-(1-\zeta)} r^{-2s}  \abs{\{f>l\}  \cap Q_{2R}^v \times B_{\frac{r}{2}}}.
\eal

Now we pick $\varepsilon  >0$ sufficiently small, depending on $l$. Recall $n = 2d(1+s) + 2s$.

If $l = 0$, we pick
\[
	\varepsilon \approx r^{-n} \int_{Q_r} (f-l)_+ \dd z. 
\]

Otherwise, if $l > 0$, we pick
\beq\label{eq:eps-l}
	\varepsilon \approx l r^{-n}\abs{\big\{ f > l\big\}  \cap Q_r}. 
\eeq
	
In the first case we find
\bals
&\int_{Q_{\frac{R}{2}}}\int_{\R^d \setminus B_R}(f-l)_+(w)  f_{l,\varepsilon}^{-(1-\zeta)}(v) K(v, w)\chi_{f(v) > l}  \dd w \dd z\\
	&\leq C r^{-2s}\int_{Q_{r}}  (f-l)_+^{\zeta}(z) \dd z+ C r^{n(1-\zeta)-2s}\left(\int_{Q_r} (f-l)_+ \dd z \right)^{\zeta}.
\eals
Due to Jensen's inequality, we find
\bals
	\int_{Q_{r}}  (f-l)_+^{\zeta}(z) \dd z = r^n \left(\frac{1}{r^n} \int_{Q_{r}}  (f-l)_+^{\zeta}(z) \dd z\right)^{\frac{1}{\zeta} \zeta}\leq r^n \left(\frac{1}{r^n} \int_{Q_{r}}  (f-l)_+(z) \dd z\right)^{ \zeta},
\eals
so that
\bals
	\int_{Q_{\frac{R}{2}}}\int_{\R^d \setminus B_R}&(f-l)_+(w)  f_{l,\varepsilon}^{-(1-\zeta)}(v) K(v, w)\chi_{f(v) > l}  \dd w \dd z\\
	&\leq C r^{n(1-\zeta)-2s}\left(\int_{Q_r} (f-l)_+ \dd z \right)^{\zeta}.
\eals
This implies \eqref{eq:L1-tail-l0} for $l = 0$, if we take out the infimum of $(f-l +\varepsilon)^{-(1-\zeta)}$ over $z \in Q_{\frac{R}{2}} \cap \chi_{f > l}$ on the left hand side, divide by the infimum, 
use $(\inf g)^{-1} = \sup g^{-1}$, so that 
\bals
	&\int_{Q_{\frac{R}{2}}}\int_{\R^d \setminus B_R}(f-l)_+(w) K(v, w)\chi_{f(v) > l}  \dd w \dd z\\
	&\leq C r^{n(1-\zeta)-2s}\left(\sup_{Q_r} (f-l)_+^{1-\zeta} + r^{-n(1-\zeta)} \left( \int_{Q_r} (f-l)_+ \dd z\right)^{1-\zeta} \right)\left(\int_{Q_r} (f-l)_+ \dd z \right)^{\zeta}\\
	&\leq C r^{n(1-\zeta)-2s}\sup_{Q_r} (f-l)_+^{1-\zeta} \left(\int_{Q_r} (f-l)_+ \dd z \right)^{\zeta}+C r^{-2s}  \int_{Q_r} (f-l)_+ \dd z\\
	&\leq C r^{n(1-\zeta)-2s}\left(\sup_{Q_r} (f-l)_+\right)^{1-\zeta} \left(\int_{Q_r} (f-l)_+ \dd z \right)^{\zeta}.
\eals

In the second case, where $l > 0$, we find from \eqref{eq:almost-the-end} with the choice of \eqref{eq:eps-l}
\bals
	&\int_{Q_{\frac{R}{2}}}\int_{\R^d \setminus B_R}(f-l)_+(w)  f_{l,\varepsilon}^{-(1-\zeta)}(v) K(v, w)\chi_{f(v) > l}  \dd w \dd z\\
	&\leq   C r^{-2s}\int_{Q_{r}}  (f-l)_+^{\zeta}(z) \dd z + C l^\zeta r^{n(1-\zeta)-2s}\left(1 + \frac{\sup_{Q_r}(f-l)_+}{l}\right) \abs{\big\{ f > l\big\}  \cap Q_r}^\zeta.
\eals

As before, we take out the infimum of $(f-l +\varepsilon)^{-(1-\zeta)}$ over $z \in Q_{\frac{R}{2}} \cap \chi_{f > l}$ on the left hand side, and find
\bals
	\int_{Q_{\frac{R}{2}}}&\int_{\R^d \setminus B_R} (f-l)_+(w) \chi_{f > l}(v) K(v, w) \dd w\dd v \\
	&\leq  C r^{-2s}\left(\sup_{Q_r} (f-l)_+\right)^{1-\zeta}   \int_{Q_{r}}  (f-l)_+^{\zeta}(z) \dd z \\
	&\quad+ C \left(\sup_{Q_r} (f-l)_+\right)^{1-\zeta}l^\zeta r^{n(1-\zeta)-2s}\left(1 + \frac{\sup_{Q_r}(f-l)_+}{l}\right) \abs{\big\{ f > l\big\}  \cap Q_r}^\zeta\\
	&\quad +Cl^{1-\zeta} r^{-n(1-\zeta)-2s}\abs{\big\{ f > l\big\}\cap Q_r}^{1-\zeta} \int_{Q_{r}}  (f-l)_+^{\zeta}(z) \dd z \\
	&\quad  + C l r^{-2s}\left(1 + \frac{\sup_{Q_r}(f-l)_+}{l}\right) \abs{\big\{ f > l\big\}  \cap Q_r}.
\eals
Again, due to Jensen's inequality
\bals
	\int_{Q_{r}}  (f-l)_+^{\zeta}(z) \dd z &\leq r^n \left(\frac{1}{r^n} \int_{Q_{r}}  (f-l)_+(z) \dd z\right)^{ \zeta}\\
	&\leq r^{n(1-\zeta)} \left(\sup_{Q_r} (f-l)_+\right)^\zeta \abs{\{f > l\} \cap Q_r}^\zeta,
\eals
and 
\[
	\int_{Q_{f}}  (f-l)_+^{\zeta}(z) \dd z\leq \sup_{Q_r} (f-l)_+^{\zeta} \abs{\big\{ f > l\big\}  \cap Q_r} \leq\left( \sup_{Q_r} (f-l)_+\right)^{\zeta} \abs{\big\{ f > l\big\}  \cap Q_r},
\]
so that
\bals
	\int_{Q_{\frac{R}{2}}}&\int_{\R^d \setminus B_R} (f-l)_+(w) \chi_{f > l}(v) K(v, w) \dd w\dd v \\
	&\leq  C \left(\frac{\sup_{Q_r} (f-l)_+}{l}\right)^{1-\zeta} r^{n(1-\zeta)-2s}\left(l + \sup_{Q_r}(f-l)_+\right) \abs{\big\{ f > l\big\}  \cap Q_r}^\zeta\\
	&\quad  + C  r^{-2s}\left(l + \sup_{Q_r}(f-l)_+ + l\left(\frac{\sup_{Q_r}(f-l)_+}{l}\right)^\zeta\right) \abs{\big\{ f > l\big\}  \cap Q_r}\\
	&\leq  Cr^{n(1-\zeta)-2s} \left[\left(\frac{\sup_{Q_r} (f-l)_+}{l}\right)^{1-\zeta} +  \left(\frac{\sup_{Q_r} (f-l)_+}{l}\right)^{\zeta}  + 1\right]\\
	&\qquad\qquad\qquad \times \left( l + \sup_{Q_r}(f-l)_+\right) \abs{\big\{ f > l\big\}  \cap Q_r}^\zeta.
\eals
This concludes the proof of \eqref{eq:L1-tail}, and therefore of the proposition.
\end{proof}

\section{De Giorgi in $L^1$}\label{sec:strategy-A}
We aim to derive De Giorgi's first lemma, which is typically a bound from $L^2$ to $L^\infty$. Such a boundedness result has been derived in \cite[Lemma 4.1]{AL} before, however, with a tail quantity on the right hand side. We aim to derive a fully local boundedness result. Since we saw in the previous section how to bound the tail in $L^1$, we perform the De Giorgi iteration in $L^1$. The result is Proposition \ref{prop:L2-Linfty-A}.

The De Giorgi argument is based on the realisation that for any sub-solution $f$ of \eqref{eq:1.1}, the level set function $(f-l)_+$ for any $l \in \R_+$ is still a sub-solution to \eqref{eq:1.1}. Then the argument relies on three steps. 
The first is the energy estimate, which is typically derived by testing the equation with the solution itself, and thus yields a bound in $L^2$. This is required to estimate the coercive term, which for our equation is the $L^2_{t,x}H^s_v$ norm of the solution. But even then, by Hölder's inequality, the tail term naturally appears in $L^1$, if we pay the supremum on the test function. Thus, we can use the non-local to local bound of Proposition \ref{prop:tail-bound}, to obtain a fully local energy estimate. The second step is the gain of integrability. This is usually done on the same Lebesgue level as the energy estimate, that is $L^2$. Since this only works for solutions on the whole space, we need to localise with a cutoff. This, however, naturally makes appear another tail term, which then, would need to be bounded in $L^2$, which we don't know how to bound. Therefore, we need a gain of integrability in $L^1$. Then, on the one hand, the tail terms are again bounded by Proposition \ref{prop:tail-bound}. On the other hand, we then also have to bound the $L^1_{t, x}W^{s, 1}_v$ norm of the localised level set function, which can be done using Cauchy-Schwarz and the energy estimate. The last step is the proper choice of control quantity, which for us is the $L^1$ norm of the level set functions, which by Chebyshev's inequality, can be related to the measure of the upper level set.
\subsection{Gain of integrability in $L^1$}
The preliminary step we require for the iteration is thus a gain of integrability in $L^1$. For this, we consider the fractional Kolmogorov equation \eqref{eq:frac-kolm}, which plays the role of the constant coefficient equation for \eqref{eq:1.1}, with a right hand side $h = h_1 + h_2$ for some $h_2 \in L^1\big([0, \tau]\times \R^{2d}\big)$ and some $h_1$ such that $$\norm{(-\Delta_v)^{-\frac{s+\epsilon}{2}} h_1}_{L^1([0, \tau]\times \R^{2d})} < +\infty$$ for $0 \leq \epsilon < s$. The role of the $\epsilon$ is the number of derivatives that we can put on the fundamental solution so that these are absorbed in the time integrability. 
We prove:
\begin{proposition}\label{prop:goi}
Let $0 \leq f$ solve (or be a sub-solution of) \eqref{eq:frac-kolm} in $[0, \tau]\times \R^{2d}$, with $f(0, x, v) = f_0(x, v) \in L^1\cap L^2(\R^{2d})$, and with $h = h_1 + h_2$ where $h_1, h_2\in L^1\cap L^2([-\tau,0]\times\R^{2d})$ such that 
$$\norm{(-\Delta_v)^{-\frac{s+\epsilon}{2}} h_1}_{L^1([-\tau, 0]\times \R^{2d})} < +\infty.$$
Then for any $0 \leq \epsilon < s$ and any $1 \leq p < 1+\frac{s-\epsilon}{s + 2d(s+1) + \epsilon}$ there holds
\bals
	\norm{f}_{L^p([0, \tau]\times \R^{2d})} \lesssim &\tau^{\frac{1}{p}-\alpha_\epsilon}\norm{(-\Delta_v)^{-\frac{s+\epsilon}{2}} h_1}_{L^1([-\tau, 0]\times \R^{2d})} +\tau^{\frac{1}{p}+\frac{1}{2}-\alpha_0}\norm{h_2}_{L^1([0, \tau]\times \R^{2d})} \\
	&+ \tau^{\frac{1}{p}+\frac{1}{2}-\alpha_0}\norm{f_0}_{L^1(\R^{2d})},
\eals
where $\alpha_\epsilon = d\left(1+\frac{1}{s}\right)\left(1-\frac{1}{p}\right)+\frac{s+\epsilon}{2s}$.
\end{proposition}
\begin{proof}
Equation \eqref{eq:frac-kolm} admits a fundamental solution, given by
\bals
	0 &\leq f(t, x, v) \\
	&\leq \int_{\R^{2d+1}}\big[h_1(t', x', v') + h_2(t', x', v')\big]J\left(t-t', x-x'-(t-t')v', v-v'\right)\dd t'\dd x'\dd v' \\
	&\qquad+ \int_{\R^{2d}}f_0(x', v')J\left(t, x-x'-tv', v-v'\right)\dd x'\dd v',
\eals
where $(t, x, v) \in (0, \tau)\times \R^{2d}$ and
\beq\label{eq:J}
	J(t, x, v) = \frac{C_d}{t^{d+\frac{d}{s}}}\mathcal J\left(\frac{x}{t^{1+\frac{1}{2s}}}, \frac{v}{t^{\frac{1}{2s}}}\right), \qquad \hat{\mathcal J}(\varphi, \xi) = \exp\left(-\int_0^1 \abs{\xi -\tau\varphi}^{2s}\dd \tau\right).
\eeq
We remark that for any $r \geq 1$ there holds for $t>0$ and for any $\epsilon \geq 0$
\bal\label{eq:fund-sol-lebesgue}
	&\left\Vert J(t, \cdot, \cdot)\right\Vert_{L^r(\R^{2d})} =t^{-d(1+\frac{1}{s})(1-\frac{1}{r})}\left\Vert\mathcal J\right\Vert_{L^r(\R^{2d})}, \\
	&\left\Vert(-\Delta)^{\frac{s+\epsilon}{2}}_vJ(t, \cdot, \cdot)\right\Vert_{L^r(\R^{2d})} =t^{-d(1+\frac{1}{s})(1-\frac{1}{r})-\frac{s+\epsilon}{2s}}\left\Vert(-\Delta)^{\frac{s+\epsilon}{2}}_v\mathcal J\right\Vert_{L^r(\R^{2d})}.
\eal
(The factor $t^{-d(1+\frac{1}{s})}$ comes from the prefactor in \eqref{eq:J}. The fractional derivative $(-\Delta)^{\frac{s+\epsilon}{2}}_v$ gives a factor in time of $t^{-\frac{s+\epsilon}{2s}}$. Finally the $L^r$ norm in space and velocity scales as $t^{\frac{d}{r}(\frac{2+2s}{2s})}$.)

We define the modified convolution
\bals
	f*_t g(x, v):= \int_{\R^{2d}}f(x', v')g(x - x' -tv', v-v')\dd x'\dd v'.
\eals	
Note that the modified convolution satisfies the usual Young inequality independent of $t$:
\bals
	\norm{f*_t g}_{L^r_{x, v}} \leq \norm{f}_{L^p_{x, v}}\norm{g}_{L^q_{x, v}}
\eals
for $1 + \frac{1}{r} = \frac{1}{p} + \frac{1}{q}$. We split $f = \tilde f_0 + f_1 + f_2$ with
\bals
	&\tilde f_0(t, \cdot, \cdot) := f_0*_tJ(t, \cdot, \cdot),\\
	&f_1(t, \cdot, \cdot) := \int_0^t (-\Delta)_v^{-\frac{s+\epsilon}{2}}h_1*_{(t-t')}(-\Delta)_v^{\frac{s+\epsilon}{2}}J(t-t', \cdot, \cdot)\dd t',\\
	&f_2(t, \cdot, \cdot) := \int_0^t h_2*_{(t-t')}J(t-t', \cdot, \cdot)\dd t'.
\eals
By Young's inequality we get for $\alpha_0 = d\left(1+\frac{1}{s}\right)\left(1-\frac{1}{p}\right)+\frac{1}{2}$ and for $\alpha_\epsilon = d\left(1+\frac{1}{s}\right)\left(1-\frac{1}{p}\right)+\frac{s+\epsilon}{2s}$
\bals
	\norm{\tilde f_0(t, \cdot, \cdot)}_{L^p(\R^{2d})} &\leq \norm{f_0}_{L^1(\R^{2d})}\norm{J(t, \cdot, \cdot)}_{L^{p}(\R^{2d})} = \norm{f_0}_{L^1(\R^{2d})}\norm{\mathcal J}_{L^{p}(\R^{2d})}t^{\frac{1}{2}-\alpha_0},\\
	\norm{f_1(t, \cdot, \cdot)}_{L^p(\R^{2d})} &\leq \int_0^t \norm{(-\Delta)_v^{-\frac{s+\epsilon}{2}} h_1(t')}_{L^1(\R^{2d})}\norm{(-\Delta)_v^{\frac{s+\epsilon}{2}}J(t-t', \cdot, \cdot)}_{L^{p}(\R^{2d})}\dd t'\\
	&= \int_0^t \norm{(-\Delta)_v^{-\frac{s+\epsilon}{2}} h_1(t')}_{L^1(\R^{2d})}\norm{(-\Delta)_v^{\frac{s+\epsilon}{2}}\mathcal J}_{L^{p}(\R^{2d})}(t-t')^{-\alpha_\epsilon}\dd t',\\
	\norm{f_2(t, \cdot, \cdot)}_{L^p(\R^{2d})} &\leq \int_0^t \norm{h_2(t')}_{L^1(\R^{2d})}\norm{J(t-t', \cdot, \cdot)}_{L^{p}(\R^{2d})}\dd t'\\
	&= \int_0^t \norm{h_2(t')}_{L^1(\R^{2d})}\norm{ \mathcal J}_{L^{p}(\R^{2d})}(t-t')^{\frac{1}{2}-\alpha_0}\dd t'.
\eals
Since $p\left(\frac{1}{2} - \alpha_0\right) > -1$ we get that $\tilde f_0 \in L^p([0, \tau]\times\R^{2d})$ and
\beqs
	\norm{\tilde f_0}_{L^p([0, \tau]\times\R^{2d})} \leq C\tau^{\frac{1}{p}+\frac{1}{2}-\alpha_0} \norm{f_0}_{L^1(\R^{2d})}.
\eeqs
For $f_1$ and $f_2$ we apply Young's inequality again and get
\bals
	&\norm{f_1}_{L^p([0, \tau]\times\R^{2d})} \leq C\norm{(-\Delta)_v^{-\frac{s+\epsilon}{2}}h_1}_{L^1([0,\tau]\times\R^{2d})}\tau^{-\alpha_\epsilon+\frac{1}{p}},\\
	&\norm{f_2}_{L^p([0, \tau]\times\R^{2d})} \leq C\norm{h_2}_{L^1([0, \tau]\times\R^{2d})}\tau^{\frac{1}{p}+\frac{1}{2}-\alpha_0},
\eals
provided that $\frac{1}{p} - \alpha_\epsilon > 0$. Equivalently, we require
\[
	1 \leq p < 1 + \frac{s-\epsilon}{s + 2d(s+1) + \epsilon}.
\]
\end{proof}

\subsection{De Giorgi iteration}

This section is dedicated to the proof of Proposition \ref{prop:L2-Linfty-A}. 
We use the tail bound from Proposition \ref{prop:tail-bound} to get a local energy estimate. Then we use the gain of integrability in $L^1$ from Proposition \ref{prop:goi}, by comparing the solution of \eqref{eq:1.1} to the fractional Kolmogorov equation \eqref{eq:frac-kolm}. We bound the right hand side by the local energy estimate, and the tail term again with Proposition \ref{prop:tail-bound}. Finally, we iterate the so gained local a priori estimate on level set functions. 

In this section, we let $f$ solve \eqref{eq:1.1} in $[-R^{2s} + t_0, t_0] \times B_{R^{1+2s}}(x_0 + (t-t_0) v_0) \times \R^d$ such that $f$ conserves mass. 
With no loss in generality, we set $z_0 = (0, 0, 0)$. We, moreover, work with zero source term. Note that the energy estimate, the gain of integrability, and correspondingly the De Giorgi iteration, have to be modified in a suitable way to account for non-zero source.

\textit{Step 1: Local energy estimate.}

 Let $\phi \in C_c^\infty(\R^{1+2d})$ be such that $\phi(t, x, v) = 0$ for $(t, x, v)$ outside $Q_{\frac{R}{4}}$ and $\phi(t, x, v) = 1$ for $(t, x, v) \in Q_{\frac{R}{8}}$. We then test \eqref{eq:1.1} with $(f - l)_+\phi^2$, so that
\bal\label{eq:aux1-improved-DG}
	0 &= \int_{\R^{1+2d}} \mathcal T f(z) \phi^2(z) (f - l)_+(z) \dd z- \int_{\R^{1+2d}}\mathcal L f(z) (f - l)_+(z)\phi^2(z) \dd z\\
	&=\frac{1}{2} \int_{\R^{1+2d}} \mathcal T (f-l)^2_+(z)  \phi^2(z)  \dd z - \int_{\R^{1+2d}}\mathcal L f(z) (f - l)_+(z)\phi^2(z) \dd z.
\eal
We observe that for any $\frac{R}{4} < r$, if we abbreviate $(f-l)_+ =: {f_l}_+$ and if we denote by $\Omega_{\rho} := [-\rho^{2s}, 0] \times B_{\rho^{1+2s}}$, then due to \eqref{eq:symmetry}
\bal\label{eq:aux-bound-0}
	&-\int_{\R^{1+2d}}\mathcal L f(z) (f - l)_+(z)\phi^2(z) \dd z \\
	&= \int_{\R^{1+2d}}\int_{\R^d} \big[f(t, x, v) - f(t, x, w)\big]\big((f-l)_+\phi^2\big)(t, x, v)K(t, x, v, w) \dd w \dd z \\
	&= \frac{1}{2}\int_{\Omega_{\frac{R}{4}}} \int_{B_r} \int_{B_r} \big[f(v) - f(w)\big]\Big[\big({f_l}_+\phi^2\big)(v)- \big({f_l}_+\phi^2\big)(w)\Big]K(v, w) \dd w \dd z\\
	&\quad+\frac{1}{2} \int_{\Omega_{\frac{R}{4}}}\int_{B_r}\int_{B_r} \big[f(v) - f(w)\big]\Big[\big({f_l}_+\phi^2\big)(v) +\big({f_l}_+\phi^2\big)(w)\Big]K(v, w) \dd w \dd z\\
	&\quad+\int_{\Omega_{\frac{R}{4}}}\int_{B_r}\int_{\R^d \setminus B_r} \big[f(v) - f(w)\big]\big({f_l}_+\phi^2\big)(z) K(v, w)\dd w \dd z\\
	&= \frac{1}{4}\int_{\Omega_{\frac{R}{4}}} \int_{B_r} \int_{B_r} \big[f(v) - f(w)\big]\Big[\big({f_l}_+\phi^2\big)(v)- \big({f_l}_+\phi^2\big)(w)\Big]\\
	&\qquad\qquad\qquad \times \big[K(v, w)+K(w, v)\big] \dd w \dd z\\
	&\quad+\int_{\Omega_{\frac{R}{4}}}\int_{B_r}\int_{\R^d \setminus B_r} \big[f(v) - f(w)\big]\big({f_l}_+\phi^2\big)(z) K(v, w)\dd w \dd z\\
	&=: \mathcal E_{loc} + \mathcal E_{tail}. 
\eal

\textit{Step 1-i: Transport.}
We integrate by parts the transport term, and use $\abs{\mathcal T \phi} \sim  R^{-2s}$.
\bal\label{eq:step1-i-a}
	-\frac{1}{2} \int_{\R^{1+2d}}\mathcal T (f-l)^2_+(z)  \phi^2(z)  \dd z
	&= \frac{1}{2} \int_{\R^{1+2d}}  (f-l)^2_+(z)  \mathcal T \phi^2(z)  \dd z\\
	& \leq C R^{-2s} \int_{Q_{\frac{R}{4}}} (f-l)^2_+(z) \dd z.
\eal	

\textit{Step 1-ii: Tail bound.}
We use \eqref{eq:L1-tail} from Proposition \ref{prop:tail-bound} for $f$, which we assume to be a solution, in particular a super-solution, so that for $r = \frac{R}{2}$
\bal\label{eq:step1-ii}
	-&\mathcal E_{tail} = \int_{\Omega_{\frac{R}{4}}} \int_{B_r}\int_{\R^d \setminus B_r} \big[f(w) - f(v)\big]\big({f_l}_+\phi^2\big)(z) K(v, w)\dd w \dd z\\
	&\leq \int_{Q_{\frac{R}{4}}} \int_{\R^d \setminus B_{\frac{R}{2}}} \big[f(t, x, w)-l\big] \big({f_l}_+\phi^2\big)(z) K(v, w)\dd w \dd z\\
	&\leq \sup_{Q_{\frac{R}{4}}} (f-l)_+ \int_{Q_{\frac{R}{4}}} \int_{\R^d \setminus B_{\frac{R}{2}}} (f-l)_+(t, x, w)  K(v, w)\chi_{f > l}(v)\dd w \dd z\\
	&\leq C R^{n(1-\zeta)-2s}\sup_{Q_{\frac{R}{4}}} (f-l)_+ \mathcal M_l \abs{\big\{ f > l\big\}  \cap Q_{R}}^\zeta,
\eal
where we recall $\mathcal M_l$ is given in \eqref{M-l}. 

\textit{Step 1-iii: Not-too-non-local operator.}
What remains to be estimated is the local contribution of the non-local operator, 
\bals
	&\mathcal E_{loc} = \frac{1}{2} \int_{B_r} \int_{B_r} \big[f(v) - f(w)\big]\Big[\big({f_l}_+\phi^2\big)(v) -\big({f_l}_+\phi^2\big)(w)\Big]K(v, w) \dd w \dd z,
\eals
where we recall ${f_l}_+ = (f-l)_+$.

\textit{Claim.} For $z = (t, x, v)$ and $r = \frac{R}{2}$
\bal\label{eq:claim1}
	\mathcal E_{loc} &\geq \frac{1}{2}\int_{B_r}\int_{B_r} \Big[\big({f_l}_+\phi\big)(v) -\big({f_l}_+\phi\big)(w) \Big]^2K(v, w) \dd v- Cr^{-2s}  \int_{B_r}  {f_l}_+^2(v)\dd v\\
	&\quad +\frac{1}{2}\int_{B_r}\int_{B_r} \big[{f_l}_-(v) - {f_l}_-(w)\big]\Big[\big({f_l}_+\phi^2\big)(v) -\big({f_l}_+\phi^2\big)(w) \Big]K(v, w) \dd w \dd v.
\eal
Indeed, since 
\beq\label{eq:aux-claim1}
	{f_l}_-(v) - {f_l}_-(w) = f(v) - f(w) - \big({f_l}_+(v) - {f_l}_+(w)\big),
\eeq
and 
\bals
	 \Big[\big({f_l}_+\phi\big)(v) -\big({f_l}_+\phi\big)(w) \Big]^2 &-\big[{f_l}_+(v) - {f_l}_+(w)\big]\Big[\big({f_l}_+\phi^2\big)(v) -\big({f_l}_+\phi^2\big)(w)\Big]\\
	 & = {f_l}_+(v){f_l}_+(w)\big[\phi(v) - \phi(w)\big]^2,
\eals	
there holds
\bals
	\int_{B_r}&\int_{B_r} \Bigg\{\Big[\big({f_l}_+\phi\big)(v) -\big({f_l}_+\phi\big)(w) \Big]^2\\
	&\qquad \qquad + \big[{f_l}_-(v) - {f_l}_-(w)\big]\Big[\big({f_l}_+\phi^2\big)(v) -\big({f_l}_+\phi^2\big)(w) \Big] \Bigg\}K(v, w) \dd w \dd v \\
	&= \int_{B_r}\int_{B_r} \Bigg\{\Big[\big({f_l}_+\phi\big)(v) -\big({f_l}_+\phi\big)(w) \Big]^2\\
	&\qquad \qquad-\big[{f_l}_+(v) -{f_l}_+(w)\big]\Big[\big({f_l}_+\phi^2\big)(v) -\big({f_l}_+\phi^2\big)(w)\Big]\Bigg\}K(v, w) \dd w \dd v \\
	&\quad +\int_{B_r}\int_{B_r} \big[f(v) - f(w)\big]\Big[\big({f_l}_+\phi^2\big)(v) -\big({f_l}_+\phi^2\big)(w) \Big]K(v, w) \dd w \dd v \\
	&=  \int_{B_r}\int_{B_r}{f_l}_+(v){f_l}_+(w)\big[\phi(v) - \phi(w)\big]^2K(v, w) \dd w \dd v  + 2\mathcal E_{loc}.
\eals
Thus, using Young's inequality and \eqref{eq:upperbound}, we further bound
\bals
	\int_{\Omega_{\frac{R}{4}}} &\int_{B_r} \int_{B_r} {f_l}_+(v){f_l}_+(w)\big[\phi(v) - \phi(w)\big]^2K(v, w)  \dd w \dd v \dd x \dd t\\
	&\leq C \norm{\phi}_{C^1_v}^2\int_{\Omega_{\frac{R}{4}}}  \int_{B_r}\int_{B_r}\Big({f_l}^2_+(v)+  {f_l}^2_+(w)\Big)\abs{v-w}^2K( v, w) \dd w \dd z \\
	&\leq Cr^{-2} r^{2-2s} \int_{Q_{\frac{R}{2}}}  {f_l}_+^2(z) \dd z.
\eals
This proves \eqref{eq:claim1}.

In particular, \eqref{eq:claim1} implies for $r = \frac{R}{2}$ 
\bal\label{eq:step1-iii}
	&\mathcal E_{loc} = \iint_{B_r \times B_r} \Big[f(v) -f(w) \Big]\big({f_l}_+\phi^2\big)(v) K(v, w) \dd w \dd v  \\
	&\geq \frac{1}{2}\iint_{B_r \times B_r} \Big[\big({f_l}_+\phi\big)(v) -\big({f_l}_+\phi\big)(w) \Big]^2K(v, w) \dd w \dd v- Cr^{-2s} \int_{B_r} {f_l}_+^2(v)\dd v\\
	&\quad +\frac{1}{2}\iint_{B_r \times B_r} \big[{f_l}_-(v) - {f_l}_-(w)\big]\Big[\big({f_l}_+\phi^2\big)(v) -\big({f_l}_+\phi^2\big)(w) \Big]K(v, w) \dd w \dd v\\
	&\geq \frac{1}{2}\iint_{B_r \times B_r} \Big[\big({f_l}_+\phi\big)(v) -\big({f_l}_+\phi\big)(w) \Big]^2K(v, w) \dd w \dd v- Cr^{-2s} \int_{B_r} {f_l}_+^2(v)\dd v\\
	&\quad -\iint_{B_r \times B_r} \Big[{f_l}_-(v)\big({f_l}_+\phi^2\big)(w) + {f_l}_-(w)\big({f_l}_+\phi^2\big)(v)\Big] K(v, w) \dd w \dd v\\
	&\geq \frac{1}{2}\iint_{B_r \times B_r} \Big[\big({f_l}_+\phi\big)(v) -\big({f_l}_+\phi\big)(w) \Big]^2K(v, w) \dd w \dd v- Cr^{-2s} \int_{B_r} {f_l}_+^2(v)\dd v.
\eal
In the last inequality we used for each $v, w \in \R^d$
\beqs
 	- {f_l}_-(w)\big({f_l}_+\phi^2\big)(v) - {f_l}_-(v)\big({f_l}_+\phi^2\big)(w) \geq 0.
\eeqs

\textit{Step 1-iv: Conclusion.}
Combining \eqref{eq:aux1-improved-DG} and \eqref{eq:aux-bound-0} with \eqref{eq:step1-i-a} 
of \textit{Step 1-i}, \eqref{eq:step1-ii} of \textit{Step 1-ii}, and \eqref{eq:step1-iii} of \textit{Step 1-iii}, we conclude
\bal\label{eq:energy-no-tail}
	&\frac{1}{8}\int_{ \Omega_{\frac{R}{4}}}\iint_{B_{\frac{R}{2}} \times B_{\frac{R}{2}}} \Big[\big((f-l)_+\phi\big)(v) -\big((f-l)_+\phi\big)(w) \Big]^2K(t, x, v, w) \dd w \dd v \dd x \dd t\\
	&\quad\leq CR^{-2s}\int_{Q_{\frac{R}{4}}} (f-l)_+^2(z)\dd z	 \\
	&\qquad +C R^{n(1-\zeta)-2s}\sup_{Q_{\frac{R}{4}}} (f-l)_+ \mathcal M_l \abs{\big\{ f > l\big\}  \cap Q_{R}}^\zeta. 
\eal

\textit{Step 2: Gain of Integrability.} 

For the gain of integrability, we use Proposition \ref{prop:goi}. 
We check that $(f-l)_+ \phi$ satisfies in $\R^{1+2d}$
\bals
	\mathcal T \big[(f-l)_+  \phi \big]&+ (-\Delta_v)^s \big[(f-l)_+ \phi \big] \\
	&\leq  \mathcal L \big[(f-l)_+ \phi\big] +  (-\Delta_v)^s \big[(f-l)_+\phi\big]
	+ (f-l)_+  \mathcal T \phi\\
	&\quad  +   \Bigg(\int_{\R^d\setminus B_{\rho}(v)} \big(f(w)-l\big)_+K(v, w) \big(\phi(v) - \phi(w)\big) \dd w \Bigg)\chi_{f > l} \tilde \eta(v) \\
	&\quad  +   \Bigg(\int_{\R^d} \big(f(w)-l\big)_+K(v, w) \big(\phi(v) - \phi(w)\big) \dd w \Bigg)\chi_{f > l} \big(1-\tilde \eta(v)\big) \\
	&\quad+ \Bigg(\int_{B_{\rho}(v)} \big(f(w)-l\big)_+K(v, w) \big(\phi(v) - \phi(w)\big) \dd w \Bigg)\chi_{f > l}\tilde \eta(v),
\eals
where $0 < \rho \leq R/8$,  $\tilde \eta = 1$ in $B_{\frac{R}{4}+2\rho}$ and $\tilde \eta = 0$ outside $B_{\frac{R}{4}+3\rho}$. 
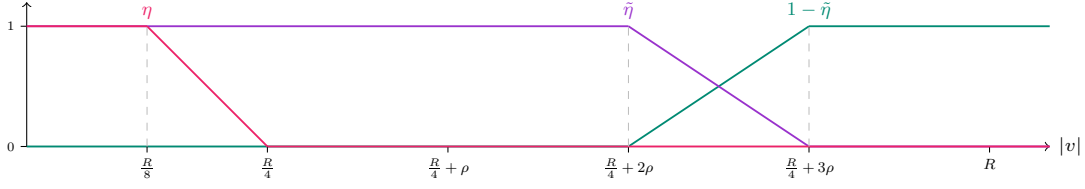
\begin{figure}{
\centering
\begin{tikzpicture}[scale=1.5]
  \draw (-5,1) -- (-5.05, 1) node[anchor=east, scale =0.7] {\footnotesize{$1$}};
  \draw (-5,0) -- (-5.05, 0) node[anchor=east, scale =0.7] {\footnotesize{$0$}};
  \draw (-4, 0) -- (-4, -0.05) node[anchor=north, scale =0.7] {\small$\frac{R}{8}$};
  \draw (-3, 0) -- (-3, -0.05) node[anchor=north, scale =0.7] {\small$\frac{R}{4}$};
  \draw (-1.5, 0) -- (-1.5, -0.05) node[anchor=north, scale =0.7] {\small$\frac{R}{4}+\rho$};
  \draw (0, 0) -- (0, -0.05) node[anchor=north, scale =0.7] {\small$\frac{R}{4}+2\rho$};
  \draw (1.5, 0) -- (1.5, -0.05) node[anchor=north, scale =0.7] {\small$\frac{R}{4}+3\rho$};
  \draw (3, 0) -- (3, -0.05) node[anchor=north, scale =0.7] {\small$R$};
  \draw[WildStrawberry, line width = 0.7pt](-5, 1) -- (-4, 1) node[anchor=south, scale =0.8] {$\eta$};
  \draw[WildStrawberry, line width = 0.7pt] (-4, 1) -- (-3, 0);
  \draw[WildStrawberry, line width = 0.7pt] (-3, 0) -- (3.5, 0);
  \begin{scope}[on background layer]
    \draw[->] (-5, 0) -- (3.5, 0) node[right] {\footnotesize{$\vert v\vert $}};
  \draw[->] (-5, 0) -- (-5, 1.2);
  \draw[PineGreen, line width = 0.7pt] (-5, 0) -- (0, 0);
  \draw[PineGreen, line width = 0.7pt] (0, 0) -- (1.5, 1)node[anchor=south, scale =0.8] {$1-\tilde \eta$};
  \draw[PineGreen, line width = 0.7pt] (1.5, 1) -- (3.5,1);
  \draw[DarkOrchid, line width = 0.7pt] (-5, 1) -- (0, 1)node[anchor=south, scale =0.8] {$\tilde \eta$};
  \draw[DarkOrchid, line width = 0.7pt] (0, 1) -- (1.5, 0);
  \draw[DarkOrchid, line width = 0.7pt] (1.5, 0) -- (3.5,0);
  \end{scope}
    \draw[lightgray, dashed] (-4, 0) -- (-4, 1);
  \draw[lightgray, dashed] (1.5, 0) -- (1.5, 1);
  \draw[lightgray, dashed] (0, 0) -- (0, 1);
\end{tikzpicture}}
\caption{The cutoffs $\eta, \tilde \eta \in C_c^\infty(\R^d)$ in velocity that are used in the proof of Proposition \ref{prop:L2-Linfty-A}.}\label{fig:cutoffs-2}
\end{figure}
The only non-local term in the inequality above is on the third line. Note that $\phi(v) - \phi(w) \leq 0$ for $v$ outside $B_{\frac{R}{4}}$ and $w$ outside $B_\rho(v)$, so that the tail term arising from the equation (on the third line) is non-negative only for $v \in B_{\frac{R}{4}}$, where $\tilde \eta(v) = 1$. 
Furthermore, we note since $\phi$ and $1-\tilde \eta$ have disjoint support, there holds $\phi(v) (1-\tilde \eta(v)) = 0$ for all $v \in \R^d$. Thus, since $0 \leq 1-\tilde \eta \leq 1$, we see that 
\beqs
	 \Bigg(\int_{\R^d} \big(f(w)-l\big)_+K(v, w) \big(\phi(v) - \phi(w)\big) \dd w \Bigg)\chi_{f > l} \big(1-\tilde \eta(v)\big)  \leq 0.
\eeqs

Therefore, there holds for $(t, x, v) \in \R^{1+2d}$
\bals
	\mathcal T &\big[(f-l)_+  \phi \big]+ (-\Delta_v)^s \big[(f-l)_+ \phi \big] \\
	&\leq  \mathcal L \big[(f-l)_+ \phi\big] +  (-\Delta_v)^s \big[(f-l)_+\phi\big]
	+ (f-l)_+  \mathcal T \phi\\
	&\quad  +   \Bigg(\int_{\R^d\setminus B_{\rho}(v)} \big(f(w)-l\big)_+K(v, w) \dd w \Bigg)\chi_{f > l} \phi \\	
	&\quad+ \Bigg(\int_{B_{\rho}(v)} \big(f(w)-l\big)_+K(v, w) \big(\phi(v) - \phi(w)\big) \dd w \Bigg)\chi_{f > l}\tilde \eta(v)\\
	&=: h_1 + h_2,
\eals
with 
\bals
	h_1 &:= \mathcal L \big[(f-l)_+ \phi\big] +  (-\Delta_v)^s \big[(f-l)_+\phi\big],\\
	h_2 &:=(f-l)_+  \mathcal T \phi +\Bigg(\int_{\R^d\setminus B_{\rho}(v)} \big(f(w)-l\big)_+K(v, w) \dd w \Bigg)\chi_{f > l} \phi  \\
	&\qquad+ \Bigg(\int_{B_{\rho}(v)} \big(f(w)-l\big)_+K(v, w) \big(\phi(v) - \phi(w)\big) \dd w \Bigg)\chi_{f > l}\tilde \eta(v).
\eals

Proposition \ref{prop:goi} then implies for any $0 < \epsilon < s$ and any $1 < p < 1 + \frac{s-\epsilon}{2d(1+s)+s+\epsilon}$ 
\bal\label{eq:goi}
	\norm{(f-l)_+  \phi}_{L^p([0, \tau]\times \R^{2d})}  \leq &\tau^{\frac{1}{p}  - \alpha_\epsilon} \norm{(-\Delta_v)^{-\frac{s+\epsilon}{2}} h_1}_{L^1([0, \tau] \times \R^{2d})} \\
	&+\tau^{\frac{1}{p} + \frac{1}{2} - \alpha_0} \norm{h_2}_{L^1([0, \tau] \times \R^{2d})} \\
	&+ \tau^{\frac{1}{p}+ \frac{1}{2} - \alpha_0} \norm{\big[(f-l)_+  \phi\big](0)}_{L^1(\R^{2d})},
\eal
where $\alpha_\epsilon = d\left(1+\frac{1}{s}\right)\left(1-\frac{1}{p}\right)+\frac{s+\epsilon}{2s}$ and, in our case, $\tau \sim R^{2s}$.

\textit{Step 2-i: Contribution of $h_1$.}
To bound $h_1$, we have to estimate, on the one hand, 
\bals
	&\norm{(-\Delta_v)^{-\frac{s+\epsilon}{2}}(-\Delta_v)^s \big[(f-l)_+\phi\big] }_{L^1([0, \tau] \times \R^{2d})} \\
	&=\norm{(-\Delta_v)^{\frac{s-\epsilon}{2}} \big[(f-l)_+\phi\big] }_{L^1([0, \tau] \times \R^{2d})}\\
	&= \int_{\Omega_{\frac{R}{4}}} \iint_{\R^d \times \R^d} \frac{\abs{(f-l)_+\phi(v) - (f-l)_+\phi(w)}}{\abs{v-w}^{d+s-\epsilon}} \dd w \dd z.
\eals
On the other hand,
we note that $(-\Delta_v)^{-\frac{s+\epsilon}{2}}\mathcal L$ is a non-local operator with a kernel $K(v, w) \abs{v-w}^{s+\epsilon}$, which satisfies the same ellipticity assumptions as $K$ with an improved order $s-\epsilon$ (instead of $2s$). In particular, to bound $h_1$ it suffices to bound for $r = \frac{R}{2}$
\bal\label{eq:H1-non-const}
	&\norm{(-\Delta_v)^{-\frac{s+\epsilon}{2}}\mathcal L \big[(f-l)_+\phi\big] }_{L^1([0, \tau] \times \R^{2d})} \\
	&\quad= \int_{\R^{1+2d}} (-\Delta_v)^{-\frac{s+\epsilon}{2}}\mathcal L \big[(f-l)_+\phi\big](z) \dd z \\
	&\quad=\int_{\Omega_{\frac{R}{4}} \times \R^d} \int_{\R^d}  \abs{\big((f-l)_+\phi\big)(w)-\big((f-l)_+\phi\big)(v)} K(v, w) \abs{v-w}^{s+\epsilon} \dd w \dd z \\
	&\quad = \int_{\Omega_{\frac{R}{4}} \times B_r} \int_{B_r} \dots + \int_{\Omega_{\frac{R}{4}} \times \R^d\setminus B_r} \int_{B_r} \dots + \int_{\Omega_{\frac{R}{4}} \times B_r} \int_{\R^d\setminus B_r} \dots. 
\eal
Then, we further note, since $\textrm{supp }\phi \subseteq Q_{\frac{R}{4}}$ and $r = \frac{R}{2}$, by \eqref{eq:upperbound}
\bal\label{eq:H1-non-const-tail}
	&\int_{\Omega_{\frac{R}{4}} \times \R^d\setminus B_r} \int_{B_r}  \abs{\big((f-l)_+\phi\big)(w)-\big((f-l)_+\phi\big)(v)} K(v, w) \abs{v-w}^{s+\epsilon} \dd w \dd z\\
	&\quad+ \int_{\Omega_{\frac{R}{4}} \times B_r} \int_{\R^d\setminus B_r}\abs{\big((f-l)_+\phi\big)(w)-\big((f-l)_+\phi\big)(v)} K(v, w) \abs{v-w}^{s+\epsilon} \dd w \dd z\\
	&\quad= \int_{\Omega_{\frac{R}{4}} \times \R^d\setminus B_{\frac{R}{2}}} \int_{B_{\frac{R}{4}}}  \big((f-l)_+\phi\big)(w) K(v, w) \abs{v-w}^{s+\epsilon} \dd w \dd z\\
	&\quad\quad+ \int_{Q_{\frac{R}{4}}} \int_{\R^d\setminus B_{\frac{R}{2}}}\big((f-l)_+\phi\big)(v) K(v, w) \abs{v-w}^{s+\epsilon} \dd w \dd z\\
	&\quad \leq C R^{-s+\epsilon} \int_{Q_{\frac{R}{4}}} (f-l)_+ \dd z.
\eal
Moreover, we introduce $\mathcal C_K$ to be such that 
\beq\label{eq:choice-CK}
	K^2(v, w)\mathcal C_K^{-1}(v, w) \abs{v-w}^{2s+d} \sim K(v, w).
\eeq
For the Boltzmann equation for instance (see \cite{ISglobal}), we would have
\beqs
	\mathcal C_K(v, w) =\int_{w' \perp w - v} f(v+w') \abs{w'}^{\gamma+2s+1}\tilde b(\cos \theta) \dd w'.
\eeqs
In particular, due to the ellipticity of $K$ in \eqref{eq:upperbound}, we know that
\beqs
	\int_{B_r(v)} \mathcal C_K(v, w) \abs{v-w}^{-d+2\epsilon} \dd w \leq C(\Lambda_0, \epsilon)r^{2\epsilon}.  
\eeqs
Thus
\bals
	&\int_{\Omega_{\frac{R}{4}} \times B_r} \int_{B_r} \abs{\big((f-l)_+\phi\big)(w)-\big((f-l)_+\phi\big)(v)} K(v, w) \abs{v-w}^{s+\epsilon} \dd w \dd z\\
	&\leq \int_{\Omega_{\frac{R}{4}} \times B_r} \int_{B_r} \left(\chi_{f > l}(t, x, v) +\chi_{f > l}(t, x, w)\right) \\
	&\qquad \qquad \times \abs{\big((f-l)_+\phi\big)(w)-\big((f-l)_+\phi\big)(v)} K(v, w) \abs{v-w}^{s+\epsilon} \dd w \dd z\\
	&\leq \int_{\Omega_{\frac{R}{4}} \times B_r} \chi_{f > l}(t, x, v)\left(\int_{B_r}   \abs{\big((f-l)_+\phi\big)(w)-\big((f-l)_+\phi\big)(v)}^2 K^2(v, w) \mathcal C_K^{-1} \abs{v-w}^{2s+d} \dd w\right)^{\frac{1}{2}}\\
	&\qquad \qquad \qquad\qquad\qquad\qquad \times \left(\int_{B_r} \abs{v-w}^{-d+2\epsilon} \mathcal C_K \dd w\right)^{\frac{1}{2}} \dd z\\
	&\quad + \int_{\Omega_{\frac{R}{4}} \times B_r} \chi_{f > l}(t, x, w)\left(\int_{B_r}   \abs{\big((f-l)_+\phi\big)(w)-\big((f-l)_+\phi\big)(v)}^2 K^2(v, w)\mathcal C_K^{-1}  \abs{v-w}^{2s+d} \dd v\right)^{\frac{1}{2}} \\
	&\qquad \qquad \qquad\qquad\qquad\qquad \times \left(\int_{B_r} \abs{v-w}^{-d+2\epsilon} \mathcal C_K \dd v\right)^{\frac{1}{2}} 	\dd w \dd x \dd t\\
	&\leq 2C(\Lambda_0, \epsilon)r^\epsilon\abs{\{ f > l\} \cap Q_{\frac{R}{2}}}^{\frac{1}{2}}\\
	&\qquad \qquad \times\left(\int_{\Omega_{\frac{R}{4}} \times B_r}\int_{B_r}   \abs{\big((f-l)_+\phi\big)(w)-\big((f-l)_+\phi\big)(v)}^2 K^2(v, w) C_K^{-1}\abs{v-w}^{2s+d} \dd w \dd z\right)^{\frac{1}{2}},
\eals
By choice of $\mathcal C_K$ in \eqref{eq:choice-CK}, we then relate the last integral to the $H^s_v$ norm of $(f-l)_+\phi$, which, due to the energy estimate \eqref{eq:energy-no-tail}, is bounded as
\bal\label{eq:H1-non-const-loc}
	&\int_{\Omega_{\frac{R}{4}} \times B_r} \int_{B_r} \abs{\big((f-l)_+\phi\big)(w)-\big((f-l)_+\phi\big)(v)} K(v, w) \abs{v-w}^s \dd w \dd z\\
	&\leq CR^\epsilon \abs{\{ f > l\} \cap Q_{\frac{R}{2}}}^{\frac{1}{2}}\left(\int_{\Omega_{\frac{R}{4}} \times B_r}\int_{B_r}   \abs{\big((f-l)_+\phi\big)(w)-\big((f-l)_+\phi\big)(v)}^2 K(v, w) \dd w \dd z\right)^{\frac{1}{2}}\\
	&\leq CR^{-s+\epsilon}\abs{\{ f > l\} \cap Q_{\frac{R}{2}}}^{\frac{1}{2}}\norm{(f-l)_+\phi}_{L^2(Q_{\frac{3R}{4}})}\\
	&\qquad +  C R^{\frac{n(1-\zeta)}{2}-s+\epsilon} \left(\sup_{Q_{\frac{R}{4}}} (f-l)_+ \mathcal M_l\right)^{\frac{1}{2}} \abs{\big\{ f > l\big\}  \cap Q_{R}}^{\frac{\zeta+1}{2}}.
\eal
Thus we combine \eqref{eq:H1-non-const}, \eqref{eq:H1-non-const-tail} and \eqref{eq:H1-non-const-loc} to get
\bal\label{eq:H1-non-const-bound}
	&\norm{(-\Delta_v)^{-\frac{s+\epsilon}{2}}\mathcal L \big[(f-l)_+\phi\big] }_{L^1([0, \tau] \times \R^{2d})} \\
	&\quad=\int_{\Omega_{\frac{R}{4}} \times \R^d} \int_{\R^d}  \abs{\big((f-l)_+\phi\big)(w)-\big((f-l)_+\phi\big)(v)} K(v, w) \abs{v-w}^s \dd w \dd z \\
	&\quad\leq C R^{-s+\epsilon} \int_{Q_{\frac{R}{4}}} (f-l)_+ \dd z +CR^{-s+\epsilon}\abs{\{ f > l\} \cap Q_{\frac{R}{2}}}^{\frac{1}{2}}\norm{(f-l)_+\phi}_{L^2(Q_{\frac{3R}{4}})}\\
	&\qquad +  CR^{\frac{n(1-\zeta)}{2}-s+\epsilon} \left(\sup_{Q_{\frac{R}{4}}} (f-l)_+ \mathcal M_l\right)^{\frac{1}{2}} \abs{\big\{ f > l\big\}  \cap Q_{R}}^{\frac{\zeta+1}{2}}.
\eal

We conclude thus the bound for $h_1$ using \eqref{eq:H1-non-const-bound}:
\bal\label{eq:H1}
	&\norm{(-\Delta_v)^{-\frac{s+\epsilon}{2}} h_1}_{L^1([0, \tau] \times \R^{2d})} \\
	&\quad\leq C R^{-s+\epsilon} \norm{(f-l)_+}_{L^1(Q_{\frac{R}{4}})} +CR^{-s+\epsilon}\abs{\{ f > l\} \cap Q_{\frac{R}{2}}}^{\frac{1}{2}}\norm{(f-l)_+\phi}_{L^2(Q_{\frac{3R}{4}})}\\
	&\qquad +  CR^{\frac{n(1-\zeta)}{2}-s+\epsilon} \left(\sup_{Q_{\frac{R}{4}}} (f-l)_+ \mathcal M_l\right)^{\frac{1}{2}} \abs{\big\{ f > l\big\}  \cap Q_{R}}^{\frac{\zeta+1}{2}}.
\eal

\textit{Step 2-ii: Contribution of $h_2$.}
To bound $h_2$, we distinguish the transport term, and the tail, which is non-singular, from the local contribution, which is singular,
\bals
	h_2 &:= \underbrace{(f-l)_+  \mathcal T \phi}_{=: h_2^{\rm{transport}}}+ \underbrace{\Bigg(\int_{\R^d\setminus B_{\rho}(v)} \big(f(w)-l\big)_+K(v, w) \dd w \Bigg)\chi_{f > l} \phi}_{=: h_2^{\textrm{tail}}} \\
	&\qquad+ \underbrace{\Bigg(\int_{B_{\rho}(v)} \big(f(w)-l\big)_+K(v, w) \big(\phi(v) - \phi(w)\big) \dd w \Bigg)\chi_{f > l}\tilde \eta(v)}_{=: h_2^{\textrm{loc}}},
\eals
The transport term gives
\bal\label{eq:H2-transport}
	\norm{h_2^{\rm{transport}}}_{L^1(\R^{1+2d})} \leq C R^{-2s} \norm{(f-l)_+}_{L^1(Q_{\frac{R}{2}})}. 
\eal
Then we bound $h_2^{\rm{loc}}$ using the commutator estimates of Lemma \ref{lem:commutator}:
we find for $\rho = \frac{R}{16}$
\bals
	\norm{h_2^{\rm{loc}}}_{L^1(\R^{1+2d})} &=\int_{Q_{\frac{R}{2}}} \Bigg(\int_{B_{\rho}(v)} \big(f(w)-l\big)_+K(v, w) \big(\phi(v) - \phi(w)\big) \dd w \Bigg)\chi_{f > l}\dd z\\
	&\leq  \abs{\{f > l\} \cap Q_{\frac{R}{2}}}^{\frac{1}{2}} \left(\int_{Q_{\frac{R}{2}}} \Bigg(\int_{B_{\rho}(v)} \big(f(w)-l\big)_+K(v, w) \big(\phi(v) - \phi(w)\big) \dd w \Bigg)^2\dd z\right)^{\frac{1}{2}}\\
	&\leq C R^{-2s}\abs{\{f > l\} \cap Q_{\frac{R}{2}}}^{\frac{1}{2}} \norm{(f-l)_+}_{L^2(Q_{\frac{R}{2}})} \\
	&\quad+ C R^{-s}\abs{\{f > l\} \cap Q_{\frac{R}{2}}}^{\frac{1}{2}} \norm{(f-l)_+}_{L^2_{t, x}H^s_v(Q_{\frac{R}{2}})},
\eals
which again by the coercivity \eqref{eq:coercivity} and the energy estimate \eqref{eq:energy-no-tail} is bounded as
\bal\label{eq:H2-loc}
	\norm{h_2^{\rm{loc}}}_{L^1(\R^{1+2d})} &\leq C R^{-2s}\abs{\{f > l\} \cap Q_{\frac{R}{2}}}^{\frac{1}{2}} \norm{(f-l)_+}_{L^2(Q_{\frac{R}{2}})} \\	
	&\quad +C R^{\frac{n(1-\zeta)}{2}-s} \left(\sup_{Q_{\frac{R}{4}}} (f-l)_+ \mathcal M_l\right)^{\frac{1}{2}} \abs{\big\{ f > l\big\}  \cap Q_{R}}^{\frac{\zeta+1}{2}}.
\eal
Finally, to bound the contribution of the tail term, we use the bound \eqref{eq:L1-tail} from Proposition \ref{prop:tail-bound} on $f$:
\bal\label{eq:H2-tail}
	\norm{h_2^{\rm{tail}}}_{L^1(\R^{1+2d})} &= \int_{Q_{\frac{R}{4}}} \int_{\R^d\setminus B_{\rho}(v)} \big(f(w)-l\big)_+K(v, w) \dd w \chi_{f > l} \phi \dd z\\
	&\leq CR^{n(1-\zeta)-2s}  \mathcal M_l \abs{\big\{ f > l\big\}  \cap Q_{R}}^{\zeta}.
\eal

\textit{Step 2-iii: Conclusion.}
Since $\phi = 1$ in $Q_{\frac{R}{8}}$ and $\phi = 0$ outside $Q_{\frac{R}{4}}$, we conclude from \eqref{eq:goi}, \eqref{eq:H1}, \eqref{eq:H2-transport}, \eqref{eq:H2-loc} and \eqref{eq:H2-tail} 
\bal\label{eq:goi-bound}
	\norm{(f-l)_+  \phi}_{L^p(Q_{\frac{R}{8}})}  &\leq CR^{s+\epsilon} \norm{(f-l)_+(t_0)}_{L^1(Q^{t_0}_{\frac{R}{4}})}+  CR^{-s+\epsilon} \norm{(f-l)_+}_{L^1(Q_{\frac{R}{2}})} \\
	&\quad+ C R^{-s+\epsilon}\abs{\{f > l\} \cap Q_{\frac{R}{2}}}^{\frac{1}{2}} \norm{(f-l)_+}_{L^2(Q_{\frac{R}{2}})} \\	
	&\quad +C R^{\frac{n(1-\zeta)}{2}-s+\epsilon} \left(\sup_{Q_{\frac{R}{4}}} (f-l)_+ \mathcal M_l\right)^{\frac{1}{2}} \abs{\big\{ f > l\big\}  \cap Q_{R}}^{\frac{\zeta+1}{2}}\\
	&\quad +CR^{n(1-\zeta)-s+\epsilon}  \mathcal M_l \abs{\big\{ f > l\big\}  \cap Q_{R}}^{\zeta}.
\eal
Let us remark that, in case of non-zero source term, we would have the additional term
\beqs
	CR^{s+\epsilon}\norm{h}_{L^\infty(Q_R)} \abs{\{f > l\} \cap Q_R} + CR^{\epsilon}\norm{h}_{L^\infty(Q_R)}^{\frac{1}{2}} \left( \sup_{Q_{R}} (f-l)_+\right)^{\frac{1}{2}} \abs{\{f > l\} \cap Q_R} 
\eeqs	
on the right hand side of \eqref{eq:goi-bound}.

\textit{Step 3: De Giorgi iteration.}

Let $L > 0$ to be determined and consider 
\beqs
	l_k = L(1 - 2^{-k}),\quad r_k = \frac{R}{8} + 2^{-k}\left(R - \frac{R}{8}\right).
\eeqs
Define $Q_k = \left(-r_k^{2s}, 0\right] \times B_{r_k^{1 + 2s}} \times B_{r_k}$ and
\beqs
	A_k :=  \int_{Q_k} (f - l_k)_+(z) \dd z.
\eeqs
We want to prove that $A_k \to 0$ as $k \to +\infty$. This then proves $$f(t, x, v) \leq L$$ for $(t, x, v) \in Q_{\frac{R}{8}}$.

Chebyschev's inequality gives
\bal
	\abs{\{f > l_{k+1}\} \cap Q_{k}} &= \abs{\left\{(f - l_k)_+> 2^{-k-2}L\right\} \cap  Q_{k}} \\
	&\leq 2^{k+2}L^{-1} \int_{Q_k} (f - l_k)_+(z) \dd z = 2^{k+2}L^{-1} A_k.
\label{eq:chebyschevaux-B}
\eal
To bound the right hand side of \eqref{eq:goi-bound} we pick $t_{k+\frac{1}{2}} \in \left[-r_k^{2s}, -r_{k+1}^{2s}\right]$ such that
\bals
	\int_{Q_k^{t_{k+\frac{1}{2}}}} (f-l_{k})_+(t_{k+\frac{1}{2}}, x, v) \dd v\dd x &\leq \frac{1}{t_{k+1}-t_k}\int_{Q_k} (f-l_{k})_+(t, x, v)\dd z \\
	&\leq C2^kR^{-2s} A_k.
\eals
Moreover, we interpolate the $L^2$ norm between $L^1$ and $L^\infty$,
\bals
	\norm{(f-l_{k+1})_+}_{L^2(Q_{k+1})} \leq \norm{(f-l_{k+1})_+}^{\frac{1}{2}}_{L^1(Q_{k+1})}\left(\sup_{Q_{k+1}} (f-l_{k+1})_+\right)^{\frac{1}{2}}.
\eals
Then we use the gain of integrability \eqref{eq:goi-bound}, Chebyshev \eqref{eq:chebyschevaux-B} and $l_{k+1} \geq l_k$ to find 
\bal\label{eq:DG_aux}
	\Bigg(\int_{t_{k+\frac{1}{2}}}^{0}&\int_{ Q^t_{k+1}}(f - l_{k+1})_+^p(z) \dd z\Bigg)^{\frac{1}{p}}\\
	&\leq C 2^{4k} R^{-s+\epsilon}A_k +  C 2^{3k} R^{-s+\epsilon}A_k \\
	&\quad + C 2^{4k} R^{-s+\epsilon}L^{-\frac{1}{2}}\left( \sup_{Q_{k+1}}(f-l_{k+1})_+\right)^{\frac{1}{2}} A_k \\
	&\quad + C2^{4k} R^{\frac{n(1-\zeta)}{2}-s+\epsilon} \left(\sup_{Q_{k+1}}(f-l_{k+1})_+\mathcal M_{l_{k+1}}\right)^{\frac{1}{2}} L^{-\frac{\zeta+1}{2}} A_k^{\frac{\zeta+1}{2}} \\
	&\quad +C 2^{4k} R^{n(1-\zeta)-s+\epsilon} \mathcal M_{l_{k+1}} L^{-\zeta} A_k^\zeta.
\eal
Again, for non-zero source term, we have the additional terms
\bals
	C &2^{k+2} R^{-s+\epsilon}\left(R^{2s} \norm{h}_{L^\infty(Q_k)}+ R^s \norm{h}_{L^\infty(Q_k)}^{\frac{1}{2}}\left(\sup_{Q_k}(f-l_k)_+\right)^{\frac{1}{2}}\right) L^{-1}A_k\\
	&\leq C 2^{k+2} R^{-s+\epsilon}\left(2R^{2s} \norm{h}_{L^\infty(Q_k)}+ \sup_{Q_k}(f-l_k)_+\right) L^{-1}A_k
\eals
on the right hand side of \eqref{eq:DG_aux}.

For arbitrary $\delta \in (0, 1]$, pick $L$ such that
\beqs
	L \geq \delta \sup_{Q_R}f,
\eeqs
or, in case of non-zero source term, 
\beqs
	L \geq\delta \left(\sup_{Q_R}f + R^{2s} \sup_{Q_R} h\right).
\eeqs

Then we bound $\mathcal M_{l_{k+1}}$ given in \eqref{M-l} by
\bals
	\mathcal M_{l_{k+1}}  &\leq \left[1 +\left(\frac{\sup_{Q_{R}}f}{L}\right)^{\zeta}+ \left(\frac{\sup_{Q_{R}} f}{L}\right)^{1-\zeta}\right] \left(L + \sup_{Q_{R}}f\right) \\
	&\leq \left[1 +\delta^{-\zeta}+ \delta^{-1+\zeta}\right] \left(1+\delta^{-1}\right)L \leq 6\delta^{-1-\zeta}L.
\eals
for $\delta \leq 1$ and $\zeta \geq \frac{1}{2}$.
Now, using $Q_{k+1} \subset Q_k$, Hölder, \eqref{eq:DG_aux} and Chebyschev's inequality \eqref{eq:chebyschevaux-B} we get
\bals
	A_{k+1} &= \int_{Q_{k+1}} (f - l_{k+1})_+(z) \dd z\\
	&\leq \left(\int_{Q_{k+1}} (f - l_{k+1})^p_+(z) \dd z\right)^{\frac{1}{p}} \abs{ \left\{f > l_{k+1}\right\} \cap  Q_{k+1}}^{1 - \frac{1}{p}}\\
	&\leq \left(\int_{t_{k+\frac{1}{2}}}^{0}\int_{Q^t_{k+1}} (f - l_{k+1})^p_+(z) \dd z\right)^{\frac{1}{p}} \abs{ \left\{f> l_{k+1}\right\} \cap  Q_{k}}^{1 - \frac{1}{p}}\\
	&\leq C 2^{4k} R^{-s+\epsilon} \Bigg\{\delta^{-\frac{1}{2}}  A_k^{1-\frac{\zeta}{2}}  +  R^{\frac{n(1-\zeta)}{2}}\delta^{-1-\frac{\zeta}{2}}   L^{1-\frac{\zeta+1}{2}} A_k^{\frac{1}{2}}\\
	&\qquad\qquad\qquad \qquad+ R^{n(1-\zeta)} \delta^{-1-\zeta}  L^{1-\zeta} A_k^{\frac{\zeta}{2}}\Bigg\}L^{-\frac{p-1}{p}} A_k^{\frac{\zeta}{2}+\frac{p-1}{p}}\\
	&\leq C 2^{4k} R^{n(1-\zeta)-s+\epsilon} \delta^{-1-\zeta}  L^{1-\zeta-\frac{p-1}{p}} A_k^{\zeta+\frac{p-1}{p}}. 
\eals
This estimate holds for any $\zeta \in (0, 1)$. We pick $\zeta \in (1 - \frac{p-1}{p}, 1)$ as close to $1$ as possible.

To determine how large we can pick $L$, we consider the sequence $A_k^* := A_0 Q^{-k}$ for some $Q > 1$ to be determined and for $k \geq 0$. Then we require $A_k^*$ to satisfy the reverse non-linear recurrence, that is
\beq
	A_{k+1}^* \geq 2^{4k} R^{n(1-\zeta)-s+\epsilon}\delta^{-1-\zeta} L^{1-\zeta-\frac{p-1}{p}}(A_k^*)^{\zeta + \frac{p-1}{p}}.
\label{eq:reverse}
\eeq
Equivalently,
\beqs
	1 \geq  2^{4k} R^{n(1-\zeta)-s+\epsilon}\delta^{-1-\zeta}L^{1-\zeta-\frac{p-1}{p}}A_0^{\zeta-1+\frac{p-1}{p}} Q^{-k(\zeta-1+\frac{p-1}{p})}Q.
\eeqs
Choose $Q$ such that $2^4Q^{-(\zeta-1 +\frac{p-1}{p})} \leq 1$ or $2^{4} \leq Q^{\zeta -1 +\frac{p-1}{p}}$.
Then \eqref{eq:reverse} holds if 
\beqs
	L^{\zeta-1+\frac{p-1}{p}}\geq  R^{n(1-\zeta)-s+\epsilon}A_0^{\zeta-1+\frac{p-1}{p}} Q\delta^{-1-\zeta}.
\eeqs
Therefore, we pick
\beq\label{eq:L}
	L =  A_0 R^{\frac{n(1-\zeta) -(s-\epsilon)}{\zeta - 1 + \frac{p-1}{p}}} 2^{\frac{4}{(\zeta-1+\frac{p-1}{p})^2}}\delta^{-\frac{1-\zeta}{\zeta -1 +\frac{p-1}{p}}} + \delta \left(\sup_{Q_R}f + R^{2s} \sup_{Q_R} h\right).
\eeq
Then since $\frac{p-1}{p} > 0$ and since $A_0 \leq  2^{\frac{4}{(\zeta-1+\frac{p-1}{p})^2}} R^{\frac{n(1-\zeta) -(s-\epsilon)}{\zeta - 1 + \frac{p-1}{p}}} \delta^{-\frac{1-\zeta}{\zeta -1 +\frac{p-1}{p}}} L$, 
we deduce inductively
\beqs
	A_k \leq 2^{-\frac{4k}{\zeta-1+\frac{p-1}{p}}}A_0,
\eeqs
so that $A_k \to 0$ as $k \to +\infty$.
In particular, for almost every $(t, x, v) \in Q_{\frac{R}{8}}$
\beqs
	 f(z) \leq L = \delta \left(\sup_{Q_R}f + R^{2s} \sup_{Q_R} h\right) + 2^{\frac{4}{(\zeta-1+\frac{p-1}{p})^2}}\delta^{-\frac{1-\zeta}{\zeta -1 +\frac{p-1}{p}}} R^{\frac{n(1-\zeta) -(s-\epsilon)}{\zeta - 1 + \frac{p-1}{p}}} \int_{Q_R} f(z) \dd z.
\eeqs
Recall that we pick $\zeta$ arbitrarily close to $1$, and $\epsilon$ arbitrarily close to $0$.
Finally, since $\sup_{Q_R} f < +\infty$ (due to the non-local $L^2$-$L^\infty$ estimate as established in \cite[Lemma 6.6]{IS}, or \cite[Lemma 4.1]{AL} together with Definition \ref{def:solutions}), so that we can absorb the first term on the left hand side with Lemma \ref{lem:covering}. 
This concludes the proof of Proposition \ref{prop:L2-Linfty-A}. Remarkably, the proof only works for solutions, and not for sub-solutions, which is in contrast to the local case.

\section{Strong Harnack inequality}\label{sec:strong_H}
\begin{proof}[Proof of the Strong Harnack inequality in Theorem \ref{thm:strong-Harnack}]
As a consequence of Proposition \ref{prop:L2-Linfty-A}, Young's inequality, the non-negativity of $f$ and the Weak Harnack \eqref{eq:weakH}, we obtain for any $\delta_1 \in (0, 1)$
\bals
	\sup_{Q_{\frac{R}{8}}(z_0)} f &\leq C R^{-(2d(1+s) + 2s)} \norm{f}_{L^1(Q_R(z_0))} + CR^{2s} \norm{h}_{L^\infty(Q_R(z_0))} \\
	&\leq \delta_1\sup_{Q_{R}(z_0)} f + C(\delta_1) R^{-\frac{2d(1+s) + 2s}{\zeta}} \Bigg(\int_{Q_R(z_0)} f^\zeta \dd z\Bigg)^{\frac{1}{\zeta}} + C R^{2s}\norm{h}_{L^\infty(Q_R(z_0))}\\
	&\leq \delta_1\sup_{Q_{R}(z_0)} f + C(\delta_1)R^{-\frac{2d(1+s) + 2s}{\zeta}} \inf_{\tilde Q_{R}^+(z_0)} f + CR^{-\frac{2d(1+s) + 2s}{\zeta} + 2s} \norm{h}_{L^\infty(Q_R(z_0))},
\eals
where we recall $\tilde Q_R^+ := Q_R\left(\left(\frac{3}{2} R^{2s} + \frac{1}{2} \left(\frac{R}{2}\right)^{2s}, 0, 0\right)\right)$.
Lemma \ref{lem:covering} permits to absorb the first term on the right hand side, and we conclude
\beqs
	\sup_{Q_{\frac{R}{8}}(z_0)} f \leq CR^{-\frac{2d(1+s) + 2s}{\zeta}} \left(\inf_{\tilde Q_{R}^+(z_0)} f+ R^{2s}\norm{h}_{L^\infty(Q_R(z_0))}\right).
\eeqs
\end{proof}

\section{Upper bound on the fundamental solution}\label{sec:upperbound}

For the sequel we assume that \eqref{eq:1.1} under assumptions \eqref{eq:coercivity}-\eqref{eq:symmetry-nondiv} admits a fundamental solution with the properties \eqref{eq:normalisation-J}, \eqref{eq:symmetry-J} and \eqref{eq:representation-J}. We denote it by $J(t, x, v; \tau, \zeta, \xi)$ for $(t, x, v), (\tau, \zeta, \xi) \in \R^{1+2d}$.

\subsection{On-diagonal bound}

The on-diagonal upper bound follows from Proposition \ref{prop:L2-Linfty-A}.

\begin{theorem}[Nash Upper Bound]\label{thm:nash}
Let $x, v, \zeta, \xi \in \R^d$ and $0 \leq \tau_0 < t < T$. Let $J(t, x, v; \tau_0, \zeta, \xi)$ be the fundamental solution of \eqref{eq:1.1} with a non-negative kernel $K$ satisfying \eqref{eq:coercivity}-\eqref{eq:symmetry}. Then there exists $C > 0$ depending on $s, d, \lambda_0, \Lambda_0$ such that
\beq\label{eq:ondiag}
	J(t, x, v; \tau_0, \zeta, \xi) \leq C (t-\tau_0)^{-\frac{2d(1+s)}{2s}}.
\eeq
\end{theorem}
\begin{proof}
Since $(t, x, v) \to J(t, x, v; \tau_0, \zeta, \xi)$ solves \eqref{eq:1.1}, we find by Proposition \ref{prop:L2-Linfty-A} with $z = (t, x, v)$, $r^{2s} = \frac{1}{4}(t- \tau_0)$ and $R = 2r$
\bals
	J(t, x, v; \tau_0, \zeta, \xi) &\leq \sup_{Q_r(z)} J(\cdot, \cdot, \cdot; \tau_0, \zeta, \xi) \\
	&\leq C (t-\tau_0)^{-\frac{2d(1+s)+2s}{2s}} \int_{Q_{2r}(z)} J(s, y, w; \tau_0, \zeta, \xi) \dd s \dd y \dd w\\
	&\leq C (t-\tau_0)^{-\frac{2d(1+s)+2s}{2s}} \int_{t - r^{2s}}^t \int_{\R^{2d}}  J(s, y, w; \tau_0, \zeta, \xi) \dd s \dd y \dd w\\
	&\leq C (t-\tau_0)^{-\frac{2d(1+s)}{2s}}.
\eals
The last inequality follows by using \eqref{eq:normalisation-J}. 
\end{proof}

As a consequence of \eqref{eq:ondiag} and \eqref{eq:normalisation-J} we get for every $0 \leq \tau_0 < \sigma < T$ and $\zeta, \xi \in \R^d$
\beq\label{eq:ondiag1}
	\Bigg(\int_{\R^{2d}} J^2(\sigma, x, v; \tau_0, \zeta, \xi) \dd x \dd v\Bigg)^{\frac{1}{2}} \leq C (\sigma-\tau_0)^{-\frac{d(1+s)}{2s}},
\eeq
and due to \eqref{eq:symmetry-J} also for every $x, v \in \R^d$
\beq\label{eq:ondiag2}
	\Bigg(\int_{\R^{2d}} J^2(\sigma, x, v; \tau_0, \zeta, \xi) \dd \zeta \dd\xi\Bigg)^{\frac{1}{2}} \leq C (\sigma-\tau_0)^{-\frac{d(1+s)}{2s}}.
\eeq

\subsection{Off-diagonal bound}

In this section, we derive an off-diagonal bound on the fundamental solution. 
\begin{theorem}[Off-diagonal bound]\label{thm:off-diag}
Let $J$ be the fundamental solution of \eqref{eq:1.1} with a non-negative kernel $K$ satisfying \eqref{eq:coercivity}-\eqref{eq:symmetry-nondiv}. Then there exists $C > 0$ depending on $d, \lambda_0, \Lambda_0, s$ such that for every $0 < \rho$, $0 \leq \tau_0 < \sigma < T$, $x, v, y_0, w_0 \in \R^d$ with $\sigma - \tau_0 \leq \frac{\rho^{2s}}{4k}$ there holds
\beqs
	J(\sigma,x, v; \tau_0, y_0, w_0) \leq C \Bigg(  (\sigma - \tau_0)^{-\frac{2d(1+s)}{2s}}\Bigg(\frac{\rho^{2s}}{k(\sigma - \tau_0)}\Bigg)^{\frac{1}{4}- \frac{\gamma}{12\rho}} +\rho^{-\frac{s}{2}} (\sigma - \tau_0)^{\frac{1}{4} - \frac{2d(1+s)}{2s}}\Bigg),
\eeqs
for
\beqs
	\gamma = \frac{1}{4}\max\left\{\abs{v- w_0},  \abs{x - y_0 - (\sigma - \tau_0) w_0}^{\frac{1}{1+2s}}\right\},
\eeqs
with $\abs{\sigma - \tau_0}^{\frac{1}{2s}} \leq 4 \gamma$. 
\end{theorem}
To prove Theorem \ref{thm:off-diag}, we use Aronson's method: we construct a function $H$ that satisfies 
\beqs
	\mathcal T H + \int_{\R^d} \left[H^{\frac{1}{2}}(v) - H^{\frac{1}{2}}(w)\right]^2 K(v, w) \dd w \leq 0. 
\eeqs
Testing \eqref{eq:1.1} with the product of the solution and $H$ then yields 
\beq\label{eq:decay-ineq-h}
	\sup_{t \in (t_0, \sigma)} \int_{\R^{2d}} f^2(t, x, v) H(t, x, v) \dd x \dd v \lesssim \int_{\R^{2d}} f^2(\tau_0,  x, v) H(\tau_0, x, v) \dd x \dd v. 
\eeq
The decay from the function $H$ determines through this inequality \eqref{eq:decay-ineq-h} how quick solutions increase inside a ball that lies outside the support of the initial data:
\begin{theorem}[Aronson's bound]\label{thm:aronson-bound}
Let $0 < \tau_0 < \sigma < T$, $0 < \rho $, $\gamma > 0$ such that $\abs{\sigma - \tau_0}^{\frac{1}{2s}} \leq \gamma$, and $(y_0, w_0) \in \R^{2d}$. Let $f_0 = f(\tau_0) \in L^2(\R^{2d})$ be such that 
\beqs
	f(t_0, x, v) = 0, \qquad \textrm{ for } \max\left\{\abs{x-y_0 - (\sigma-\tau_0) w_0}^{\frac{1}{1+2s}}, \abs{v-w_0}\right\} \leq \gamma.
\eeqs
Let $f \in L^\infty\big((\tau_0, T) \times \R^{2d}\big)$ solve \eqref{eq:1.1} in $(\tau_0, T) \times \R^{2d}$ with a non-local operator whose kernel is non-negative and satisfies \eqref{eq:coercivity}, \eqref{eq:upperbound}, \eqref{eq:symmetry}, \eqref{eq:symmetry-nondiv}.
Then there exists $k \geq 1$,  and $C > 0$ depending on $s, d, \lambda_0, \Lambda_0$ such that if $\sigma - \tau_0 < \frac{\rho^{2s}}{4k}$ there holds
\beqs
	\Abs{f(\sigma,y_0, w_0) } \leq C \Bigg[(\sigma - \tau_0)^{-\frac{d(1+s)}{2s}} \Bigg(\frac{\rho^{2s}}{k(\sigma - \tau_0)} \Bigg)^{\frac{1}{2}- \frac{\gamma}{6\rho}}+  \rho^{-s} (\sigma - \tau_0)^{\frac{1}{2}-\frac{d(1+s)}{2s}} \Bigg]\norm{f_0}_{L^2(\R^{2d})}.
\eeqs
\end{theorem}
Before proving Theorem \ref{thm:aronson-bound}, we complete the proof of the off-diagonal bound in Theorem \ref{thm:off-diag} using Theorem \ref{thm:aronson-bound} and the on-diagonal bounds in Theorem \ref{thm:nash}.

\begin{proof}[Proof of Theorem \ref{thm:off-diag}]
\textit{Step 1: First consequence of Theorem \ref{thm:aronson-bound}.}

Let $0 < \rho$, and $0 \leq \tau_0 < \sigma$ be such that $\sigma - \tau_0 \leq \frac{\rho^{2s}}{4k}$. Let $v, \xi, x, \zeta \in \R^d$. Then we claim that for $0 \leq \tau_0 < t < \sigma$ there holds
\bal\label{eq:offdiag-aux}
	&\int_{\R^{2d} \setminus Q^{\tau_0}_\gamma(t, x, v)} J^2(t, x, v; \tau_0, \zeta, \xi) \dd \zeta \dd \xi \\
	&\qquad\qquad\leq C (t - \tau_0)^{-\frac{2d(1+s) }{2s}} \Bigg(\frac{\rho^{2s}}{k(\sigma - \tau_0)}\Bigg)^{\frac{1}{2}- \frac{\gamma}{6\rho}} + \rho^{-s} (t - \tau_0)^{\frac{1}{2}-\frac{2d(1+s)}{2s}} , \qquad x, v \in \R^d, \\
	&\int_{\R^{2d} \setminus Q^{t}_\gamma(\tau_0, \zeta, \xi)} J^2(t, x, v; \tau_0, \zeta, \xi) \dd x \dd v \\
	&\qquad\qquad\leq C (t - \tau_0)^{-\frac{2d(1+s) }{2s}} \Bigg(\frac{\rho^{2s}}{k(\sigma - \tau_0)}\Bigg)^{\frac{1}{2}- \frac{\gamma}{6\rho}} +  \rho^{-s} (t - \tau_0)^{\frac{1}{2}-\frac{2d(1+s)}{2s}}, \qquad \zeta, \xi \in \R^d,
\eal
where 
\beqs
	Q^{\tau_0}_\gamma(t, x, v) :=\left\{(\zeta, \xi) \in \R^{2d} : \max\left\{ \abs{\zeta - x - (t-\tau_0) v}^{\frac{1}{1+2s}}, \abs{\xi - v}\right\} < \gamma \right\},
\eeqs
and similarly,
\beqs
	Q^{t}_\gamma(\tau_0, \zeta, \xi) :=\left\{(x, v) \in \R^{2d} : \max\left\{ \abs{x - \zeta - (\tau_0-t) \xi}^{\frac{1}{1+2s}}, \abs{v - \xi}\right\} < \gamma \right\}.
\eeqs

To prove this, we first note
\bals
	\int_{Q_\gamma^{\tau_0}(t, x, v)} J^2(t, x, v; \tau_0, \zeta, \xi) \dd \zeta \dd \xi = \int_{Q_\gamma^{0}(0, y, w)} J^2(t, y -(t-\tau_0) w, w ; \tau_0, \zeta, \xi) \dd \zeta \dd \xi.
\eals
Then we define
\beqs
	f(\sigma, y_0, w_0) := \int_{\R^{2d} \setminus Q_\gamma^{0}(0, y, w)} J(\sigma, y_0, w_0; \tau_0, \zeta, \xi) J(t, y -(t-\tau_0) w, w; \tau_0, \zeta, \xi) \dd \zeta \dd \xi.
\eeqs
This function $f$ is a non-negative solution of \eqref{eq:1.1} for $0 \leq \tau_0 < \sigma < T$ with initial condition
\bals
	f(\tau_0, y_0, w_0) = \begin{cases}
		0, &\qquad \textrm{ if } (y_0, w_0) \in Q_\gamma^{0}(0, y, w),  \\
		J(t, y -(t-\tau_0) w, w; \tau_0, y_0, w_0), &\qquad \textrm{ if } (y_0, w_0) \in\R^{2d} \setminus Q_\gamma^{0}(0, y, w),
	\end{cases}
\eals
due to \eqref{eq:representation-J} and \eqref{eq:symmetry-J}.
Setting $(\sigma, y_0, w_0) = (t, y - (t-\tau_0)w, w)$ yields by Theorem \ref{thm:aronson-bound} 
\bals
	&\int_{\R^{2d} \setminus Q_\gamma^{0}(0, y, w)} J^2(t, y -(t-\tau_0) w, w; \tau_0, \zeta, \xi) \dd \zeta \dd \xi \\
	&=f(\sigma, y - (t- \tau_0) w, w) \\
	&\leq  C (\sigma - \tau_0)^{-\frac{d(1+s) }{2s}} \Bigg(\frac{\rho^{2s}}{k(\sigma - \tau_0)}\Bigg)^{\frac{1}{2}- \frac{\gamma}{6\rho}}  \norm{f_0}_{L^2(\R^{2d})} + \rho^{-s} (\sigma - \tau_0)^{\frac{1}{2}-\frac{d(1+s)}{2s}} \norm{f_0}_{L^2(\R^{2d})}\\
	&\leq  C (\sigma - \tau_0)^{-\frac{d(1+s)}{2s}} \Bigg(\frac{\rho^{2s}}{k(\sigma - \tau_0)}\Bigg)^{\frac{1}{2}- \frac{\gamma}{6\rho}}  \norm{J(t, y -(t-\tau_0) w, w; \tau_0, \cdot, \cdot)}_{L^2(\R^{2d})} \\
	&\qquad + \rho^{-s} (\sigma - \tau_0)^{\frac{1}{2}-\frac{d(1+s)}{2s}} \norm{J(t, y -(t-\tau_0) w, w; \tau_0, \cdot, \cdot)}_{L^2(\R^{2d})}\\
	&\leq  C (\sigma - \tau_0)^{-\frac{2d(1+s)}{2s}} \Bigg(\frac{\rho^{2s}}{k(\sigma - \tau_0)}\Bigg)^{\frac{1}{2}- \frac{\gamma}{6\rho}}   + \rho^{-s} (\sigma - \tau_0)^{\frac{1}{2}-\frac{2d(1+s)}{2s}}.
\eals
The last inequality follows from the on-diagonal bound \eqref{eq:ondiag2}. This yields the first inequality in claim \eqref{eq:offdiag-aux}. Together with \eqref{eq:symmetry-J} this also yields the second inequality in \eqref{eq:offdiag-aux}.

\textit{Step 2: Second consequence of Theorem \ref{thm:aronson-bound}.}

We fix $0 \leq \tau_0 < \sigma < T$ such that $\sigma - \tau_0 \leq \frac{\rho^{2s}}{4k}$ and $x, v, y_0, w_0 \in \R^d$. We consider
\beq\label{eq:gamma}
	\gamma = \frac{1}{4}\max\left\{\abs{v- w_0},  \abs{x - y_0 - (\sigma - \tau_0) w_0}^{\frac{1}{1+2s}}\right\},
\eeq
and assume $\abs{\sigma - \tau_0}^{\frac{1}{2s}} \leq 4 \gamma$.
Then by \eqref{eq:representation-J}
\bals
	J&(\sigma,x, v; \tau_0, y_0, w_0) \\
	&= \int_{\R^{2d}} J\Big(\sigma, x, v; \sigma - \frac{\sigma - \tau_0}{2}, x', v'\Big) J\Big(\sigma - \frac{\sigma - \tau_0}{2}, x', v'; \tau_0, y_0, w_0\Big) \dd x' \dd v'\\
	&= \underbrace{\int_{\R^{2d}\setminus Q^{\frac{\sigma-\tau_0}{2}}_\gamma(0, x, v) } \dots \dd x' \dd v'}_{=: J_1}  + \underbrace{\int_{Q_\gamma^{
	\frac{\sigma - \tau_0}{2}}(0, x, v)} \dots \dd x' \dd v'}_{=: J_2}.
\eals
For $J_1$ we use \eqref{eq:ondiag1} and \eqref{eq:offdiag-aux} to get
\bals
	J_1 &\leq \Bigg(\int_{\R^{2d}\setminus Q^{\frac{\sigma - \tau_0}{2}}_\gamma(0, x, v) } J^2\Big(\sigma, x, v; \sigma -\frac{\sigma - \tau_0}{2}, x', v'\Big)  \dd x' \dd v'\Bigg)^{\frac{1}{2}}  \\
	&\qquad  \times\Bigg( \int_{\R^{2d}\setminus Q^{\frac{\sigma - \tau_0}{2}}_\gamma(0, x, v) }   J^2\Big(\sigma - \frac{\sigma - \tau_0}{2}, x', v'; \tau_0, y_0, w_0\Big) \dd x' \dd v'\Bigg)^{\frac{1}{2}}\\
	&\leq C \Bigg(  (\sigma - \tau_0)^{-\frac{2d(1+s)}{4s}}\Bigg(\frac{\rho^{2s}}{k(\sigma - \tau_0)}\Bigg)^{\frac{1}{4}- \frac{\gamma}{12\rho}} +\rho^{-\frac{s}{2}} (\sigma - \tau_0)^{\frac{1}{4} - \frac{2d(1+s)}{4s}}  \Bigg)  (\sigma-\tau_0)^{-\frac{2d(1+s)}{4s}}\\
	&\leq C \Bigg(  (\sigma - \tau_0)^{-\frac{2d(1+s)}{2s}}\Bigg(\frac{\rho^{2s}}{k(\sigma - \tau_0)}\Bigg)^{\frac{1}{4}- \frac{\gamma}{12\rho}} +\rho^{-\frac{s}{2}} (\sigma - \tau_0)^{\frac{1}{4} - \frac{2d(1+s)}{2s}}  \Bigg).
\eals
For $J_2$ we first note that for $x', v' \in Q_\gamma^{\frac{\sigma - \tau_0}{2}}(0, x, v)$ by choice of $\gamma$ in \eqref{eq:gamma} either
\bals
	\abs{v' - w_0} \geq \abs{v - w_0} - \abs{v' - v} \geq 4 \gamma - \gamma \geq \gamma,
\eals
or
\bals
	(4 \gamma)^{1+2s} &= \abs{x - y_0 - (\sigma - \tau_0) w_0}\\
	 &= \abs{\left(x - x'  - \frac{(\sigma - \tau_0)}{2} v \right)+ \left( x' - y_0 - \frac{(\sigma - \tau_0)}{2} w_0 \right) + \frac{(\sigma - \tau_0)}{2}(v - w_0)}\\	&\leq \abs{\gamma^{1+2s} +\abs{ x' - y_0 - \frac{(\sigma - \tau_0)}{2} w_0} + \frac{1}{2}4^{2s}\gamma^{2s} 4 \gamma } \\
	&\leq (1+2^{1+4s})\gamma^{1+2s} +\abs{ x' - y_0 - \frac{(\sigma - \tau_0)}{2} w_0},
\eals	
since $x' \in Q_\gamma^{\frac{\sigma - \tau_0}{2}}(0, x, v)$ and by choice of $\gamma$ in \eqref{eq:gamma}, so that
\bals
	\abs{ x' - y_0 - \frac{(\sigma - \tau)}{2} w_0}^{\frac{1}{1+2s}} \geq (4^{1+2s} -  (1+2^{1+4s}))^{\frac{1}{1+2s}} \gamma \geq \gamma.
\eals

In other words, $Q_\gamma^{\frac{\sigma - \tau_0}{2}}(0, x, v) \subset \R^{2d} \setminus Q_\gamma^{\frac{\sigma - \tau_0}{2}}(0, y_0, w_0)$ so that again by \eqref{eq:ondiag1} and \eqref{eq:offdiag-aux} we get
\bals
	J_2 &\leq \Bigg(\int_{\R^{2d}\setminus Q^{\frac{\sigma - \tau_0}{2}}_\gamma(0, y_0, w_0) } J^2\left(\sigma, x, v; \sigma -\frac{\sigma - \tau_0}{2}, x', v'\right)  \dd x' \dd v'\Bigg)^{\frac{1}{2}}  \\
	&\qquad\times\Bigg( \int_{\R^{2d}\setminus Q^{\frac{\sigma - \tau_0}{2}}_\gamma(0, y_0, w_0) }   J^2\left(\sigma- \frac{\sigma - \tau_0}{2}, x', v'; \tau_0, y_0, w_0\right) \dd x' \dd v'\Bigg)^{\frac{1}{2}}\\
	&\leq C \Bigg(  (\sigma - \tau_0)^{-\frac{2d(1+s)}{4s}}\Bigg(\frac{\rho^{2s}}{k(\sigma - \tau_0)}\Bigg)^{\frac{1}{4}- \frac{\gamma}{12\rho}}  +\rho^{-\frac{s}{2}} (\sigma - \tau_0)^{\frac{1}{4} - \frac{2d(1+s)}{4s}} \Bigg)  (\sigma-\tau_0)^{-\frac{2d(1+s)}{4s}}\\
	&\leq C \Bigg(  (\sigma - \tau_0)^{-\frac{2d(1+s) }{2s}}\Bigg(\frac{\rho^{2s}}{k(\sigma - \tau_0)}\Bigg)^{\frac{1}{4}- \frac{\gamma}{12\rho}}  +\rho^{-\frac{s}{2}} (\sigma - \tau_0)^{\frac{1}{4} - \frac{2d(1+s)}{2s}} \Bigg),
\eals 
which proves Theorem \ref{thm:off-diag}.
\end{proof}

It remains to establish Aronson's method.

\subsection{Aronson's method}\label{sec:aronson}
The proof of Theorem \ref{thm:aronson-bound} relies on Aronson's method: we aim to construct a suitable function $H$ to derive \eqref{eq:decay-ineq-h}. 

\subsubsection{Decay relation}

We aim to define a function that satisfies for some $\rho > 0$ and for some constant $c > 0$
\beq\label{eq:aronson-condition}
	\mathcal T H(v)+ c \underbrace{\int_{B_\rho(v)} \Big[ H^{\frac{1}{2}}(v) - H^{\frac{1}{2}}(w)\Big]^2 \big[K(v, w) + K(w, v)\big] \dd w}_{=: \mathcal L_\rho H^{\frac{1}{2}}}\leq 0.
\eeq
Then we can derive the following statement. 
\begin{proposition}\label{prop:aronson}
Let $0 < \tau_0 <\sigma < T$ and $0 < \rho$.
Let $f \in L^\infty\big((\tau_0, T) \times \R^{2d}\big)$ solve \eqref{eq:1.1} in $(\tau_0, T) \times \R^{2d}$ with a non-local operator whose kernel is non-negative and satisfies \eqref{eq:upperbound}, \eqref{eq:symmetry}, \eqref{eq:symmetry-nondiv} (if $s \geq \frac{1}{2}$). Then for every bounded function $H : [\tau_0, \sigma] \times \R^{2d} \to [0, \infty)$ such that $H^{\frac{1}{2}} \in L^2\big( (\tau_0, \sigma) \times \R^{d}; H^s_v(\R^d)\big)$ 
satisfies \eqref{eq:aronson-condition} in $(\tau_0, \sigma) \times \R^{2d}$,
there exists a constant $C = C(\Lambda_0, s, d)$ such that
\bal\label{eq:step1}
	\sup_{t \in (\tau_0, \sigma)} &\int_{\R^{2d}} f^2(t, x, v) H(t, x, v)  \dd x \dd v\\
	\leq &\int_{\R^{2d}} f^2(\tau_0, x, v) H(\tau_0, x, v) \dd x \dd v \\
	&+C \rho^{-2s}\Norm{H}_{L^\infty([\tau_0, \sigma]\times \R^{2d})} \norm{f}_{L^2([\tau_0, \sigma]\times \R^{2d})}^2.
\eal
\end{proposition}
\begin{proof}[Proof of Proposition \ref{prop:aronson}]
For $R \geq \max\{2, 2\rho+1\}$ we consider $\varphi_R \in C^\infty_c(\R^{2d})$ with $0 \leq \varphi_R \leq 1$ such that $\varphi_R \equiv 1$ for $(x, v) \in B_{(R-1)^{1+2s}} \times B_{R-1}$ and $\varphi_R \equiv 0$ for $(x, v)$ outside $B_{R^{1+2s}} \times B_{R}$, with bounded derivatives and such that $\abs{v \cdot \varphi_R} \sim R^{-2s}$. We test \eqref{eq:1.1} with $fH\varphi_R^2$ where $H$ satisfies \eqref{eq:aronson-condition} over $[\tau_0, \tau_1]\times \R^{2d}$ for $0 \leq \tau_0 \leq \tau_1 \leq \sigma$ and get
\bals
	0 =\int_{[\tau_0, \tau_1]\times \R^{2d}}\mathcal Tf &\Big(fH\varphi_R^2\Big)\dd z - \int_{[\tau_0,\tau_1]\times \R^{2d}}\mathcal Lf \Big(fH\varphi_R^2\Big) \dd z.
\eals

\textit{Step 1: Transport operator.}

First we integrate by parts the transport operator
\bal
	&\int_{[\tau_0, \tau_1]\times \R^{2d}}\mathcal Tf \Big(fH\varphi_R^2\Big) \dd z \\
	&\quad= \frac{1}{2}\int_{\R^{2d}} f^2 H \varphi_R^2\Big\vert_{t = \tau_0}^{\tau_1} \dd x\dd v- \int_{[\tau_0, \tau_1] \times \R^{2d}}f^2H \varphi_R v\cdot \nabla_x\varphi_R\dd z - \frac{1}{2}\int_{[\tau_0,\tau_1]\times\R^{2d}} f^2\varphi_R^2 \mathcal T H \dd z\\
	&\quad\geq \frac{1}{2}\int_{\R^{2d}} f^2 H \varphi_R^2\Big\vert_{t = \tau_0}^{\tau_1} \dd x\dd v- CR^{-2s} \int_{[\tau_0, \tau_1] \times \R^{2d}}f^2H \varphi_R \dd z - \frac{1}{2}\int_{[\tau_0,\tau_1]\times\R^{2d}} f^2\varphi_R^2 \mathcal T H \dd z.
\label{eq:aux1-transport}
\eal

\textit{Step 2: Non-local operator.}

Now we deal with the non-local term. 
We write 
\bal\label{eq:E-H}
	\int_{\R^{d}}&\big(\mathcal Lf \big)fH\varphi_R^2 \dd v \\
	&= \frac{1}{2} \int_{B_{2R}}\int_{B_{2R}} \big[f(w) - f(v)\big] \Big[\big(fH\varphi_R^2\big)(v) - \big(fH\varphi_R^2\big)(w)\Big] K(v, w)\dd w\dd v \\
	&\quad+ \frac{1}{2} \int_{B_{2R}}\int_{B_{2R}} \big[f(w) - f(v)\big] \Big[\big(fH\varphi_R^2\big)(v) + \big(fH\varphi_R^2\big)(w)\Big]K(v, w)  \dd w\dd v\\
	&\quad+ \int_{B_{2R}}\int_{\R^d \setminus B_{2R}} \big[f(w) - f(v)\big] \big(fH\varphi_R^2\big)(v) K(v, w) \dd w\dd v\\
	&= \frac{1}{2} \int_{B_{2R}}\int_{B_{2R}} \big[f(w) - f(v)\big] \Big[\big(fH\varphi_R^2\big)(v) - \big(fH\varphi_R^2\big)(w)\Big] K(v, w)\dd w\dd v \\
	&\quad+ \int_{B_{2R}}\int_{\R^d \setminus B_{2R}} \big[f(w) - f(v)\big] \big(fH\varphi_R^2\big)(v) K(v, w) \dd w\dd v\\
	&=: \mathcal I_{loc} + \mathcal I_{tail},
\eal
where we used the divergence structure \eqref{eq:symmetry}.

\textit{Step 2-i: Tail.}

First we note that the tail tends to zero as $R \to \infty$:
\bal\label{eq:tail-aux}
	\mathcal I_{tail} &=\int_{B_{2R}}\int_{\R^d \setminus B_{2R}} \big[f(w) - f(v)\big] \big(fH\varphi_R^2\big)(v) K(v, w) \dd w\dd v \\
	&\leq \int_{B_{R}}\int_{\R^d \setminus B_{2R}} f(w) \big(fH\varphi_R^2\big)(v) K(v, w) \dd w\dd v \\
	&\leq C\Lambda R^{-2s} \norm{H\varphi_R^2}_{L^\infty} \norm{f}_{L^1(B_R)} \norm{f}_{L^\infty(\R^d\setminus B_{2R})} \xrightarrow[R \to \infty]{} 0.
\eal

\textit{Step 2-ii: Symmetric not-too-non-local non-locality.}

Second, to bound $\mathcal I_{loc}$ from \eqref{eq:E-H}, we first distinguish the non-singular from the singular part. We write
\bal\label{eq:I-loc}
	\mathcal I_{loc}&=\int_{B_{2R}}\int_{B_{2R}} \big[f(w) - f(v)\big] \Big[\big(fH\varphi_R^2\big)(v) - \big(fH\varphi_R^2\big)(w)\Big] K(v, w)\dd w\dd v\\
	&=\int_{B_{2R}}\int_{B_{2R}} \big[f(w) - f(v)\big] \Big[\big(fH\varphi_R^2\big)(v) - \big(fH\varphi_R^2\big)(w)\Big] K(v, w)\chi_{\abs{v-w} > \rho} \dd w\dd v\\
	&\quad +\int_{B_{2R}}\int_{B_{2R}} \big[f(w) - f(v)\big] \Big[\big(fH\varphi_R^2\big)(v) - \big(fH\varphi_R^2\big)(w)\Big] K(v, w)\chi_{\abs{v-w} < \rho} \dd w\dd v\\
	&=: \mathcal I_{loc}^{ns} + \mathcal I_{loc}^{s}. 
\eal

Then, on the one hand, we crudely bound with \eqref{eq:upperbound}
\bal\label{eq:I-loc-ns}
	\mathcal I_{loc}^{ns} &= \int_{B_{2R}}\int_{B_{2R}} \big[f(w) - f(v)\big] \Big[\big(fH\varphi_R^2\big)(v) - \big(fH\varphi_R^2\big)(w)\Big] K(v, w)\chi_{\abs{v-w} > \rho} \dd w\dd v\\
	&\leq \int_{B_{2R}}\int_{B_{2R}} \Big[f(w)\big(fH\varphi_R^2\big)(v) + f(v)\big(fH\varphi_R^2\big)(w) \Big]K(v, w)\chi_{\abs{v-w} > \rho} \dd w\dd v\\
	&\leq \int_{B_{2R}}\int_{B_{2R}}\big[f^2(w)+f^2(v)\big]\big(H\varphi_R^2\big)(v) + \big(H\varphi_R^2\big)(w) \big] K(v, w)\chi_{\abs{v-w} > \rho} \dd w\dd v\\
	&\leq C \rho^{-2s} \norm{H}_{L^\infty} \int_{B_{2R}}f^2(v)\dd v.
\eal

On the other hand, to bound the singular part of $\mathcal I_{loc}$, we realise
\bals
	 \big[f(v) &- f(w) \big]\Big[ \big(f H \varphi_R^2\big)(v) - \big(f H \varphi_R^2\big)(w)\Big]  \\
	 &= \Big[\big(fH^{\frac{1}{2}}\varphi_R\big)(v) - \big(fH^{\frac{1}{2}}\varphi_R\big)(w)\Big]^2 - f(v) f(w) \Big[\big(\varphi_R H^{\frac{1}{2}}\big)(v) - \big(\varphi_R H^{\frac{1}{2}}\big)(w)\Big]^2,
\eals
so that 
\bals
	\mathcal I_{loc}^s&= \int_{B_{2R}}\int_{B_{2R}}f(v) f(w) \Big[\big(\varphi_R H^{\frac{1}{2}}\big)(v) - \big(\varphi_R H^{\frac{1}{2}}\big)(w)\Big]^2 K(v, w) \chi_{\abs{v-w} < \rho}\dd v \dd w\\
	&\quad -\int_{B_{2R}}\int_{B_{2R}}\Big[\big(fH^{\frac{1}{2}}\varphi_R\big)(v) - \big(fH^{\frac{1}{2}}\varphi_R\big)(w)\Big]^2K(v, w) \chi_{\abs{v-w} < \rho} \dd w \dd v.
\eals

Now notice that for some $c_0 \geq 1$ there holds
\bals
	 f(v) &f(w) \Big[\big(\varphi_R H^{\frac{1}{2}}\big)(v) - \big(\varphi_R H^{\frac{1}{2}}\big)(w)\Big]^2\\
	  &\leq c_0 \Big[ f^2(v) + f^2(w)\Big] \Big[\big(\varphi_R(v) - \varphi_R(w)\big)^2\big(H(v) + H(w)\big) + \Big(H^{\frac{1}{2}}(v) - H^{\frac{1}{2}}(w)\Big)^2\Big(\varphi_R^2(v) + \varphi_R^2(w)\Big)\Big].
\eals
Thus
\bal\label{eq:I-loc-s}
	\mathcal I_{loc}^s &\leq \int_{B_{2R}}\int_{B_{2R}}f(v) f(w) \Big[\big(\varphi_R H^{\frac{1}{2}}\big)(v) - \big(\varphi_R H^{\frac{1}{2}}\big)(w)\Big]^2K(v, w) \chi_{\abs{v-w} < \rho} \dd v \dd w\\
	&\lesssim_{c_0} \underbrace{\int_{B_{2R}}\int_{B_{2R}} \big[ f^2(v) + f^2(w)\big] \big[\varphi_R(v) - \varphi_R(w)\big]^2\big[H(v) + H(w)\big] K(v, w) \chi_{\abs{v-w} < \rho} \dd w \dd v}_{=: \mathcal I_{loc}^{s, \varphi}}\\
	& \quad+  \underbrace{\int_{B_{2R}}\int_{B_{2R}}\big[ f^2(v) + f^2(w)\big] \Big[H^{\frac{1}{2}}(v) - H^{\frac{1}{2}}(w)\Big]^2\Big[\varphi_R^2(v) + \varphi_R^2(w)\Big]K(v, w) \chi_{\abs{v-w} < \rho} \dd w \dd v}_{=: \mathcal I_{loc}^{s, H}}.
\eal
We start with $\mathcal I_{loc}^{s, \varphi}$: we observe that $\varphi_R(v) = \varphi_R(w)$ for $w \in B_\rho(v)$ with $v \in B_{\frac{R-1}{2}}$. Thus
\bals
	\mathcal I_{loc}^{s, \varphi}&= \int_{B_{\frac{R-1}{2}}}\int_{B_{2R} \cap B_\rho(v) } \big[ f^2(v) + f^2(w)\big] \big[\varphi_R(v) - \varphi_R(w)\big]^2\big[H(v) + H(w)\big] K(v, w) \dd w \dd v \\
	&\quad+\int_{B_{2R} \setminus B_{\frac{R-1}{2}}}\int_{B_{2R} \cap B_\rho(v) } \big[ f^2(v) + f^2(w)\big] \big[\varphi_R(v) - \varphi_R(w)\big]^2\big[H(v) + H(w)\big] K(v, w) \dd w \dd v\\
	&= \int_{B_{2R} \setminus B_{\frac{R-1}{2}}}\int_{B_{2R} \cap B_\rho(v) } \big[ f^2(v) + f^2(w)\big] \big[\varphi_R(v) - \varphi_R(w)\big]^2\big[H(v) + H(w)\big] K(v, w) \dd w \dd v.
\eals
Since the integrand is bounded, we obtain as $R \to \infty$
\beq\label{eq:I-loc-s-phi}
	\mathcal I_{loc}^{s, \varphi} \to 0.
\eeq
We continue with $\mathcal I_{loc}^{s, H}$: due to \eqref{eq:symmetry}, the integrand is symmetric, so that
\bals
	\frac{1}{2}\mathcal I_{loc}^{s, H} &= \int_{B_{2R}}\int_{B_{2R}} f^2(v) \big[\varphi_R^2(v) +\varphi_R^2(w)\big] \Big[ H^{\frac{1}{2}}(v) - H^{\frac{1}{2}}(w)\Big]^2 K(v, w)\chi_{\abs{v-w} <\rho}\dd w \dd v\\
	&= \int_{B_{\frac{R-1}{2}}}\int_{B_{2R} \cap B_\rho(v)} f^2(v) \big[\varphi_R^2(v) +\varphi_R^2(w)\big] \Big[ H^{\frac{1}{2}}(v) - H^{\frac{1}{2}}(w)\Big]^2K(v, w) \dd w \dd v\\
	&\quad + \int_{B_{2R} \setminus B_{\frac{R-1}{2}}}\int_{B_{2R} \cap B_\rho(v)} f^2(v) \big[\varphi_R^2(v) +\varphi_R^2(w)\big] \Big[ H^{\frac{1}{2}}(v) - H^{\frac{1}{2}}(w)\Big]^2 K(v, w)\dd w \dd v\\
	&=: \mathcal I_{loc}^{H-decay} +\mathcal I_{loc}^{H-error}. 
\eals
The first term encodes the localised decay of $H$. The second term tends to zero as $R \to \infty$, as the integrand is bounded. 
Concretely, since $\varphi_R(v) = \varphi_R(w)$ for any $w \in B_\rho(v)$ and $v \in B_{\frac{R-1}{2}}$ if $R \geq 1 + 2\rho$, we find
\bals
	 \mathcal I_{loc}^{H-decay}= 2\int_{B_{\frac{R-1}{2}}}\int_{B_{2R} \cap B_\rho(v)} f^2(v) \varphi_R^2(v) \Big[ H^{\frac{1}{2}}(v) - H^{\frac{1}{2}}(w)\Big]^2K(v, w) \dd w \dd v.
\eals
For the second term, we see that the integrand is bounded so that 
\beqs
	 \mathcal I_{loc}^{H-error} \leq 2 \int_{B_{2R} \setminus B_{\frac{R-1}{2}}}\int_{B_{2R} \cap B_\rho(v)} f^2(v) \Big[ H^{\frac{1}{2}}(v) - H^{\frac{1}{2}}(w)\Big]^2 K(v, w) \dd w \dd v\xrightarrow[R\to \infty]{} 0.
\eeqs
Thus, as $R \to \infty$, we get 
\bal\label{eq:I-loc-s-H}
	\mathcal I_{loc}^{s, H}  \leq  &4\int_{\R^d}\int_{B_\rho(v) } f^2(v) \varphi_R^2(v) \Big[ H^{\frac{1}{2}}(v) - H^{\frac{1}{2}}(w)\Big]^2 K(v, w)\dd w \dd v.
\eal

The combination of \eqref{eq:I-loc} with \eqref{eq:I-loc-ns}, \eqref{eq:I-loc-s}, \eqref{eq:I-loc-s-phi} and \eqref{eq:I-loc-s-H} yields as $R \to \infty$
\bal\label{eq:sym2}
	\mathcal I_{loc}\leq &2\int_{\R^d}\int_{B_\rho(v)}  f^2(v)\Big[ H^{\frac{1}{2}}(v) - H^{\frac{1}{2}}(w)\Big]^2 \big[K(v, w) +K(w, v)\big] \dd w \dd v \\
	&+ C \rho^{-2s}\Norm{H}_{L^\infty} \norm{f}_{L^2(\R^d)}^2.
\eal

Therefore, by letting $R \to \infty$, we conclude from \eqref{eq:E-H}, \eqref{eq:tail-aux}, \eqref{eq:sym2}
\bal\label{eq:nonloc-op}
	\int_{\R^{d}}\big(\mathcal Lf \big)fH \dd v &\leq 2c_0 \int_{\R^d}\int_{B_\rho(v)}  f^2(v)\Big[ H^{\frac{1}{2}}(v) - H^{\frac{1}{2}}(w)\Big]^2 \big[K(v, w) +K(w, v)\big] \dd w \dd v \\
	&\quad+ C \rho^{-2s}\norm{H}_{L^\infty(\R^d)} \norm{f}_{L^2(\R^d)}^2.
\eal

\textit{Step 3: Conclusion.}

We assemble the pieces. Equation \eqref{eq:aux1-transport} implies
\bals
	\frac{1}{2}\int_{\R^{2d}} f^2(\tau_1, x, v) H(\tau_1, x, v) \varphi_R^2 \dd x \dd v &\leq  \frac{1}{2}\int_{\R^{2d}} f^2(\tau_0, x, v) H(\tau_0, x, v) \varphi_R^2 \dd x \dd v\\
	&\quad+ CR^{-2s} \int_{[\tau_0, \tau_1] \times \R^{2d}}f^2(z)H(z) \varphi_R(x, v) \dd z \\
	&\quad+ \int_{[\tau_0,\tau_1]\times\R^{2d}} f^2(z) \varphi_R^2(x, v)\mathcal T H(z) \dd z\\
	&\quad+ \int_{[\tau_0,\tau_1]\times\R^{2d}} \big(\mathcal L f\big)(z) f(z) H(z) \varphi_R^2(x, v) \dd z,
\eals
which as $R \to \infty$ yields together with \eqref{eq:nonloc-op}
\bals
	&\frac{1}{2}\int_{\R^{2d}} f^2(\tau_1, x, v) H(\tau_1, x, v)  \dd x \dd v \\
	&\leq   \frac{1}{2}\int_{\R^{2d}} f^2(\tau_0, x, v) H(\tau_0, x, v) \dd x \dd v +  \int_{[\tau_0,\tau_1]\times\R^{2d}} f^2(z)  \mathcal T H(z) \dd z\\
	&\quad + 2c_0 \int_{\R^d}\int_{B_\rho(v)}  f^2(v)\Big[ H^{\frac{1}{2}}(v) - H^{\frac{1}{2}}(w)\Big]^2 \big[K(v, w) +K(w, v)\big] \dd w \dd v\\
	&\quad+C \rho^{-2s}\Norm{H}_{L^\infty} \norm{f}_{L^2([\tau_0, \tau_1]\times \R^{2d})}^2 \\
	&\leq   \frac{1}{2}\int_{\R^{2d}} f^2(\tau_0, x, v) H(\tau_0, x, v) \dd x \dd v +C \rho^{-2s}\Norm{H}_{L^\infty} \norm{f}_{L^2([\tau_0, \tau_1]\times \R^{2d})}^2,
\eals
since by construction $H$ satisfies \eqref{eq:aronson-condition} with $c = 2c_0$. This concludes the proof of Proposition \ref{prop:aronson}.
\end{proof}

\subsubsection{Construction of $H$.}
In this section, we construct a function $H$ that satisfies \eqref{eq:aronson-condition}.

\begin{lemma}\label{lem:existence-H}
Let $y_0, w_0 \in \R^d$.  Let $0 < \rho$ and $0 \leq \tau_0  <\sigma$. Let $k \geq 1$ be such that $\sigma - \tau_0 \leq \frac{\rho^{2s}}{4k}$.
For $(t, x, v) \in [\tau_0, \sigma] \times \R^{2d}$, we define $\delta(t) := 2(\sigma - \tau_0) - (t - \tau_0)$ and
\beq\label{eq:H}
	H(t, x, v) := e^{-\max\Big\{1, \frac{1}{3\rho}\max\Big(\abs{v-w_0}, \abs{x - y_0 - (\sigma + t - 2\tau_0)w_0}^{\frac{1}{1+2s}} \Big)\Big\}\log\big(\frac{\rho^{2s}}{k \delta(t)}\big)}.
\eeq
Then there exist a constant $C_1 > 0$ depending only on $s, d, \lambda_0, \Lambda_0$ such that, if $k > C_1$, then $H$ satisfies \eqref{eq:aronson-condition}, where $K$ is a non-negative kernel satisfying \eqref{eq:upperbound}, \eqref{eq:symmetry}, \eqref{eq:symmetry-nondiv} (if $s \geq \frac{1}{2}$). 
\end{lemma}
\begin{proof}
We verify that $H$ satisfies \eqref{eq:aronson-condition} by a case distinction. 

We abbreviate in the sequel
\[
	\abs{X-Y_0} := \abs{x-y_0-(\sigma + t - 2\tau_0)w_0},
\]
for readability.

\begin{enumerate}[i.]

\item For $(t, x, v)$ such that $\abs{v-w_0} < 2\rho$ and  $\abs{X-Y_0}^{\frac{1}{1+2s}} < 3\rho$:
then 
\beqs
	H(t, x, v) = \frac{k\delta(t)}{\rho^{2s}}, 
\eeqs
so that
\beqs
	\partial_t H = - \frac{k}{\rho^{2s}}, \qquad v \cdot \nabla_x H = 0, \qquad \mathcal L_\rho H^{\frac{1}{2}} = 0.
\eeqs
Thus for any $k > 0$ \eqref{eq:aronson-condition} is satisfied.

\item  For $(t, x, v)$ such that $\abs{v-w_0}  < 2\rho$ and $\abs{X-Y_0}^{\frac{1}{1+2s}} > 3\rho$.
Then 
\beqs
	H(t, x, v)  = e^{-\frac{\abs{X-Y_0}^{\frac{1}{1+2s}}}{3\rho} \log\big(\frac{\rho^{2s}}{k \delta(t)}\big)},
\eeqs
so that
\beqs	
	\mathcal L_\rho H^{\frac{1}{2}} = 0,
\eeqs
and
\bals
	&\mathcal T H(t, x, v)\\
	&= -\Bigg[\frac{1}{3\rho(1+2s)} \log\Big(\frac{\rho^{2s}}{k\delta(t)}\Big) \abs{X-Y_0}^{-\frac{1+4s}{1+2s}} (X-Y_0) \cdot (v - w_0) \\
	&\qquad+ \frac{\abs{X-Y_0}^{\frac{1}{1+2s}}}{3\rho\delta(t)} \Bigg]H(t, x, v) \\
	&\leq C \Bigg[ \frac{\rho^{2s}}{3 \rho k\delta(t)} (3\rho)^{-2s} (2\rho)  - \frac{1}{\delta(t)}  \Bigg]e^{-\frac{\abs{X-Y_0}^{\frac{1}{1+2s}}}{3\rho} \log\big(\frac{\rho^{2s}}{k \delta(t)}\big)} \\
	&\leq C \Bigg[ \frac{1}{k\delta(t)}   - \frac{1}{\delta(t)}\Bigg]e^{-\frac{\abs{X-Y_0}^{\frac{1}{1+2s}}}{3\rho} \log\big(\frac{\rho^{2s}}{k \delta(t)}\big)}\leq 0,
\eals
for any $k \geq 1$. We used the bounds $\abs{v-w_0}  < 2\rho, \abs{X-Y_0}^{\frac{1}{1+2s}} > 3\rho$, and the fact that $\log a < a$ for $a > 0$.

\item For $(t, x, v)$ such that $2 \rho < \abs{v-w_0}  <  3\rho$ and $\abs{X-Y_0}^{\frac{1}{1+2s}} < 3\rho$.
Then 
\beqs
	H(t, x, v)  = \frac{k\delta(t)}{\rho^{2s}},
\eeqs
but for $w \in B_\rho(v)$ it might happen that $\abs{w - w_0} > 3\rho$, in which case 
\beqs
	H(t, x, w) = e^{-\frac{1}{3\rho}\abs{w-w_0} \log\big(\frac{\rho^{2s}}{k \delta}\big)} \leq \frac{k\delta(t)}{\rho^{2s}}= H(t, x, v), 
\eeqs
with 
\beqs
	\nabla_v H^{\frac{1}{2}}(t, x, w) = -\frac{1}{6\rho} \log\Big(\frac{\rho^{2s}}{k \delta}\Big)\frac{(w-w_0)}{\abs{w-w_0} } e^{-\frac{1}{6\rho}\abs{w-w_0} \log\big(\frac{\rho^{2s}}{k \delta}\big)}.
\eeqs
We fix $v_0 \in B_\rho(v)$ with $\abs{v_0 - w_0} = 3\rho$ such that for every $w \in  \R^d \setminus B_{3\rho}(w_0)$ there holds $\abs{v_0 -w} \leq 2 \abs{v -w}$. As Kassmann-Weidner \cite{KW-aronson} point out, this works for example if $v_0 \in \partial B_{3\rho}(w_0)$ so that $v_0$ minimises $\textrm{dist}(v, \R^d \setminus B_{3\rho}(w_0))$. 
Then $B_\rho(v) \subset B_{2\rho}(v_0)$ and $H(t, x, v_0) = H(t, x, v)$. 
Note that
$$\Abs{H^{\frac{1}{2}}(t, x, v_0) - H^{\frac{1}{2}}(t, x, w)}^2 \leq  \abs{\nabla_v H^{\frac{1}{2}}(t, x, v_0)}^2 \abs{v_0-w}^{2}$$
since 
\beqs
	\sup_{w \in \R^d\setminus B_{3\rho}(w_0)} \Abs{\nabla_v H^{\frac{1}{2}}(t, x, w)} = \Abs{\nabla_v H^{\frac{1}{2}}(t, x, v_0)}, \quad \textrm{ for } v_0 \in \partial B_{3\rho}(w_0).
\eeqs
Therefore, using the upper bound \eqref{eq:upperbound} and $\abs{v_0 - w_0} = 3\rho$, as well as $\log a \leq a^{1/2}$, we get 
\bal\label{eq:lsym-iii}
	\mathcal L_\rho H^{\frac{1}{2}} (v) &= \int_{B_\rho(v)} \left(H^{\frac{1}{2}} (v) - H^{\frac{1}{2}} (w)\right)^2 \big(K(v, w) + K(w, v)\big) \dd w\\
	&=\int_{B_\rho(v)\setminus B_{3\rho}(w_0)}  \left(H^{\frac{1}{2}} (v) - H^{\frac{1}{2}} (w)\right)^2 \big(K(v, w) + K(w, v)\big)  \dd w\\
	&\leq 2\int_{B_\rho(v)\setminus B_{3\rho}(w_0)}  \left(H^{\frac{1}{2}} (v_0) - H^{\frac{1}{2}} (w)\right)^2 K(v_0, w)  \dd w\\
	&\leq C\int_{B_\rho(v)\setminus B_{3\rho}(w_0)}  \abs{\nabla_v H^{\frac{1}{2}} (v_0)}^2 \abs{v_0-w}^2 K(v_0, w)  \dd w\\
	&\leq C \rho^{2-2s} \frac{1}{36\rho^2} \left(\log\Big(\frac{\rho^{2s}}{k \delta}\Big)\right)^2 e^{-\frac{1}{3\rho}\abs{v_0-w_0} \log\big(\frac{\rho^{2s}}{k \delta}\big)}\\
	&= C \rho^{-2s}  \left(\frac{\rho^{2s}}{k \delta}\right) \frac{k \delta}{\rho^{2s}} =: c_1 \rho^{-2s},
\eal
for some constant $c_1 > 0$ independent of $\rho, \delta(t), k$.
Together with 
\beqs
	\mathcal T H(t, x, v) = - \frac{k}{\rho^{2s}},
\eeqs
we thus find
\beqs
	\mathcal T H + \mathcal L_\rho H^{\frac{1}{2}} \leq 0
\eeqs
for $k \geq c_1$.

\item  For $(t, x, v)$ such that $2 \rho < \abs{v-w_0}  <  3\rho$ and $\abs{X-Y_0}^{\frac{1}{1+2s}} > 3\rho$.
Then 
\beqs
	H(t, x, v)  = e^{-\frac{\abs{X-Y_0}^{\frac{1}{1+2s}}}{3\rho} \log\big(\frac{\rho^{2s}}{k \delta(t)}\big)},
\eeqs
but for $w \in B_\rho(v)$ it might happen that $\abs{w - w_0} > \abs{X-Y_0}^{\frac{1}{1+2s}} >  3\rho$, in which case 
\beqs
	H(t, x, w) = e^{-\frac{1}{3\rho}\abs{w-w_0} \log\big(\frac{\rho^{2s}}{k \delta(t)}\big)}\leq e^{-\frac{\abs{X-Y_0}^{\frac{1}{1+2s}}}{3\rho} \log\big(\frac{\rho^{2s}}{k \delta(t)}\big)} = H(t, x, v).
\eeqs
Similar to before, we fix $v_0 \in B_\rho(v)$ with $\abs{v_0 - w_0} = \abs{X-Y_0}^{\frac{1}{1+2s}}$ such that for every $w \in  \R^d \setminus B_{\abs{X-Y_0}^{\frac{1}{(1+2s)}}}(w_0)$ it holds $\abs{v_0 -w} \leq 2 \abs{v -w}$. For instance, we choose $v_0 \in \partial B_{\abs{X-Y_0}^{\frac{1}{(1+2s)}}}(w_0)$ so that $v_0$ minimises the distance between $v$ and \\
$\R^d \setminus B_{\abs{X-Y_0}^{\frac{1}{(1+2s)}}}(w_0)$. 
Then $B_\rho(v) \subset B_{2\rho}(v_0)$ and $H(t, x, v_0) = H(t, x, v)$. If we denote by $\Omega := B_\rho(v)\setminus B_{\abs{X-Y_0}^{\frac{1}{(1+2s)}}}(w_0)$, then we find similar to before
\bals
	\mathcal L_\rho H^{\frac{1}{2}} (v) &= \int_{B_\rho(v)} \left(H^{\frac{1}{2}} (v) - H^{\frac{1}{2}} (w)\right)^2 \big(K(v, w) + K(w, v)\big) \dd w\\
	&=\int_{\Omega } \left(H^{\frac{1}{2}} (v) - H^{\frac{1}{2}} (w)\right)^2 \big(K(v, w) + K(w, v)\big) \dd w\\
	&\leq \abs{\nabla_v H^{\frac{1}{2}}(v_0)}^2 \int_{\Omega}  \abs{v_0-w}^2 K(v_0, w) \dd w\\
	&\leq C \rho^{2-2s} \frac{1}{36\rho^2} \left(\log\Big(\frac{\rho^{2s}}{k \delta}\Big)\right)^2 e^{-\frac{1}{3\rho}\abs{X-Y_0}^{\frac{1}{1+2s}} \log\big(\frac{\rho^{2s}}{k \delta}\big)}\\
	&=: c_2 \frac{1}{k \delta(t)}  e^{-\frac{1}{3\rho}\abs{X-Y_0}^{\frac{1}{1+2s}} \log\big(\frac{\rho^{2s}}{k \delta}\big)},
\eals
for $c_2$ independent of $k, \rho, \delta(t)$.
Finally, using $\log a < a$ for $a > 0$
\bals	
	&\mathcal T H(t, x, v) \\
	&= -\Bigg[\frac{1}{3\rho(1+2s)} \log\Big(\frac{\rho^{2s}}{k\delta(t)}\Big) \abs{X-Y_0}^{-\frac{1+4s}{1+2s}} (X-Y_0) \cdot (v - w_0) \\
	&\qquad+ \frac{\abs{X-Y_0}^{\frac{1}{1+2s}}}{3\rho\delta(t)}\Bigg]H(t, x, v)\\
	&\leq C  \Bigg[\frac{1}{3\rho} \log\Big(\frac{\rho^{2s}}{k\delta(t)}\Big)\frac{\rho}{\rho^{2s}}  - \frac{1}{\delta(t)}\Bigg]e^{-\frac{\abs{X-Y_0}^{\frac{1}{1+2s}}}{3\rho} \log\big(\frac{\rho^{2s}}{k \delta(t)}\big)} \\
	&=: c_3 \left[\frac{1}{\delta(t)k} -\frac{1}{\delta(t)} \right]e^{-\frac{\abs{X-Y_0}^{\frac{1}{1+2s}}}{3\rho} \log\big(\frac{\rho^{2s}}{k \delta(t)}\big)}.
\eals
Thus $\mathcal TH+ \mathcal L_\rho H^{\frac{1}{2}} \leq 0$, for $k \geq c_2 + c_3$.

\item For $(t, x, v)$ such that $3 \rho < \abs{v-w_0}$ and either $\abs{X-Y_0}^{\frac{1}{1+2s}} < 3\rho$ or \\
$3 \rho <  \abs{X-Y_0}^{\frac{1}{1+2s}} < \abs{v-w_0}$:
Then 
\beqs
	H(t, x, v) := e^{-\frac{\abs{v-w_0}}{3\rho}\log\big(\frac{\rho^{2s}}{k \delta(t)}\big)},
\eeqs
and
\beqs
	\mathcal T H(t,x, v) =-\frac{\abs{v-w_0}}{3\rho\delta(t)}e^{-\frac{\abs{v-w_0}}{3\rho}\log\big(\frac{\rho^{2s}}{k \delta(t)}\big)} \leq - \frac{c_4}{\delta(t)} e^{-\frac{\abs{v-w_0}}{3\rho}\log\big(\frac{\rho^{2s}}{k \delta(t)}\big)},
\eeqs
for some $c_4 > 0$ independent of $k, \delta(t), \rho$.
Moreover, again as before, using the fact that $\abs{w - w_0} > \abs{v-w_0} - \rho$ we get
\bals
	\mathcal L_\rho H^{\frac{1}{2}} (v) &= \int_{B_\rho(v)} \left(H^{\frac{1}{2}} (v) - H^{\frac{1}{2}} (w)\right)^2 \big(K(v, w) + K(w, v)\big) \dd w\\
	&\leq \sup_{v' \in B_\rho(v)} \abs{\nabla_v H^{\frac{1}{2}}(v')}^2 \int_{B_\rho(v)} \abs{v-w}^2  K(v, w) \dd w \\
	&\leq  C \rho^{2-2s} e^{-\frac{\abs{v-w_0}-\rho}{3\rho}\log\big(\frac{\rho^{2s}}{k \delta(t)}\big)}\Bigg[\log\Bigg(\frac{\rho^{2s}}{k \delta(t)}\Bigg)\Bigg]^2 \frac{1}{(6\rho)^2}\\
	&\leq  C \rho^{2-2s} e^{-\frac{\abs{v-w_0}}{3\rho}\log\big(\frac{\rho^{2s}}{k \delta(t)}\big)}\Bigg[\log\Bigg(\frac{\rho^{2s}}{k \delta(t)}\Bigg)\Bigg]^2\Bigg(\frac{\rho^{2s}}{k \delta(t)}\Bigg)^{\frac{1}{3}} \frac{1}{(6\rho)^2}\\
	&\leq \frac{c_5}{  k \delta(t)} e^{-\frac{\abs{v-w_0}}{3\rho}\log\big(\frac{\rho^{2s}}{k \delta(t)}\big)},
\eals
for some $c_5> 0$ independent of $k, \rho, \delta(t)$.
Thus
\bals
	\mathcal T H + \mathcal L_\rho H^{\frac{1}{2}}  \leq  \left[- \frac{c_4}{\delta(t)} +  \frac{c_5}{k\delta(t)} \right] e^{-\frac{\abs{v-w_0}}{3\rho}\log\big(\frac{\rho^{2s}}{k \delta(t)}\big)} \leq 0,
\eals
for $k > \frac{c_5}{c_4}$.

\item  For $(t, x, v)$ such that $3\rho   <  \abs{v-w_0} < \abs{X-Y_0}^{\frac{1}{1+2s}}$.
Then 
\beqs
	H(t, x, v) := e^{-\frac{\abs{X-Y_0}^{\frac{1}{1+2s}}}{3\rho}\log\big(\frac{\rho^{2s}}{k \delta(t)}\big)},
\eeqs
and
\bals
	&\mathcal T H(t, x,v)\\
	&= -\Bigg[\frac{1}{3\rho(1+2s)} \log\Big(\frac{\rho^{2s}}{ k\delta(t)}\Big) \abs{X-Y_0}^{-\frac{1+4s}{1+2s}} (X-Y_0) \cdot (v - w_0)\\
	&\qquad + \frac{\abs{X-Y_0}^{\frac{1}{1+2s}}}{3\rho\delta(t)}\Bigg]H(t, x, v) \\
	&\leq C\frac{\abs{X-Y_0}^{\frac{1}{1+2s}}}{\rho\delta(t)}\Bigg[ \frac{\rho^{2s}\abs{X-Y_0}^{\frac{-2s}{1+2s}}}{k}  - 1\Bigg]H(t, x, v)  \\
	&\leq C\frac{\abs{X-Y_0}^{\frac{1}{1+2s}}}{\rho\delta(t)}\Bigg[ \frac{1}{k}  - 1\Bigg]e^{-\frac{\abs{X-Y_0}^{\frac{1}{1+2s}}}{3\rho} \log\big(\frac{\rho^{2s}}{k \delta(t)}\big)}\\
	&\leq C\Bigg[ \frac{1}{\delta(t)k}  - \frac{1}{\delta(t)}  \Bigg]e^{-\frac{\abs{X-Y_0}^{\frac{1}{1+2s}}}{3\rho} \log\big(\frac{\rho^{2s}}{k \delta(t)}\big)} \\
	&=: c_6  \Bigg[ \frac{1}{\delta(t)k}  - \frac{1}{\delta(t)}  \Bigg]e^{-\frac{\abs{X-Y_0}^{\frac{1}{1+2s}}}{3\rho} \log\big(\frac{\rho^{2s}}{k \delta(t)}\big)}.
\eals
The last inequality uses $k > 1$. 
Moreover, using $\log a < a^{1/2}$ we denote by $$\Omega := B_\rho(v)\setminus B_{\abs{X-Y_0}^{1/(1+2s)}}(w_0),$$ and we again take $v_0 \in \partial B_{\abs{X-Y_0}^{1/(1+2s)}}(w_0)$, 
so that 
\bals
	\mathcal L_\rho H^{\frac{1}{2}} (v) &= \int_{\Omega} \left(H^{\frac{1}{2}} (v) - H^{\frac{1}{2}} (w)\right)^2 \big(K(v, w) + K(w, v)\big) \dd w\\
	&\leq C \rho^{2-2s}   \abs{\nabla_v H^{\frac{1}{2}}(v_0)}^2 \\
	&\leq C \rho^{2-2s} \left[\log\Big(\frac{\rho^{2s}}{k \delta(t)}\Big)\right]^2 \Big(\frac{1}{6\rho}\Big)^2 e^{-\frac{\abs{X-Y_0}^{\frac{1}{1+2s}}}{3\rho}\log\big(\frac{\rho^{2s}}{k \delta(t)}\big)}\\
	&\leq C \frac{1}{k\delta(t)}e^{-\frac{\abs{X-Y_0}^{\frac{1}{1+2s}}}{3\rho}\log\big(\frac{\rho^{2s}}{k \delta(t)}\big)}\\
	&=: c_{7}\frac{1}{k\delta(t)}e^{-\frac{\abs{X-Y_0}^{\frac{1}{1+2s}}}{3\rho}\log\big(\frac{\rho^{2s}}{k \delta(t)}\big)}.
\eals
Thus, for \eqref{eq:aronson-condition} to hold in this case, we need $k >\frac{c_7 }{c_6}$.
\end{enumerate}
Overall, we conclude the proof upon choosing 
\beqs
	k > \max\left\{ 1, c_1, c_2 + c_3, \frac{c_5}{c_4}, \frac{c_7} {c_6}\right\},
\eeqs  
where $c_i > 0$ for $i = 1, \dots, 7$ are constants independent of $\rho$ and $\delta(t)$.
\end{proof}

\subsubsection{Proof of Aronson's bound}
The existence of a function satisfying \eqref{eq:aronson-condition} given by Lemma \ref{lem:existence-H} combined with Proposition \ref{prop:aronson} implies Theorem \ref{thm:aronson-bound}.
\begin{proof}[Proof of Theorem \ref{thm:aronson-bound}]
\textit{Step 1: Bounds on $H$.}

We recall the definition of $H$ from \eqref{eq:H}: for $\delta(t) = 2(\sigma - \tau_0) - (t - \tau_0)$ we have defined
\beqs
	H(t, x, v) = e^{-\max\Big\{1, \frac{1}{3\rho}\max\Big(\abs{v-w_0}, \abs{x - y_0 -  (\sigma + t - 2\tau_0)w_0}^{\frac{1}{1+2s}} \Big)\Big\}\log\big(\frac{\rho^{2s}}{k \delta(t)}\big)}.
\eeqs

\textit{Step 1-(i): The lower bound.}
For $(t, x, v) \in [\tau_0, \sigma] \times B_{(2\rho)^{1+2s}}(y_0 + (\sigma + t - 2\tau_0) w_0)\times B_{2\rho}(w_0)$ we find
\beq\label{eq:H-lowerbound}
	H(t, x, v) \geq \Bigg(\frac{\rho^{2s}}{k\delta(t)}\Bigg)^{-1} \geq\Bigg(\frac{\rho^{2s}}{k(\sigma - \tau_0)}\Bigg)^{-1}.
\eeq
We used $\delta(t) \geq \sigma -\tau_0$ and $\sigma - \tau_0 \leq \frac{\rho^{2s}}{4k}$.

\textit{Step 1-(ii): The upper bound.}
For $t = \tau_0$ and $(x, v)$ such that $$\max\big\{\abs{x-y_0 - (\sigma  - \tau_0) w_0}^{\frac{1}{1+2s}}, \abs{v-w_0}\big\} \geq \gamma$$ we find
\beq\label{eq:H-upperbound}
	H(\tau_0, x, v) \leq  \Bigg(\frac{\rho^{2s}}{k\delta(\tau_0)}\Bigg)^{-\frac{\gamma}{3\rho}}  \leq  \Bigg(\frac{\rho^{2s}}{2k(\sigma -\tau_0)}\Bigg)^{-\frac{\gamma}{3\rho}}.
\eeq
We used $\sigma - \tau_0 \leq \frac{\rho^{2s}}{4k}$.

\textit{Step 1-(iii): The general upper bound.}
Note that for $(t, x, v) \in [\tau_0, \sigma]\times \R^{2d}$ we have
\beq\label{eq:H-upperbound-gen}
	H(t, x,v) \leq \Bigg(\frac{\rho^{2s}}{k\delta(t)}\Bigg)^{-1}  \leq \Bigg(\frac{\rho^{2s}}{2k(\sigma - \tau_0)}\Bigg)^{-1}.
\eeq

\textit{Step 2: Assembly.}

For fixed $(y_0, w_0) \in \R^{2d}$ we denote $$Q^{0}_{2\rho}(t, y_0, w_0) := B_{(2\rho)^{1+2s}}(y_0 + (\sigma + t - 2\tau_0)w_0 ) \times B_{2\rho}(w_0),$$
to obtain with Proposition \ref{prop:aronson}, \eqref{eq:H-upperbound} and \eqref{eq:H-upperbound-gen}
\bals
	\sup_{t \in (\tau_0, \sigma)} &\int_{Q^0_{2\rho}(t, y_0, w_0)} f^2(t, x, v) H(t, x, v)  \dd x \dd v\\
	&\leq  \sup_{t \in (\tau_0, \sigma)} \int_{\R^{2d}} f^2(t, x, v) H(t, x, v)  \dd x \dd v\\
	&\leq \int_{\R^{2d}\setminus Q^0_{\gamma}(\tau_0, y_0, w_0)} f^2(\tau_0, x, v) H(\tau_0, x, v) \dd x \dd v \\
	&\quad +C \rho^{-2s}\Norm{H}_{L^\infty([\tau_0, \sigma]\times \R^{2d})} \norm{f}_{L^2([\tau_0, \sigma]\times \R^{2d})}^2\\
	&\leq  \Bigg(\frac{\rho^{2s}}{2k(\sigma -\tau_0)}\Bigg)^{-\frac{\gamma}{3\rho}}  \norm{f_0}^2_{L^2(\R^{2d})}+C \rho^{-2s}\left(\frac{\rho^{2s}}{2k(\sigma - \tau_0)}\right)^{-1}  \norm{f}_{L^2([\tau_0, \sigma]\times \R^{2d})}^2.
\eals	
This implies due to \eqref{eq:H-lowerbound}
\bals
	&\sup_{t \in (\tau_0, \sigma)} \int_{Q^0_{2\rho}(t, y_0, w_0)} f^2(t, x, v)  \dd x \dd v\\
	&\leq \Bigg(\frac{\rho^{2s}}{k(\sigma - \tau_0)}\Bigg)\Bigg(\frac{\rho^{2s}}{2k(\sigma -\tau_0)}\Bigg)^{-\frac{\gamma}{3\rho}}  \norm{f_0}^2_{L^2(\R^{2d})}\\
	&\quad+C\Bigg(\frac{\rho^{2s}}{k(\sigma - \tau_0)}\Bigg) \left(\frac{\rho^{2s}}{2k(\sigma - \tau_0)}\right)^{-1} \rho^{-2s}\norm{f}_{L^2([\tau_0, \sigma]\times \R^{2d})}^2\\
	&\leq\Bigg(\frac{\rho^{2s}}{k(\sigma -\tau_0)}\Bigg)^{1-\frac{\gamma}{3\rho}} \norm{f_0}^2_{L^2(\R^{2d})}+C \rho^{-2s} \norm{f}_{L^2([\tau_0, \sigma]\times \R^{2d})}^2.
\eals

Since we assume that \eqref{eq:1.1} admits a fundamental solution, we estimate with Young's convolution inequality and with \eqref{eq:normalisation-J}
\bals
	 \norm{f}_{L^2([\tau_0, \sigma]\times \R^{2d})}^2 &\leq  \norm{f_0 \ast_v J}_{L^2([\tau_0, \sigma]\times \R^{2d})}^2 \leq (\sigma - \tau_0) \norm{f_0}_{L^2(\R^{2d})}^2\norm{J}_{L^\infty([\tau_0, \sigma]; L^1(\R^{2d}))}^2 \\
	 &= (\sigma - \tau_0) \norm{f_0}_{L^2(\R^{2d})}^2,
\eals
so that 
\beq\label{eq:aux-pw-ub}
	\sup_{t \in (\tau_0, \sigma)} \int_{Q^0_{2\rho}(t, y_0, w_0)} f^2(t, x, v)  \dd x \dd v\leq \Bigg[\Bigg(\frac{\rho^{2s}}{k(\sigma -\tau_0)}\Bigg)^{1-\frac{\gamma}{3\rho}} + \rho^{-2s} (\sigma - \tau_0) \Bigg]\norm{f_0}_{L^2(\R^{2d})}^2.
\eeq

Then using Proposition \ref{prop:L2-Linfty-A} with $r^{2s} = \frac{\sigma - \tau_0}{2^{2s}}$, $R = 2r$ for $\rho \geq r$ and \eqref{eq:aux-pw-ub}
\bals
	\Abs{f(\sigma, y_0, w_0)}^2 &\leq \sup_{Q_r(\sigma, y_0, w_0)} f^2\\
	&\leq C r^{-(2d(1+s) +2s)} \int_{Q_{2r}(\sigma, y_0, w_0)} f^2 \dd z\\
	&\leq C r^{- (2d(1+s) +2s) +2s } \sup_{t \in (\tau_0, \sigma)} \int_{Q^0_{2\rho}(t, y_0, w_0)}  f^2(t, x, v)  \dd x \dd v\\
	&\leq  C (\sigma - \tau_0)^{-\frac{2d(1+s)}{2s}} \Bigg(\frac{\rho^{2s}}{k(\sigma - \tau_0)}\Bigg)^{1- \frac{\gamma}{3\rho}}  \norm{f_0}^2_{L^2(\R^{2d})} \\
	&\quad+ \rho^{-2s} (\sigma - \tau_0)^{1-\frac{2d(1+s)}{2s}} \norm{f_0}_{L^2(\R^{2d})}^2.
\eals
This concludes the proof of Theorem \ref{thm:aronson-bound}. 
\end{proof}

\subsection{Proof of the polynomial decay}

With Theorem \ref{thm:nash} and Theorem \ref{thm:off-diag} at hand, we use standard methods to derive the decay on the fundamental solution.

\begin{proof}[Proof of Theorem \ref{thm:upperbounds}]
Due to the on-diagonal bounds \eqref{eq:ondiag}, it suffices to consider $(\sigma - \tau_0) \leq \frac{\rho^{2s}}{4k}$.
We choose 
\beq\label{eq:rho}
	\rho = \frac{\gamma}{12}\left( \frac{1}{4} + \frac{2d(1+s) +2s}{2s}\right)^{-1}.
\eeq
If we abbreviate $\abs{X-Y_0} := \abs{x - y_0 - (\sigma - \tau_0) (v-w_0)}$, then the off-diagonal bounds from Theorem \ref{thm:off-diag} yield
\bals
	J(\sigma,x, v; \tau_0, y_0, w_0)  \lesssim  &(\sigma - \tau_0) \Big(\max\left\{\abs{v- w_0},  \abs{X-Y_0}^{\frac{1}{1+2s}}\right\}\Big)^{-(2d(1+s) +2s)}\\
	&+ (\sigma - \tau_0)^{\frac{1}{4}-\frac{2d(1+s)}{2s}}\left( \max\left\{\abs{v- w_0}, \abs{X-Y_0}^{\frac{1}{1+2s}}\right\}\right)^{-\frac{s}{2}},
\eals
which concludes the proof of Theorem \ref{thm:upperbounds}.
\qedhere
\end{proof}

\begin{proof}[Proof of Theorem \ref{thm:upperbounds-conditional}]
Again, due to the on-diagonal bounds \eqref{eq:ondiag}, it suffices to consider $(\sigma - \tau_0) \leq \frac{\rho^{2s}}{4k}$. Due to \eqref{eq:meyers}, it suffices to use Aronson's method on $J_\rho$. In particular, the term arising from the non-singular part in Proposition \ref{prop:aronson} vanishes, that is $\mathcal I_{loc}^{ns} = 0$ in \eqref{eq:I-loc}. This implies that in Theorem \ref{thm:aronson-bound} only the first term on the right hand side remains, from which we deduce Theorem \ref{thm:off-diag} on $J_\rho$ with an improved right hand side:
\beqs
	J_\rho(\sigma,x, v; \tau_0, y_0, w_0) \leq C  (\sigma - \tau_0)^{-\frac{2d(1+s)}{2s}}\Bigg(\frac{\rho^{2s}}{k(\sigma - \tau_0)}\Bigg)^{\frac{1}{4}- \frac{\gamma}{12\rho}}.
\eeqs
If we choose $\rho$ as above in \eqref{eq:rho} then together with \eqref{eq:meyers} this yields for $(\sigma - \tau_0) \leq \frac{\rho^{2s}}{4k}$
\bals
	J(\sigma,x, v; \tau_0, y_0, w_0)  \lesssim  (\sigma - \tau_0) \max\left\{\abs{v- w_0},  \abs{X-Y_0}^{\frac{1}{1+2s}}\right\}\Big)^{-(2d(1+s) +2s)},
\eals
which concludes the proof of Theorem \ref{thm:upperbounds-conditional}.
\qedhere
\end{proof}

\section{Lower bound on the fundamental solution}\label{sec:lowerbound}

\subsection{Pointwise Harnack inequality}

\begin{theorem}\label{thm:pw-harnack}
Let \(s\in(\frac12,1)\), and let \(f\) be a non-negative solution of
\eqref{eq:1.1} in \([0,T]\times\R^{2d}\), with zero source term and
with a non-negative kernel satisfying
\eqref{eq:coercivity}--\eqref{eq:symmetry-nondiv}.
Let $0<\tau_0<t\leq T,
(y_0,w_0),(x,v)\in\R^{2d}.$
Then there exists a constant $C$ depending on $s, d, \lambda_0, \Lambda_0$ and $T$ such that
\begin{equation}
\label{eq:pw-H-equi}
\begin{aligned}
f(\tau_0,y_0,w_0)
\leq{}&
f(t,x,v)
\exp\Bigg[
C\Bigg(
1+\frac{t-\tau_0}{\tau_0}+
\left(
\frac{\abs{v-w_0}}{(t-\tau_0)^{\frac1{2s}}}
+
\frac{\abs{x-y_0-(t-\tau_0)w_0}}
     {(t-\tau_0)^{1+\frac1{2s}}}
\right)^{\frac{2s}{2s-1}}
\Bigg)
\Bigg].
\end{aligned}
\end{equation}
\end{theorem}
\begin{proof}

By scaling and translating Theorem \ref{thm:strong-Harnack}, there
exist constants
\[
c_{\rm H}\in(0,1),\qquad C_{\rm H}>1,
\]
depending only on \(s,d,\lambda_0,\Lambda_0\), with the following
property. Let \(\delta>0\), and suppose that the kinetic cylinder
required below is contained in the domain of the equation. If
\[
\abs{v_2-v_1}
\leq c_{\rm H}\delta^{\frac1{2s}}
\]
and
\[
\abs{x_2-x_1-\delta v_1}
\leq c_{\rm H}\delta^{1+\frac1{2s}},
\]
then
\begin{equation}\label{e.one.link}
f(t_1,x_1,v_1)
\leq
C_{\rm H}f(t_1+\delta,x_2,v_2).
\end{equation}

Let
\[
D:=x-y_0-(t-\tau_0) w_0,\qquad
B:=\frac{D}{t-\tau_0}-\frac{(v-w_0)}{2}.
\]
For \(a\in[0,1]\), set
\[
V(a):=w_0+a(v-w_0)+6a(1-a)B
\]
and
\[
X(a):=y_0+(t-\tau_0)\int_0^aV(b)\,\dd b.
\]
Then
\[
V(0)=w_0,\qquad V(1)=v,\qquad
X(0)=y_0,\qquad X(1)=x,
\]
because
\[
\int_0^1V(a)\,\dd a
=w_0+\frac{(v-w_0)}{2}+B
=\frac{x-y_0}{(t-\tau_0)}.
\]
Moreover,
\[
X'(a)=(t-\tau_0) V(a)
\]
and
\[
\norm{V'}_{L^\infty(0,1)}
\leq
C\left(\abs{(v-w_0)}+\frac{\abs{D}}{(t-\tau_0)}\right).
\]

For \(\nu=0,\dots,N\), define
\[
t^\nu:=\tau_0+\frac{\nu}{N}(t-\tau_0),\qquad
x^\nu:=X\left(\frac{\nu}{N}\right),\qquad
v^\nu:=V\left(\frac{\nu}{N}\right).
\]
Thus
\[
(t^0,x^0,v^0)=(\tau_0,y_0,w_0),
\qquad
(t^N,x^N,v^N)=(t,x,v).
\]

Set
\[
\mathcal D
:=
\frac{\abs{v-w_0}}{(t-\tau_0)^{1/(2s)}}
+
\frac{\abs{x-y_0-(t-\tau_0) w_0}}
     {(t-\tau_0)^{1+1/(2s)}}.
\]
For every \(\nu=0,\dots,N-1\), the mean-value theorem gives
\[
\abs{v^{\nu+1}-v^\nu}
\leq
C(t-\tau_0)^{1/(2s)}\frac{\mathcal D}{N}.
\]
Furthermore, since \(X'=(t-\tau_0) V\),
\begin{align*}
&x^{\nu+1}-x^\nu
 -(t^{\nu+1}-t^\nu)v^\nu=
(t-\tau_0)\int_{\nu/N}^{(\nu+1)/N}
\left[
V(a)-V\left(\frac{\nu}{N}\right)
\right]\dd a,
\end{align*}
and hence
\[
\abs{x^{\nu+1}-x^\nu
 -(t^{\nu+1}-t^\nu)v^\nu}
\leq
C(t-\tau_0)^{1+1/(2s)}
\frac{\mathcal D}{N^2}.
\]

We choose
\[
N
=
\left\lceil
C_0\left(
1+\frac{t-\tau_0}{\tau_0}
+\mathcal D^{\frac{2s}{2s-1}}
\right)
\right\rceil,
\]
where \(C_0\) is sufficiently large depending on the one-link Harnack geometry.
Since \(s>\frac12\), choosing \(C_0\) sufficiently large ensures
\[
\abs{v^{\nu+1}-v^\nu}
\leq
c_{\rm H}\left(\frac{t-\tau_0}{N}\right)^{1/(2s)}
\]
and
\[
\abs{x^{\nu+1}-x^\nu
 -(t^{\nu+1}-t^\nu)v^\nu}
\leq
c_{\rm H} \left(\frac{(t-\tau_0)}{N}\right)^{\frac{1+2s}{2s}}.
\]
Indeed, the choice of \(N\) gives
\[
\mathcal D
\leq
C N^{\frac{2s-1}{2s}}.
\]
Therefore
\[
\frac{\mathcal D}{N}
\leq
C N^{-\frac1{2s}}, \qquad 
\frac{\mathcal D}{N^2}
\leq
C N^{-1-\frac1{2s}}.
\]
Moreover,
\[
N\geq C_0\frac{t-\tau_0}{\tau_0},
\]
and hence
\[
\frac{t-\tau_0}{N}\leq \frac{\tau_0}{C_0}.
\]
Thus, by choosing \(C_0\) sufficiently large, all the kinetic
cylinders used in the chain are contained in
\([0,T]\times\R^{2d}\).

The one-link inequality \eqref{e.one.link} therefore gives
\[
f(t^\nu,x^\nu,v^\nu)
\leq
C_{\rm H}
f(t^{\nu+1},x^{\nu+1},v^{\nu+1})
\]
for every \(\nu=0,\dots,N-1\). Iterating,
\[
f(\tau_0,y_0,w_0)
\leq
C_{\rm H}^Nf(t,x,v).
\]
Finally,
\[
C_{\rm H}^N
=
\exp(N\log C_{\rm H})
\]
and, by the definition of \(N\),
\[
N
\leq
C\left(
1+\frac{t-\tau_0}{\tau_0}
+\mathcal D^{\frac{2s}{2s-1}}
\right).
\]
Thus
\begin{align*}
f(\tau_0,y_0,w_0)
\leq{}&
f(t,x,v)
\exp\Bigg[
C\Bigg(
1+\frac{t-\tau_0}{\tau_0}+
\left(
\frac{\abs{v-w_0}}{(t-\tau_0)^{\frac1{2s}}}
+
\frac{\abs{x-y_0-(t-\tau_0)w_0}}
     {(t-\tau_0)^{1+\frac1{2s}}}
\right)^{\frac{2s}{2s-1}}
\Bigg)
\Bigg].
\end{align*}
\end{proof}

\subsection{Sufficient condition for the exponential lower bound}
\begin{theorem}\label{thm:suff-cond}
Let \(s\in(\frac12,1)\).
Let $f$ be a non-negative weak solution of \eqref{eq:1.1} in $[0, T] \times \R^{2d}$ with \eqref{eq:coercivity}-\eqref{eq:symmetry-nondiv} that is essentially bounded almost everywhere. Suppose that for some $\alpha > 0$ there holds
\beq\label{eq:suff-cond}
	\mathcal M := \inf_{0 < t < T} \int_{\max\left\{ \abs{x}^{\frac{2s}{1+2s}}, \abs{v }^{2s} \right\} < \alpha t} f(t, x, v) \dd x \dd v > 0.
\eeq
Then there exists constants $C_1, C_2 > 0$ such that 
\[
f(t,x,v)
\geq
C_1t^{-\frac{d(1+s)}s}
\exp\left[
-C_2\left(
1+
\left(
\frac{|v|}{t^{1/(2s)}}
+
\frac{|x|}{t^{1+1/(2s)}}
\right)^{\frac{2s}{2s-1}}
\right)
\right].
\]
The constant $C_1$ depends on $s, d, \alpha, T, \lambda_0, \Lambda_0$ and $\mathcal M$, whereas $C_2$ depends only on $s, d, T, \lambda_0, \Lambda_0$. 
\end{theorem}
\begin{remark}
To draw a connection to the literature on hypoelliptic lower bounds, there has been a result discussing Gaussian lower bounds for solutions to the Boltzmann equation by Imbert-Mouhot-Silvestre in \cite{IMS}. The authors show that for any $t \in [0, T]$ there exists $a(t), b(t) > 0$ such that any non-negative solution of the Boltzmann equation satisfies
\beqs
	f(t, x, v) \geq a(t) e^{-b(t) \abs{v}^2}.
\eeqs
We explicitly give expressions for $a$ and $b$, however, we are only able to treat symmetric kernels due to the tail bound of Section \ref{sec:tail-bound}.
\end{remark}
\begin{proof}
Consider \(y,w\in\R^d\) such that
\[
\max\left\{
    \abs{y}^{\frac{2s}{1+2s}},
    \abs{w}^{2s}
\right\}
<\frac{\alpha t}{4}.
\]
By Theorem \ref{thm:pw-harnack}, there exists
\(C_\alpha>0\), depending only on
\(s,d,\lambda_0,\Lambda_0,T\) and \(\alpha\), such that
\begin{equation}\label{eq:suff-cond-aux}
f(t/4,y,w)
\leq
C_\alpha f(t/2,0,0).
\end{equation}

Moreover,
\[
\left|
\left\{(y,w)\in\R^{2d}:
\max\left\{
    \abs{y}^{\frac{2s}{1+2s}},
    \abs{w}^{2s}
\right\}
<\frac{\alpha t}{4}
\right\}
\right|
\approx
t^{\frac{2d(1+s)}{2s}},
\]
where the implicit constants may depend on \(\alpha,d,s\).
Integrating \eqref{eq:suff-cond-aux} over this set and using
\eqref{eq:suff-cond} at time \(t/4\), we obtain
\bals
f(t/2,0,0)
&\geq
C t^{-\frac{2d(1+s)}{2s}}
C_\alpha^{-1}
\int_{\max\left\{
    \abs{y}^{\frac{2s}{1+2s}},
    \abs{w}^{2s}
\right\}<\frac{\alpha t}{4}}
f(t/4,y,w)\,\dd y\,\dd w
\\
&\geq
C\mathcal M
t^{-\frac{2d(1+s)}{2s}}
C_\alpha^{-1}.
\eals

On the other hand, Theorem \ref{thm:pw-harnack}, applied between
\((t/2,0,0)\) and \((t,x,v)\), gives
\[
f(t/2,0,0)
\leq
f(t,x,v)
\exp\left[
C\left(
1+
\left(
\frac{|v|}{t^{1/(2s)}}
+
\frac{|x|}{t^{1+1/(2s)}}
\right)^{\frac{2s}{2s-1}}
\right)
\right].
\]

Combining the preceding two estimates, we obtain
\[
f(t,x,v)
\geq
c\,\mathcal M\,t^{-\frac{d(1+s)}s}
\exp\left[
-C\left(
\frac{|v|}{t^{1/(2s)}}
+
\frac{|x|}{t^{1+1/(2s)}}
\right)^{\frac{2s}{2s-1}}
\right],
\]
after absorbing all constants depending on \(\alpha\) into \(c\).
This proves the theorem.
\end{proof}

\subsection{Proof of Theorem \ref{thm:lowerbounds}}

\begin{proof}[Proof of Theorem \ref{thm:lowerbounds}]
We fix \(0<\tau_0<\sigma<T\). We first establish a uniform lower bound on the mass of the
fundamental solution in a kinetic ball centred on the characteristic
starting from \((y_0,w_0)\).

\medskip
\noindent
\textit{Step 1: Uniform near-diagonal mass bound.}

For \(0<t<T-\tau_0\), we define
\[
A_t:=
\left\{(x,v)\in\R^{2d}:
\max\left\{
    \abs{x-y_0-tw_0}^{\frac{2s}{1+2s}},
    \abs{v-w_0}^{2s}
\right\}<\alpha t
\right\},
\]
where \(\alpha>0\) will be chosen below. For
\(\tau_0\leq \tau<\tau_0+t\), we let
\[
F_t(\tau,y,w)
:=
\int_{A_t}
J(\tau_0+t,x,v;\tau,y,w)\,\dd x\,\dd v.
\]
By the non-negativity and normalisation of the fundamental solution,  $0\leq F_t\leq 1$.
Moreover, as a function of the initial variables
\((\tau,y,w)\), \(F_t\) solves the backward equation
\[
    \partial_\tau F_t
    +w\cdot\nabla_yF_t
    +\mathcal L_\tau F_t=0
\]
in \((\tau_0,\tau_0+t)\times\R^{2d}\), and its terminal value is
\[
    F_t(\tau_0+t,y,w)=\mathbbm 1_{A_t}(y,w).
\]

We reverse time and velocity by setting
\[
    G_t(r,y,\omega)
    :=
    F_t(\tau_0+t-r,y,-\omega),
    \qquad 0<r\leq t.
\]
Then \(G_t\) solves a forward kinetic equation
\[
    \partial_rG_t+\omega\cdot\nabla_yG_t
    =\widetilde{\mathcal L}_rG_t,
\]
where the transformed kernel is obtained from \(K\) by time reversal
and velocity reflection. In particular, it satisfies the same
coercivity, integral upper-bound and symmetry assumptions, with the
same structural constants. The initial value of \(G_t\) is
\[
    G_t(0,y,\omega)
    =
    \mathbbm 1_{A_t}(y,-\omega).
\]
This function equals one in a kinetic neighbourhood of
\[
    (y_0+tw_0,-w_0).
\]
The characteristic for the reversed equation starting
from this point reaches \((y_0,-w_0)\) at time \(t\), since
\[
    y_0+tw_0+t(-w_0)=y_0.
\]

We then choose a universal \(c_0>0\) sufficiently small and
\(\alpha>0\) sufficiently large that the kinetic cylinder of radius
\[
    r_t:=c_0t^{\frac1{2s}}
\]
used below has its section at time \(r=0\) compactly contained in
\[
    \left\{(y,\omega):(y,-\omega)\in A_t\right\}.
\]
We extend \(G_t\) locally to negative times by setting it equal to one
there. Since its trace at \(r=0\) is also equal to one on the
spatial-velocity section of this cylinder, the extension introduces
no jump in the cylinder. It is therefore a non-negative
super-solution there.

After translating the cylinder and rescaling it by \(r_t\) to a
fixed unit kinetic cylinder, the past cylinder lies in the region
where the extended function is identically one, while the
corresponding future cylinder contains
\((t,y_0,-w_0)\). The scaled Weak Harnack inequality,
Theorem \ref{thm:weakH}, therefore gives
\[
    1
    \leq
    C\inf_{\widetilde Q^+}G_t
    \leq
    C G_t(t,y_0,-w_0),
\]
where \(C\) depends only on
\(d,s,\lambda_0,\Lambda_0\) and the fixed cylinder geometry.
Consequently,
\[
    F_t(\tau_0,y_0,w_0)
    =
    G_t(t,y_0,-w_0)
    \geq c
\]
for a constant \(c>0\) independent of \(t\). Hence
\begin{equation}\label{eq:uniform-local-mass}
\inf_{0<t<T-\tau_0}
\int_{A_t}
J(\tau_0+t,x,v;\tau_0,y_0,w_0)\,\dd x\,\dd v
\geq c>0.
\end{equation}

\medskip
\noindent
\textit{Step 2: Reduction to unit time.}

For \(0<\theta<T_*:=\frac{T-\tau_0}{\sigma-\tau_0},\)
we define the rescaled fundamental solution
\[
\begin{aligned}
U(\theta,X,V)
:={}&
(\sigma-\tau_0)^{\frac{2d(1+s)}{2s}}
J\Big(
    \tau_0+(\sigma-\tau_0)\theta,\,
    y_0+(\sigma-\tau_0)\theta w_0
       +(\sigma-\tau_0)^{\frac{1+2s}{2s}}X,\,
    w_0+(\sigma-\tau_0)^{\frac1{2s}}V;\,
    \tau_0,y_0,w_0
\Big).
\end{aligned}
\]
By the scaling and Galilean invariance of the equation, \(U\) is a
non-negative solution of an equation of the form \eqref{eq:1.1}
whose kernel satisfies the same structural assumptions.

For every \(0<\theta<1\), the change of variables
\[
    x=y_0+(\sigma-\tau_0)\theta w_0
      +(\sigma-\tau_0)^{\frac{1+2s}{2s}}X,
    \qquad
    v=w_0+(\sigma-\tau_0)^{\frac1{2s}}V
\]
and \eqref{eq:uniform-local-mass} give
\begin{align*}
&\int_{\max\left\{
       \abs{X}^{\frac{2s}{1+2s}},
       \abs{V}^{2s}
     \right\}<\alpha\theta}
U(\theta,X,V)\,\dd X\,\dd V
\\
&\qquad =
\int_{\max\left\{
       \abs{x-y_0-(\sigma-\tau_0)\theta w_0}^{\frac{2s}{1+2s}},
       \abs{v-w_0}^{2s}
     \right\}<\alpha(\sigma-\tau_0)\theta}
J(\tau_0+(\sigma-\tau_0)\theta,x,v;
  \tau_0,y_0,w_0)\,\dd x\,\dd v
\\
&\qquad \geq c.
\end{align*}
Therefore
\[
\inf_{0<\theta<1}
\int_{\max\left\{
       \abs{X}^{\frac{2s}{1+2s}},
       \abs{V}^{2s}
     \right\}<\alpha\theta}
U(\theta,X,V)\,\dd X\,\dd V
\geq c>0.
\]

The Nash upper bound, Theorem \ref{thm:nash}, shows that \(U\) is
locally essentially bounded for every positive time. Thus the proof
of Theorem \ref{thm:suff-cond}, which only uses positive times,
applies to \(U\). We consequently obtain
\[
U(1,X,V)
\geq
c
\exp\left[
-C
\left(
|V|+|X|
\right)^{\frac{2s}{2s-1}}
\right].
\]

\medskip
\noindent
\textit{Step 3: Return to the original variables.}

Taking
\[
    X=
    \frac{x-y_0-(\sigma - \tau_0) w_0}
         {(\sigma-\tau_0)^{\frac{1+2s}{2s}}},
    \qquad
    V=
    \frac{v-w_0}{(\sigma-\tau_0)^{\frac1{2s}}},
\]
we conclude that
\begin{align*}
J(\sigma,x,v;\tau_0,y_0,w_0)
\geq{}&
c(\sigma-\tau_0)^{-\frac{d(1+s)}s}\\
&\times
\exp\left[
-C\left(
\frac{|v-w_0|}{(\sigma-\tau_0)^{1/(2s)}}
+
\frac{|x-y_0-(\sigma-\tau_0)w_0|}
     {(\sigma-\tau_0)^{1+1/(2s)}}
\right)^{\frac{2s}{2s-1}}
\right],
\end{align*}
which proves Theorem \ref{thm:lowerbounds}.
\end{proof}

\textbf{Acknowledgements.}

We thank Clément Mouhot for his inspiring intuition that he transmits in every discussion we have, for knowing exactly how to prevent us from turning in circles and for his continuous support.
Moreover, we thank Marvin Weidner and Moritz Kassmann for pointing out an issue in the sketch of the proof in our proceeding, as well as for their useful feedback on the first version of this manuscript; in particular, for their comments on the decay estimate. We also thank Lukas Niebel for his encouraging remarks. 
Finally, we gratefully acknowledge Markus Reiß for his interest in our work and his suggestions.
This work was supported by the Cambridge Trust.

\appendix

\section{Higher regularity estimate}\label{sec:approx}

The purpose of this appendix is to clarify the use of a test function with less differentiability than what is usually required. This is necessary to justify the construction of the test function in the argument of Proposition \ref{prop:tail-bound} in case that $s \geq \frac{1}{2}$. 

One can show that solutions to the fractional Laplacian of order $2s \in (0,2)$ are in fact $H^{2s-}$, rather than $H^s$ as expected from an energy estimate. Therefore the equation can in principle be tested against a function in $H^{0+}$. The test function we work with in Proposition \ref{prop:tail-bound} is in $H^{\frac{s}{2}}$ for any $s \in (0, 1)$; thus, if we can show that the solution to \eqref{eq:1.1} satisfies better regularity, say $H^{\frac{3s}{2}}$, then this test function is indeed permissible. 

To this end, we compare our solutions to the solution of the fractional Laplacian, so that we can quantify the loss of regularity from the roughening of the coefficient. We show that we can work with a solution that satisfies the higher differentiability that is required in order for the test function in \eqref{eq:psi-l} to be admissible. In the limit, where the smoothening of the fractional Laplacian is lost, this higher regularity vanishes as well. However, we can run the argument of Proposition \ref{prop:tail-bound} on the solutions before we take the limit, where the admissibility of the test function is guaranteed. Since the higher regularity is not used quantitatively, but merely qualitatively, we end up with the result of Proposition \ref{prop:tail-bound} uniformly in the approximation parameter. We can then take the limit in the end, so that the results apply to solutions of \eqref{eq:1.1}. The equation is linear, so that the solution is unique, and the limit is uniquely determined through the limit equation, which corresponds to the equation with rough coefficients.



Let $\eta: \R^d \to \R_+$ be a cutoff, such that $0 \leq \eta \leq 1$ with $\eta =1$ in $B_1$, $\eta = 0$ outside $B_2$ and $\eta(v) = - \abs{v} + 2$ for $v \in B_2 \setminus B_1$. Denote by $\eta_\delta$ the cutoff $\eta_\delta(v) = \eta\big(\frac{v}{\delta}\big) \in C_c^\infty(B_{2\delta})$ satisfying $0 \leq \eta_\delta \leq 1$ with $\eta_\delta = 1$ in $B_{\delta}$ and $\eta_\delta = 0$ outside $B_{2\delta}$. Then we consider
\[
	K_{\delta}(v, w) =  \eta_\delta(v-w) c_{d, s}\abs{v-w}^{-(d+2s)} + \big(1-\eta_\delta(v-w)\big) K(v, w),
\]
where $c_{d} > 0$ is the constant such that
\[
	(-\Delta_v)^{s} f(v) = c_{d, s}\int_{\R^d} \big(f(w) - f(v)\big)  \abs{v-w}^{-(d+2s)} \dd w.
\]
Note that $\eta_\delta \xrightarrow[\delta \to 0]{} 0,$
pointwise almost everywhere and in $L^p$ for any $1 \leq p <\infty$.
Now let $f_\delta$ solve 
\beq\label{eq:eq-delta}
	\mathcal T f_\delta = \mathcal L_{\delta} f_\delta := \int_{\R^d} \big(f_\delta(w) - f_\delta(v)\big) K_{\delta}(v, w) \dd w, 
\eeq
in $[t_1, t_2] \times B_{R^{1+2s}} \times \R^d$. 
Note $$f_\delta \in L^2((t_1, t_2) \times B_{R^{1+2s}}; H^s(\R^d))$$ uniformly in $\delta$. To see this, we test \eqref{eq:eq-delta} with $f_\delta \phi^2$ where $\phi \in C_c^\infty([t_1, t_2] \times B_{R^{1+2s}})$ is a cutoff in time and space such that $\phi = 1$ in $\big[\frac{t_1}{2}, \frac{t_2}{2}\big] \times B_{\big(\frac{R}{2}\big)^{1+2s}}$ and $\phi = 0$ outside $[t_1, t_2] \times B_{R^{1+2s}}$. We get
\bals
	&\int_{[t_1, t_2] \times B_{R^{1+2s}} \times \R^d} \mathcal T f_\delta f_\delta \phi^2 \dd z \\
	&= \int_{[t_1, t_2] \times B_{R^{1+2s}}} \int_{\R^d} \int_{\R^d} \big(f_\delta(w) - f_\delta(v)\big) f_\delta(v)\phi^2(t, x) K_\delta(v, w) \dd w\dd v\dd x \dd t\\
	&= -\frac{1}{2}\int_{[t_1, t_2] \times B_{R^{1+2s}}}\phi^2(t, x) \int_{\R^d} \int_{\R^d} \big(f_\delta(w) - f_\delta(v)\big)^2 K_\delta(v, w) \dd w\dd v\dd x \dd t.\\
	&= -\frac{ c_d}{2}\int_{[t_1, t_2] \times B_{R^{1+2s}}}\phi^2(t, x) \int_{\R^d} \int_{\R^d}\eta_\delta(v-w) \frac{\big(f_\delta(w) - f_\delta(v)\big)^2}{\abs{v-w}^{d+2s}} \dd w \dd z
	\\&\quad-\frac{1}{2}\int_{[t_1, t_2] \times B_{R^{1+2s}}}\phi^2(t, x) \int_{\R^d} \int_{\R^d}(1-\eta_\delta(v-w)) \big(f_\delta(w) - f_\delta(v)\big)^2K(v, w)   \dd w\dd z.
\eals
Using that $(1-\eta_\delta) \geq \frac{1}{4}$ outside $B_{\frac{5\delta}{4}}$ and $\eta_\delta \geq \frac{1}{4}$ inside $B_{\frac{7\delta}{4}}$, we find with the coercivity of $K$ \eqref{eq:coercivity}
\bals
	&\int_{[t_1, t_2] \times B_{R^{1+2s}} \times \R^d} \mathcal T f_\delta f_\delta \phi^2 \dd z \\
	&\leq -\frac{1}{4}\frac{ c_d}{2}\int_{[t_1, t_2] \times B_{R^{1+2s}}}\phi^2(t, x) \int_{\R^d} \int_{\abs{v-w} < \frac{7\delta}{4}} \frac{\big(f_\delta(w) - f_\delta(v)\big)^2}{\abs{v-w}^{d+2s}}\dd w \dd z
	\\&\quad-\frac{1}{4}\cdot\frac{ 1}{2}\int_{[t_1, t_2] \times B_{R^{1+2s}}}\phi^2(t, x) \int_{\R^d} \int_{\abs{v-w} > \frac{5\delta}{4}} \big(f_\delta(w) - f_\delta(v)\big)^2K(v, w)   \dd w\dd z\\ 
	&\leq -\frac{1}{4}\frac{ c_d}{2}\int_{[t_1, t_2] \times B_{R^{1+2s}}}\phi^2(t, x) \int_{\R^d} \int_{\abs{v-w} < \frac{7\delta}{4}} \frac{\big(f_\delta(w) - f_\delta(v)\big)^2}{\abs{v-w}^{d+2s}}\dd w \dd z
	\\&\quad-\frac{1}{4}\frac{ \lambda}{2}\int_{[t_1, t_2] \times B_{R^{1+2s}}}\phi^2(t, x) \int_{\R^d} \int_{\abs{v-w} > \frac{5\delta}{4}} \frac{\big(f_\delta(w) - f_\delta(v)\big)^2}{\abs{v-w}^{d+2s}} \dd w\dd z.
\eals
Thus 
\bals
	\sup_{t\in[t_1, t_2]} &\int_{B_{(\frac{R}{2})^{1+2s}} \times \R^d} f_\delta^2(t, x, v) \dd x \dd v\\
	 &+ C_{d, s, \lambda} \int_{\big[\frac{t_1}{2}, \frac{t_2}{2}\big] \times B_{(\frac{R}{2})^{1+2s}}} \int_{\R^d} \int_{\R^d} \frac{\big(f_\delta(w) - f_\delta(v)\big)^2}{\abs{v-w}^{d+2s}}  \dd w\dd z \\
	&\quad\leq - \frac{1}{2} \int_{[t_1, t_2] \times B_{R^{1+2s}} \times \R^d}  v \cdot \nabla_x \big(f_\delta^2\big)  \phi^2 \dd z\\
	& \quad\leq  CR^{-2s} \int_{[t_1, t_2] \times B_{R^{1+2s}} \times \R^d}  f_\delta^2  \dd z.
\eals
In particular, $C_{d, s, \lambda}$ and $C$ are uniform constants in $\delta$, thus $f_\delta \in L^2((t_1, t_2) \times B_{R^{1+2s}}; H^s(\R^d))$ uniformly in $\delta$.

Moreover, we observe that $f_\delta$ satisfies
\[
	\mathcal T f_\delta = (-\Delta_v)^{s} f_\delta + \int_{\abs{v-w} > \delta} \big(f_\delta(w) -f_\delta(v)\big) \big(1-\eta_\delta(v-w)\big) \big(K(v, w) - c_{d, s}\abs{v-w}^{-(d+2s)}\big)\dd w
\]
in $[t_1, t_2] \times B_{R^{1+2s}} \times \R^d$. This is the Kolmogorov equation with bounded right hand side
\[
	h_\delta(t, x, v) :=  \int_{\abs{v-w} > \delta} \big(f_\delta(w) -f_\delta(v)\big) \big(1-\eta_\delta(v-w)\big) \big(K(v, w) - c_{d, s}\abs{v-w}^{-(d+2s)}\big)\dd w.
\]
Thus $f_\delta$ is given by 
\[
	f_\delta(t, x, v) = h_\delta \ast J(t, x,v) + f_\delta(t_1, x, v) \ast J(t_1, x, v),
\]
where $J$ is the fundamental solution of the Kolmogorov equation in $[t_1, t_2] \times B_{R^{1+2s}} \times \R^d$.
We find on the one hand, for any $2 \leq p < \frac{2d(s+1) + 2s}{d(s+1)}$, if we let $\frac{1}{q} = \frac{1}{2} - \frac{1}{p}$, then by Young's inequality, the scaling of the fundamental solution \eqref{eq:fund-sol-lebesgue} and Theorem \ref{thm:boundedness-Hs}
\bals
	&\norm{f_{\delta}}_{L^p([t_1, t_2] \times B_{R^{1+2s}} \times \R^d)}\\
	 &\quad\leq \norm{(-\Delta_v)^{-\frac{s}{2}}h_{\delta}}_{L^2([t_1, t_2] \times B_{R^{1+2s}} \times \R^d)}\norm{(-\Delta_v)^{\frac{s}{2}}J}_{L^q([t_1, t_2] \times B_{R^{1+2s}} \times \R^d)} \\
	&\quad\leq C R^{-s} \norm{f_{\delta}}_{L^2([t_1, t_2] \times B_{R^{1+2s}}; \dot H^s_v(\R^d))}	\\
	&\quad\leq C R^{-2s}\norm{f_{\delta}}_{L^2([t_1, t_2] \times B_{R^{1+2s}} \times \R^d))},
\eals
uniformly in $\delta$. 

On the other hand, since $\norm{h_\delta}_{L^2} \leq C \delta^{-2s}\norm{f_{\delta}}_{L^2}$, we find, due to \eqref{eq:fund-sol-lebesgue}, $f_\delta \in L^1([t_1, t_2]);L^2(B_{R^{1+2s}});  H^{\frac{3s}{2}}(\R^d))$ with a constant that degenerates as $\delta \to 0$. This follows by using Young's inequality for convolutions:
\bals
	&\norm{f_\delta}_{ L^1([t_1, t_2]);L^2(B_{R^{1+2s}}); \dot H^{\frac{3s}{2}}(\R^d))}\\
	&\leq \norm{h_\delta \ast J}_{ L^1([t_1, t_2]);L^2(B_{R^{1+2s}}); \dot H^{\frac{3s}{2}}(\R^d))} + \norm{f_\delta(t_1)\ast J}_{ L^1([t_1, t_2]);L^2(B_{R^{1+2s}}); \dot H^{\frac{3s}{2}}(\R^d))} \\
	&\leq \norm{h_\delta \ast (-\Delta_v)^{\frac{3s}{4}} J}_{L^1([t_1, t_2]);L^2(B_{R^{1+2s}}\times\R^d)} + \norm{f_\delta(t_1)\ast (-\Delta_v)^{\frac{3s}{4}} J}_{L^1([t_1, t_2]);L^2(B_{R^{1+2s}}\times\R^d)} \\
	&\leq  C \norm{h_\delta \ast t^{-\frac{3}{4}}}_{L^1([t_1, t_2]);L^2(B_{R^{1+2s}}\times\R^d)}+  C \norm{f_\delta(t_1) }_{L^2(B_{R^{1+2s}}\times\R^d)}\\
	& \leq C \norm{h_\delta}_{L^2([t_1, t_2]\times B_{R^{1+2s}}\times\R^d)}+  C \norm{f_\delta(t_1) }_{L^2(B_{R^{1+2s}}\times\R^d)}\\
	& \leq C \delta^{-2s} \norm{f_\delta}_{L^2([t_1, t_2]\times B_{R^{1+2s}}\times\R^d)}+  C \norm{f_\delta(t_1) }_{L^2(B_{R^{1+2s}}\times\R^d)}.
\eals
Thus we can use any test function in $L^\infty_{t} L^2_x  H_v^{\frac{s}{2}}$ to give sense to the equation; in particular $\psi_l$ constructed in \eqref{eq:psi-l} is admissible. We deduce the a priori estimate on $f_\delta$, but since the $\dot H_v^{\frac{3s}{2}}$ norm of $f_\delta$ is not quantitatively affecting the estimate of Proposition \ref{prop:tail-bound}, the final estimate is uniform in $\delta$. 
We find:
\bal\label{eq:apriori-fdelta}
	\int_{Q_{\frac{3R}{4}}}&\int_{\R^d \setminus B_R} (f_\delta-l)_+(w) \chi_{f_\delta > l}(v) K(v, w) \dd w\dd v \\
	&\leq CR_0^{n(1-\zeta)-2s} \mathcal M_l \abs{\big\{ f_\delta > l\big\}  \cap Q_{R_0}}^\zeta,
\eal
where $C$ does not depend on $\delta$. 
Now we want to pass to the limit. Note that as $\delta \to 0$, the coefficients $K_\delta$ converge to $K$ pointwise a.e. Due to \cite[Lemma 4.1]{AL} and \cite[Theorem 1.1]{IS} (which both rely merely on the $L^2_{t, x} H^s_v$ norm of $f_\delta$) we know $f_\delta \in L^\infty_{loc}([t_1, t_2] \times B_{R^{1+2s}} \times \R^d)$ uniformly in $\delta$, and moreover  $f_\delta \in C^{\alpha}_{loc}([t_1, t_2] \times B_{R^{1+2s}} \times \R^d)$ uniformly in $\delta$ for some $\alpha \in (0, 1)$. As a consequence of the uniform estimates, $f_\delta$ converges to some $f: [t_1, t_2] \times B_{R^{1+2s}} \times \R^d \to \R_+$ pointwise almost everywhere in $[t_1, t_2] \times B_{R^{1+2s}} \times \R^d$, weakly in $L^2([t_1, t_2] \times B_{R^{1+2s}}; H^s(\R^d))$ by weak compactness, and strongly in $L^1([t_1, t_2] \times B_{R^{1+2s}} \times \R^d)$ by uniform integrability, pointwise convergence and tightness. In particular, $f \geq 0$ and it satisfies 
\bals
	\int_{Q_{\frac{3R}{4}}}\int_{\R^d \setminus B_R} (f-l)_+(w) \chi_{f > l}(v) K(v, w) \dd w\dd v \leq CR_0^{n(1-\zeta)-2s} \mathcal M_l \abs{\big\{ f > l\big\}  \cap Q_{R_0}}^\zeta.
\eals
It remains to check that $f$ satisfies the limit equation $$\mathcal T f = \int_{\R^d} \big(f(w) - f(v) \big)K(v, w) \dd w.$$

Let $\varphi \in L^2((t_1, t_2) \times B_{R^{1+2s}},  H^s(\R^d))$ such that $\varphi(\cdot, \cdot, v)$ has compact support in $(t_1, t_2) \times B_{R^{1+2s}}$. Then $f_\delta$ satisfies
\bals
	&\int_{(t_1, t_2)\times B_{R^{1+2s}} \times \R^d} (\mathcal T f_\delta)\varphi\dd z\\
	 &\quad+ \int_{(t_1, t_2)}\int_{B_{R^{1+2s}}}\int_{\R^d} \int_{\R^d} \big(f_\delta(v) - f_\delta(w)\big) K_\delta(v, w)\varphi(v) \dd w\dd v\dd x\dd t  = 0.
\eals
Integrating by parts and using the global convergence of the coefficients pointwise and strongly in $L^1$, and by weak compactness in $L^2_{t, x} H^s_v$, we see that $f$ satisfies 
\bals
	&-\int_{(t_1, t_2)\times B_{R^{1+2s}} \times \R^d}     f\mathcal T\varphi\dd z \\
	&\quad+ \int_{(t_1, t_2)}\int_{B_{R^{1+2s}}}\int_{\R^d} \int_{\R^d} f(v) K(v, w)(\varphi(v)-\varphi(w)\big) \dd w\dd v\dd x\dd t  = 0,
\eals
that is
\bals
	&\int_{(t_1, t_2)\times B_{R^{1+2s}} \times \R^d}   (\mathcal T f)\varphi\dd z \\
	&\quad+ \int_{(t_1, t_2)}\int_{B_{R^{1+2s}}}\int_{\R^d} \int_{\R^d}\big(f(v)-f(w)\big) K(v, w)\varphi(v) \dd w\dd v\dd x\dd t  = 0.
\eals

\section{Upper bound in the parabolic case}\label{app:improved-decay}
In this subsection, we discuss how the method outlined above to derive the polynomial decay of Theorem \ref{thm:upperbounds} can be improved for kernels that satisfy a pointwise upper bound \eqref{eq:upperbounds-pw} in case that we are in the parabolic setting: $J$ is constant in $x$. The most significant sub-optimality stems from the way we treated the non-singular part in Aronson's method in Proposition \ref{prop:aronson}. For pointwisely bouded kernels, instead of \eqref{eq:step1}, we obtain:
\bal\label{eq:step1-improved}
	\sup_{t \in (\tau_0, \sigma)} \int_{\R^{2d}} f^2(t, x, v) H(t, x, v)  \dd x \dd v\leq &\int_{\R^{2d}} f^2(\tau_0, x, v) H(\tau_0, x, v) \dd x \dd v \\
	&+C \rho^{-(d+2s)}\Norm{H}_{L^\infty([\tau_0, \sigma]\times \R^{2d})} \norm{f}_{L^2_t([\tau_0, \sigma]; L^1_v(\R^{d}))}^2.
\eal
This follows by replacing the estimate on the non-singular part of \eqref{eq:I-loc-ns} by 
\bals
	\mathcal I_{loc}^{ns} &= \int_{B_{2R}}\int_{B_{2R}} \big[f(w) - f(v)\big] \Big[\big(fH\varphi_R^2\big)(v) - \big(fH\varphi_R^2\big)(w)\Big] K(v, w)\chi_{\abs{v-w} > \rho} \dd w\dd v\\
	&\leq \int_{B_{2R}}\int_{B_{2R}} \Big[f(w)\big(fH\varphi_R^2\big)(v) + f(v)\big(fH\varphi_R^2\big)(w) \Big]K(v, w)\chi_{\abs{v-w} > \rho} \dd w\dd v\\
	&\leq C \rho^{-2s-d} \norm{H}_{L^\infty} \int_{B_{2R}} f(w) \dd w\int_{B_{2R}} f(v) \dd v.
\eals
The construction of $H$ remains the same. Therefore, this improvement leads in the proof of Theorem \ref{thm:aronson-bound} to the replacement of \eqref{eq:aux-pw-ub} by
\bals
	\sup_{t \in (\tau_0, \sigma)} \int_{B^0_{2\rho}(t, w_0)} f^2(t, v)  \dd v\leq \Bigg(\frac{\rho^{2s}}{k(\sigma -\tau_0)}\Bigg)^{1-\frac{\gamma}{3\rho}}\norm{f_0}_{L^2(\R^{d})}^2 + \rho^{-(d+2s)} (\sigma - \tau_0)\norm{f_0}_{L^1(\R^{d})}^2,
\eals
which subsequently implies an improved estimate in Theorem \ref{thm:aronson-bound}:
\bal\label{eq:aronson-imp}
	\abs{f(\sigma, w_0)}^2 \leq  C (\sigma - \tau_0)^{-\frac{d}{2s}} \Bigg(\frac{\rho^{2s}}{k(\sigma - \tau_0)}\Bigg)^{1- \frac{\gamma}{3\rho}}  \norm{f_0}^2_{L^2(\R^{2d})} + \rho^{-(d+2s)} (\sigma - \tau_0)^{1-\frac{d}{2s}} \norm{f_0}_{L^1(\R^{d})}^2.
\eal
Using \eqref{eq:aronson-imp} in the proof of Theorem \ref{thm:off-diag}, yields the following improved estimate in Theorem \ref{thm:off-diag}:
\beqs
	J(\sigma, v; \tau_0, w_0) \leq C \Bigg(  (\sigma - \tau_0)^{-\frac{d}{2s}}\Bigg(\frac{\rho^{2s}}{k(\sigma - \tau_0)}\Bigg)^{\frac{1}{4}- \frac{\gamma}{12\rho}} +\rho^{-\frac{d+2s}{4}} (\sigma - \tau_0)^{\frac{2s-3d}{8s}}\Bigg),
\eeqs
so that eventually, choosing $\rho$ as in \eqref{eq:rho}, we conclude for $(\sigma - \tau_0) \leq \frac{\rho^{2s}}{4}$ 
\bals
	J(\sigma, v; \tau_0,w_0)  \lesssim  &(\sigma - \tau_0) \abs{v- w_0}^{-(d +2s)}+ (\sigma - \tau_0)^{\frac{2s-3d}{8s}} \abs{v- w_0}^{-\frac{d+2s}{4}}. 
\eals
Note that in the parabolic case, the optimal decay in velocity is $d+2s$, which suggests that our method - which does not rely on Meyer's decomposition - is sub-optimal by a factor of $\frac{1}{4}$.

\bibliographystyle{acm}
\bibliography{bib}

\end{document}